\documentclass[aos,preprint]{imsart}
\usepackage{amsmath,amssymb,amsthm}
\usepackage[numbers]{natbib}
\usepackage{graphicx,epsfig,epstopdf,array,multirow}
\usepackage{rotating}
\usepackage{float}
\usepackage{adjustbox}
\usepackage{pdflscape}
\allowdisplaybreaks[2]

\usepackage{color}
\definecolor{DarkBlue}{rgb}{0,0.18,0.55}

\usepackage{xr}
\theoremstyle{plain}
\newtheorem{theorem}{Theorem}[section]
\newtheorem{lemma}{Lemma}[section]
\newtheorem{corollary}{Corollary}[section]

\newtheorem{assumption}{Assumption}
\theoremstyle{definition}

\newtheorem{remark}{Remark}[section]

\newcommand{\hb}{\widehat\beta}

\renewcommand{\tilde}{\widetilde}

\renewcommand{\hb}{\widehat\beta}
\renewcommand{\qed}{\hfill{\tiny \ensuremath{\blacksquare} }}

\newcommand{\Ep}{{\mathrm{E}}}

\renewcommand{\Pr}{{\mathrm{P}}}

\DeclareMathOperator*{\argmin}{arg\,min}

\begin{document}

\begin{frontmatter}

\title{On Cross-Validated Lasso in High Dimensions\thanksref{T1}}
\runtitle{On Cross-Validated Lasso}
\thankstext{T1}{Date: May, 2016. Revised \today. We thank Mehmet Caner, Matias Cattaneo, Yanqin Fan, Sara van de Geer, Jerry Hausman, James Heckman, Roger Koenker, Andzhey Koziuk, Miles Lopes, Jinchi Lv, Rosa Matzkin, Anna Mikusheva, Whitney Newey, Jesper Sorensen, Vladimir Spokoiny, Larry Wasserman, and seminar participants in many places for helpful comments. Chetverikov's work was partially funded by NSF Grant SES - 1628889. Liao's work was partially funded by NSF Grant SES - 1628889.}

\begin{aug}
\author{\fnms{Denis} \snm{Chetverikov}\thanksref{m1}\ead[label=e1]{chetverikov@econ.ucla.edu}}
\author{\fnms{Zhipeng} \snm{Liao}\thanksref{m2}\ead[label=e2]{zhipeng.liao@econ.ucla.edu}}
\author{\fnms{Victor} \snm{Chernozhukov}\thanksref{m3}\ead[label=e3]{vchern@mit.edu}}

\runauthor{Chetverikov Liao Chernozhukov}

\affiliation{UCLA\thanksmark{m1}, UCLA\thanksmark{m2}, MIT\thanksmark{m3}}

\address{Department of Economics, UCLA\\
Bunche Hall, 8369 \\
315 Portola Plaza \\
Los Angeles, CA 90095, USA.\\
\printead{e1}}

\address{Department of Economics, UCLA\\
Bunche Hall, 8379 \\
315 Portola Plaza \\
Los Angeles, CA 90095, USA.\\
\printead{e2}}

\address{Department of Economics and\\
Operations Research Center, MIT \\
50 Memorial Drive \\
Cambridge, MA 02142, USA.\\
\printead{e3}}

\end{aug}

\begin{abstract}
In this paper, we derive non-asymptotic error bounds for the Lasso estimator when the penalty parameter for the estimator is chosen using $K$-fold cross-validation. Our bounds imply that the cross-validated Lasso estimator has nearly optimal rates of convergence in the prediction, $L^2$, and $L^1$ norms. For example, we show that in the model with the Gaussian noise and under fairly general assumptions on the candidate set of values of the penalty parameter, the estimation error of the cross-validated Lasso estimator converges to zero in the prediction norm with the $\sqrt{s\log p / n}\times \sqrt{\log(p n)}$ rate, where $n$ is the sample size of available data, $p$ is the number of covariates, and $s$ is the number of non-zero coefficients in the model. Thus, the cross-validated Lasso estimator achieves the fastest possible rate of convergence in the prediction norm up to a small logarithmic factor $\sqrt{\log(p n)}$, and similar conclusions apply for the convergence rate both in $L^2$ and in $L^1$ norms. Importantly, our results cover the case when $p$ is (potentially much) larger than $n$ and also allow for the case of non-Gaussian noise. Our paper therefore serves as a justification for the widely spread practice of using cross-validation as a method to choose the penalty parameter for the Lasso estimator.
\end{abstract}

\begin{keyword}
\kwd{cross-validation}
\kwd{Lasso}
\kwd{high-dimensional models}
\kwd{sparsity}
\kwd{non-asymptotic bounds}
\end{keyword}

\end{frontmatter}

\section{Introduction}
Since its invention by Tibshirani in \cite{T96}, the Lasso estimator has
become increasingly important in many fields, and a large number of papers
have studied its properties. Many of these papers have been concerned with
the choice of the penalty parameter $\lambda $ required for the
implementation of the Lasso estimator. As a result, several methods to
choose $\lambda $ have been proposed and theoretically justified; see \cite%
{ZHT07}, \cite{BRT09}, \cite{BC13}, \cite{SZ13}, and \cite{PB18} among other papers. Nonetheless, in practice researchers often rely upon cross-validation to choose $\lambda $,
see \cite{CJ15}, and in fact, based on simulation evidence, using
cross-validation to choose $\lambda $ remains a leading recommendation in
the theoretical literature (see textbook-level discussions in \cite{BG11}, 
\cite{HTW15}, and \cite{G15}). However, to the best of our knowledge, there
exist very few results about properties of the Lasso
estimator when $\lambda $ is chosen using cross-validation; see a review below. The purpose of this paper is to fill this gap and to
derive non-asymptotic error bounds for the cross-validated Lasso estimator
in different norms.

We consider the regression model 
\begin{equation}  \label{eq: model}
Y=X^{\prime }\beta + \varepsilon,\quad {\mathrm{E}}[\varepsilon \mid X] = 0,
\end{equation}
where $Y$ is a dependent variable, $X=(X_1,\dots,X_p)^{\prime }$ a $p$%
-vector of covariates, $\varepsilon$ unobserved scalar noise, and $%
\beta=(\beta_1,\dots,\beta_p)^{\prime }$ a $p$-vector of coefficients.
Assuming that a random sample of size $n$, $(X_i,Y_i)_{i=1}^n$, from the
distribution of the pair $(X,Y)$ is available, we are interested in
estimating the vector of coefficients $\beta$. We consider triangular array
asymptotics, so that the distribution of the pair $(X,Y)$, and in particular
the dimension $p$ of the vector $X$, is allowed to depend on $n$ and can be
larger or even much larger than $n$. For simplicity of notation, however, we
keep this dependence implicit.

We impose a standard assumption that the vector of coefficients $\beta $ is
sparse in the sense that $s=s_{n}=\Vert \beta \Vert _{0}=\sum_{j=1}^{p}1\{
\beta _{j}\neq 0\}$ is relatively small. Under this assumption, the
effective way to estimate $\beta $ was proposed by Tibshirani in \cite{T96},
who introduced the Lasso estimator, 
\begin{equation}
\widehat{\beta }(\lambda )=\argmin_{b\in \mathbb{R}^{p}}\left( \frac{1}{n}%
\sum_{i=1}^{n}(Y_{i}-X_{i}^{\prime }b)^{2}+\lambda \Vert b\Vert _{1}\right) ,
\label{eq: lasso}
\end{equation}%
where for $b=(b_{1},\dots ,b_{p})^{\prime }\in \mathbb{R}^{p}$, $\Vert
b\Vert _{1}=\sum_{j=1}^{p}|b_{j}|$ denotes the $L^{1}$ norm of $b$, and $%
\lambda $ is some penalty parameter (the estimator suggested in Tibshirani's
paper takes a slightly different form but over time the version 
\eqref{eq:
lasso} has become more popular, probably for computational reasons). In
principle, the optimization problem in \eqref{eq: lasso} may have multiple
solutions, but to simplify presentation and to avoid unnecessary
technicalities, we assume throughout the paper, without further notice, that the distribution of $X$
is absolutely continuous with respect to Lebesgue measure on $\mathbb{R}%
^{p}$, in which case the optimization problem in \eqref{eq: lasso} has the
unique solution with probability one; see Lemma 4 in \cite{T13}. Without this assumption, our results would apply to the sparsest solution.

To perform the Lasso estimator $\widehat{\beta }(\lambda )$, one has to
choose the penalty parameter $\lambda $. If $\lambda $ is chosen
appropriately, the Lasso estimator attains the optimal rate of convergence
under fairly general conditions; see, for example, \cite{BRT09}, \cite{BC11}%
, and \cite{RT11}. On the other hand, if $\lambda $ is not chosen
appropriately, the Lasso estimator may not be consistent or may have a
slower rate of convergence; see \cite{C14}. Therefore, it is important to
choose $\lambda $ appropriately. In this paper, we show that $K$-fold
cross-validation indeed provides an appropriate way to choose $\lambda $.
More specifically, we derive non-asymptotic error bounds for the Lasso
estimator $\widehat{\beta }(\lambda )$ with $\lambda =\widehat{\lambda }$
being chosen by $K$-fold cross-validation in the prediction, $L^{2}$, and $%
L^{1}$ norms. Our bounds reveal that the cross-validated Lasso estimator
attains the optimal rate of convergence up to certain logarithmic factors in
all of these norms. For example, when the conditional distribution of the
noise $\varepsilon $ given $X$ is Gaussian, the $L^{2}$ norm bound in
Theorem \ref{thm: l1 and l2 bounds} implies that 
\begin{equation*}
\Vert \widehat{\beta }(\widehat{\lambda })-\beta \Vert _{2}=O_{P}\left( 
\sqrt{\frac{s\log p}{n}}\times \sqrt{\log (pn)}\right) ,
\end{equation*}%
where for $b=(b_{1},\dots ,b_{p})^{\prime }\in \mathbb{R}^{p}$, $\Vert
b\Vert _{2}=(\sum_{j=1}^{p}b_{j}^{2})^{1/2}$ denotes the $L^{2}$ norm of $b$%
. Here, $\sqrt{s\log p/n}$ represents the optimal rate of convergence, and
the cross-validated Lasso estimator attains this rate up to a small $\sqrt{%
\log (pn)}$ factor. Throughout the paper, we assume that $K$ is fixed, i.e., independent of $n$. Our results therefore do not cover leave-one-out cross-validation.

Given that cross-validation is often used to choose the penalty parameter $%
\lambda$ and given how popular the Lasso estimator is, understanding the
rate of convergence of the cross-validated Lasso estimator seems to be an
important research question. Yet, to the best of our knowledge, the only
results in the literature about the cross-validated Lasso estimator are due
to Homrighausen and McDonald \cite{HM13a, HM14, HM13b} and Miolane and
Montanari \cite{MM18} but all these papers imposed extremely strong
conditions and made substantial use of these conditions meaning that it is
not clear how to relax them. In particular, \cite{HM14} assumed that $p$ is
much smaller than $n$, and only showed consistency of the (leave-one-out)
cross-validated Lasso estimator. \cite{HM13b}, which strictly improves upon 
\cite{HM13a}, assumed that the smallest value of $\lambda$ in the candidate
set, over which cross-validation search is performed, is so large that all
considered Lasso estimators are guaranteed to be sparse, but, as we explain
below, it is exactly the low values of $\lambda$ that make the analysis of
the cross-validated Lasso estimator difficult. (In addition, and equally
important, the smallest value of $\lambda$ in \cite{HM13b} exceeds the
Bickel-Ritov-Tsybakov $\lambda = \lambda^*$, and we find via simulations
that the cross-validated $\lambda = \widehat \lambda$ is smaller than $%
\lambda^*$, at least with high probability, whenever the candidate set is
large enough, see Remarks \ref{rem: 1} and \ref{rem: proof} for further
details; this suggests that the cross-validated $\lambda$ based on the
Homrighausen-McDonald candidate set will be with high probability equal to
the smallest value in the candidate set, which makes the cross-validation
search less interesting.) \cite{MM18} assumed that $p$ is proportional to $n$
and that the vector $X$ consists of i.i.d. Gaussian random variables, and
their estimation error bounds do not converge to zero whenever $K$ is fixed
(independent of $n$). In contrast to these papers, we allow $p$ to be much
larger than $n$ and $X$ to be non-Gaussian, with possibly correlated
components, and we also allow for very large candidate sets.

Other papers that have been concerned with cross-validation in the context
of the Lasso estimator include Chatterjee and Jafarov \cite{CJ15} and Lecu%
\'{e} and Mitchell \cite{LM12}. \cite{CJ15} developed a novel
cross-validation-type procedure to choose $\lambda$ and showed that the
Lasso estimator based on their choice of $\lambda$ has a rate of convergence
depending on $n$ via $n^{-1/4}$. Their procedure to choose $\lambda$,
however, is related to but different from the classical cross-validation
procedure used in practice, which is the target of study in our paper. \cite%
{LM12} studied classical cross-validation but focused on estimators that
differ from the Lasso estimator in important ways. For example, one of the
estimators they considered is the average of subsample Lasso estimators, $%
K^{-1}\sum_{k=1}^K\widehat \beta_{-k}(\lambda)$, for $\widehat
\beta_{-k}(\lambda)$ defined in \eqref{eq: subsample lasso} in the next
section. Although the authors studied properties of the cross-validated
version of such estimators in great generality, it is not immediately clear
how to apply their results to obtain bounds for the cross-validated Lasso
estimator itself. We also emphasize that our paper is not related to Abadie
and Kasy \cite{AK18} because they do consider the cross-validated Lasso
estimator but in a very different setting, and, moreover, their results are
in the spirit of those in \cite{LM12}. (The results of \cite{AK18} can be
applied in the regression setting \eqref{eq: model} but the application
would require $p$ to be smaller than $n$ and their estimators in this case
would differ from the cross-validated Lasso estimator studied here.)

Finally, we emphasize that deriving a rate of convergence of the
cross-validated Lasso estimator is a non-standard problem. From the Lasso
literature perspective, a fundamental problem is that most existing results
require that $\lambda$ is chosen so that $\lambda > 2\|n^{-1}\sum_{i=1}^n
X_i \varepsilon_i\|_{\infty}$, at least with high probability, but,
according to simulation evidence, this inequality typically does not hold if 
$\lambda$ is chosen by cross-validation, meaning that existing results can
not be used to analyze the cross-validated Lasso estimator; see Section \ref%
{sec: main results} for more details and \cite{G15}, page 105, for
additional complications. Also, classical techniques to derive properties of
cross-validated estimators developed, for example, in \cite{L87} do not
apply to the Lasso estimator as those techniques are based on the linearity
of the estimators in the vector of values $(Y_1,\dots,Y_n)^{\prime }$ of the
dependent variable, which does not hold in the case of the Lasso estimator.
More recent techniques, developed, for example, in \cite{W03}, help to
analyze sub-sample Lasso estimators like those studied in \cite{LM12} but
are not sufficient for the analysis of the full-sample Lasso estimator
considered here. See \cite{AC10} for an extensive review of results on
cross-validation available in the literature.

%We also want to emphasize that we do not necessarily advocate using cross-validation to choose the penalty parameter $\lambda$. In fact, it is well-known from simulations that other techniques can substantially outperform cross-validation if we are interested in model selection, for example. The aim of our paper instead is to derive theoretical properties of the cross-validated Lasso estimator.

%Having good, theoretically justified, procedures for selecting the regularization parameter for the Lasso estimator is important. If the model is over-regularized, the Lasso estimator may overlook significant effects. On the other hand, if the model is under-regularized, the Lasso estimator may find effects that are actually not present in the data (spurious effects). In both cases, the researcher may end up with highly misleading conclusions and detrimental policy recommendations. 

The rest of the paper is organized as follows. In the next section, we
describe the cross-validation procedure. In Section \ref{sec: regularity
conditions}, we state our regularity conditions. In Section \ref{sec: main
results}, we present our main results. In Section \ref{sec: sparsity bound},
we describe novel sparsity bounds, which constitute one of the main
building blocks in our analysis of the cross-validated Lasso estimator. In
Section \ref{sec: simulations}, we conduct a small Monte Carlo simulation
study demonstrating that performance of the Lasso estimator based on the penalty parameter selected by cross-validation is comparable and often better than that of the Lasso estimator based on various plug-in rules. In Section \ref{sec: proofs estimation error}, we provide proofs of the main results on the estimation error bounds. In Section \ref{sec: proofs sparsity}, we provide proofs of our sparsity bounds. In Section \ref{sec: technical lemmas}, we collect some technical lemmas that are useful for the
proofs of the main results.

\medskip

\textbf{Notation.} Throughout the paper, we use the following notation. For
any vector $b=(b_{1},\dots ,b_{p})^{\prime }\in \mathbb{R}^{p}$, we use $%
\Vert b\Vert _{0}=\sum_{j=1}^{p}1\{b_{j}\neq 0\}$ to denote the number of
non-zero components of $b$, $\Vert b\Vert _{1}=\sum_{j=1}^{p}|b_{j}|$ to
denote its $L^{1}$ norm, $\Vert b\Vert _{2}=(\sum_{j=1}^{p}b_{j}^{2})^{1/2}$
to denote its $L^{2}$ norm, $\Vert b\Vert _{\infty }=\max_{1\leq j\leq
p}|b_{j}|$ to denote its $L^{\infty }$ norm, and $\Vert b\Vert
_{2,n}=(n^{-1}\sum_{i=1}^{n}(X_{i}^{\prime }b)^{2})^{1/2}$ to denote its
prediction norm. Also, for any random variable $Z$, we use $\|Z\|_{\psi_1}$ and $\|Z\|_{\psi_2}$ to denote its $\psi_1$- and $\psi_2$- Orlicz norms. In addition, we denote $X_{1}^{n}=(X_{1},\dots ,X_{n})$. Moreover, we use $\mathcal{S}^{p}$ to denote the unit sphere in $\mathbb{R}%
^{p}$, that is, $\mathcal{S}^{p}=\{ \delta \in \mathbb{R}^{p}\colon \Vert
\delta \Vert _{2}=1\}$, and for any $\ell>0$, we use $\mathcal S^p(\ell)$ to denote the $\ell$-sparse subset of $\mathcal S^p$, that is, $\mathcal S^p(\ell) = \{\delta\in \mathcal S^p\colon \|\delta\|_0\leq \ell\}$. We introduce more notation in the beginning of
Section \ref{sec: proofs estimation error}, as required for the proofs of
the main results.

%One of the most important machine learning techniques is Lasso. However, implementing Lasso, as well as many other machine learning techniques, requires selecting a regularization parameter. To choose this parameter in practice, it is often recommended to use cross-validation. However, little is known about properties of cross-validated techniques. In particular, to the best of our knowledge, there are no results in the literature about the rate of convergence of cross-validated Lasso. In this paper, we fill this gap in the literature.

\section{Cross-Validation\label{sec: cross validation}}

As explained in the Introduction, to choose the penalty parameter $\lambda $
for the Lasso estimator $\widehat{\beta }(\lambda )$, it is common practice
to use cross-validation. In this section, we describe the procedure in
details. Let $K$ be some strictly positive (typically small) integer, and
let $(I_{k})_{k=1}^{K}$ be a partition of the set $\{1,\dots ,n\}$; that is,
for each $k\in \{1,\dots ,K\}$, $I_{k}$ is a subset of $\{1,\dots ,n\}$, for
each $k,k^{\prime }\in \{1,\dots ,K\}$ with $k\neq k^{\prime }$, the sets $%
I_{k}$ and $I_{k^{\prime }}$ have empty intersection, and $\cup
_{k=1}^{K}I_{k}=\{1,\dots ,n\}$. For our asymptotic analysis, we will assume
that $K$ is a constant that does not depend on $n$. Further, let $\Lambda
_{n}$ be a set of candidate values of $\lambda $. Now, for $k=1,\dots ,K$
and $\lambda \in \Lambda _{n}$, let 
\begin{equation}
\widehat{\beta }_{-k}(\lambda )=\argmin_{b\in \mathbb{R}^{p}}\left( \frac{1}{%
n-n_{k}}\sum_{i\notin I_{k}}(Y_{i}-X_{i}^{\prime }b)^{2}+\lambda \Vert
b\Vert _{1}\right)  \label{eq: subsample lasso}
\end{equation}%
be the Lasso estimator corresponding to all observations excluding those in $%
I_{k}$ where $n_{k}=|I_{k}|$ is the size of the subsample $I_{k}$. As in the
case with the full-sample Lasso estimator $\widehat{\beta }(\lambda )$ in %
\eqref{eq: lasso}, the optimization problem in \eqref{eq: subsample lasso}
has the unique solution with probability one under our maintained assumption
that the distribution of $X$ is absolutely continuous with respect to the
Lebesgue measure on $\mathbb{R}^{p}$. Then the cross-validation choice of $%
\lambda $ is 
\begin{equation}
\widehat{\lambda }=\argmin_{\lambda \in \Lambda
_{n}}\sum_{k=1}^{K}\sum_{i\in I_{k}}(Y_{i}-X_{i}^{\prime }\widehat{\beta }%
_{-k}(\lambda ))^{2}.  \label{eq: cross validation}
\end{equation}%
The cross-validated Lasso estimator in turn is $\widehat{\beta }(\widehat{%
\lambda })$. In the literature, the procedure described here is also often
referred to as $K$-fold cross-validation. For brevity, however, we simply
refer to it as cross-validation. Below we will study properties of $\widehat{%
\beta }(\widehat{\lambda })$.

We emphasize one more time that although the properties of the estimators $%
\widehat \beta_{-k}(\widehat \lambda)$ have been studied in great generality
in \cite{LM12}, there are very few results in the literature regarding the
properties of $\widehat \beta(\widehat \lambda)$, which is the estimator
used in practice.

\section{Regularity Conditions\label{sec: regularity conditions}}

Recall that we consider the model in \eqref{eq: model}, the Lasso estimator $%
\widehat{\beta }(\lambda )$ in \eqref{eq: lasso}, and the cross-validation
choice of $\lambda $ in \eqref{eq: cross validation}. Let $c_{1}$, $C_{1}$, $%
a$ and $q$ be some strictly positive numbers where $a<1$ and $q > 4$.
Also, let $r\geq 0$ be an integer. In addition,
denote 
\begin{equation}
M_{n}=({\mathrm{E}}[\Vert X\Vert _{\infty }^{q}])^{1/q}.
\label{eq: Mn definition}
\end{equation}%
Throughout the paper, we assume that $s\geq 1$. Otherwise, one has
to replace $s$ by $s\vee 1$. To derive our results, we will impose the
following regularity conditions.

\begin{assumption}[Covariates]
\label{as: covariates} The random vector $X=(X_{1},\dots ,X_{p})^{\prime }$
is such that: (a) for all $\delta \in \mathcal{S}^{p}(n+s)$, we have $\Pr(|X'\delta|\geq c_1)\geq c_1$ and (b) for all $\delta \in \mathcal{S}^{p}(n^{2/q+c_1}M_n^2 s\log^3(pn))$, we have $({\mathrm{E}}[|X^{\prime }\delta|^{2}])^{1/2}\leq C_{1}$.
\end{assumption}

Part (a) of this assumption can be interpreted as a probability version of the ``no multicollinearity condition.'' It is slightly stronger than a more widely used expectation version of the same condition, namely $\Ep[(X'\delta)^2]\geq c_1$ for all $\delta\in\mathcal S^p(n+s)$ (with a possibly different value of the constant $c_1$), meaning that all $(n+s)$-sparse eigenvalues of the population Gram matrix $\Ep[XX']$ are bounded away from zero. Part (b) requires that sufficiently sparse eigenvalues of the matrix $\Ep[XX']$ are bounded from above uniformly over $n$. Note that neither part (a) nor part (b) of Assumption \ref{as: covariates} imposes bounds on the eigenvalues of the empirical
Gram matrix $n^{-1}\sum_{i=1}^{n}X_{i}X_{i}^{\prime }$ (of course, if $p>n$,
the smallest eigenvalue of this matrix is necessarily zero and the largest one can grow with $n$, potentially fast). %To develop some intuition about Assumption \ref{as: covariates}, we consider two simple examples.

%\begin{example}[Gaussian covariates]
%\label{ex: gaussian} Suppose that $X=(X_{1},\dots ,X_{p})^{\prime }$
%consists of independent standard Gaussian random variables. Then for all $\delta \in \mathcal{S%
%}^{p}$, the random variable $X^{\prime }\delta $ is standard Gaussian as
%well, and so Assumption \ref{as: covariates} holds for $C_1\geq 1$ and any sufficiently small $c_1$.\qed
%\end{example}

%\begin{example}[Bounded covariates]
%\label{ex: bounded independent} Suppose that $X=(X_{1},\dots ,X_{p})^{\prime
%}$ consists of independent zero-mean bounded continuous random variables. In
%particular, suppose for simplicity that $\max_{1\leq j\leq p}|X_{j}|\leq 1$
%almost surely. In addition, suppose that the variance of each $X_{j}$ is
%bounded away from zero by a constant $a$. Then Assumptions \ref{as: covariates}(a,b,d)
%hold with $\Delta _{n}=1$, $c_1 = \sqrt a$, and $C_{1}=1$ trivially. Also, for all $t>0$ and $%
%\delta \in \mathcal{S}^{p}$, we have $\Pr (|X^{\prime }\delta |>t)\leq 2\exp
%(-t^{2}/2)$ by Hoeffding's inequality. Therefore, Assumption \ref{as: covariates}(c) is
%satisfied with $\Gamma _{n}=2$ by standard calculations.\qed
%\end{example}

\begin{assumption}[Growth condition]
\label{as: growth condition} The following growth condition is satisfied: $n^{4/q}M_{n}^{4}s\log ^{4}(p n)\leq C_{1}n^{1-c_{1}}$.
\end{assumption}

Assumption \ref{as: growth condition} is a mild growth condition restricting
some moments of $\|X\|_{\infty}$, the number of non-zero coefficients in the model $s$ and the number of parameters in the model $p$. When all components of the vector $X$ are bounded by a constant almost surely, this assumption reduces to
\begin{equation*}
s\log ^{4}p\leq C_{1}n^{1-c_{1}}.
\end{equation*}%
Thus, Assumptions \ref{as: covariates} and \ref{as: growth condition} do allow for the high-dimensional case, with $p$ being much larger than $n$. However, we note that these assumptions are stronger than those used with more conservative choices of $\lambda$; see \cite{BRT09, BC11} for example.

% In the remark below, we discuss conditions of this assumption in the two examples given above.

%\begin{remark}[Growth conditions in Examples \ref{ex: gaussian} and \ref{ex: bounded independent}] In Example \ref{ex: gaussian}, this assumption reduces to the following conditions: (i) $s(\log n)^4(\log p)^2/n^{1 - \epsilon} = o(1)$ for some constant $\epsilon>0$ and (ii) $s^2\log p / n  = o(1)$ since in this case, $M_n \leq C_q(\log p)^{1/2}$ for all $q>4$ and some constant $C_q$ that depends only on $q$. In Example \ref{ex: bounded independent}, Assumption \ref{as: growth condition} reduces to the following conditions: (i) $s(\log n)^4(\log p)/n^{1 - \epsilon} = o(1)$ for some constant $\epsilon>0$ and (ii) $s^2\log p / n  = o(1)$ since in this case, $M_n \leq 1$ for all $q>4$.\qed
%\end{remark}

\begin{assumption}[Noise]
\label{as: heterogeneity} There exists a standard Gaussian random variable $%
e $ that is independent of $X$ and a function $Q\colon \mathbb{R}^{p}\times 
\mathbb{R}\rightarrow \mathbb{R}$ that is thrice continuously differentiable
with respect to the second argument such that $\varepsilon = Q(X,e)$ and for all $x\in\mathbb R^p$, (i) $c_{1} \leq Q_{2}(x,e)\leq C_{1}(1+|e|^{r})$, (ii) $\|Q_{22}(x,e)\|_{\psi_2} \leq C_1$, and (iii) $\|Q_{222}(x,e)\|_{\psi_1} \leq C_1$, where we use index $2$ to denote the derivatives with respect to the second argument, so that $Q_{222}(X,e) = \partial^3 Q(X,e)/\partial e^3$, for example.
%\begin{equation}\label{eq: skorohod representation}
%c_{1}<Q_{2}(X,e)\leq C_{1}(1+|e|^{r})\quad \text{and}\quad {\mathrm{E}}\left[
%\left \vert Q_{22}(X,e)\right \vert ^{q/2}\mid X\right] \leq C_{1}^{q/2}
%\end{equation}%
%almost surely, where $Q_{2}(X,e)=\partial Q(X,e)/\partial e$ and $%
%Q_{22}(X,e)=\partial ^{2}Q(X,e)/\partial e^{2}$. MODIFY
\end{assumption}
Letting $\Phi$ and $F_{\varepsilon|X}$ denote the cdf of the $N(0,1)$ distribution and the conditional cdf of $\varepsilon$ given $X$, respectively, it follows that whenever $F_{\varepsilon |X}$ is continuous almost surely, the random variable $e = \Phi^{-1}(F_{\varepsilon|X}(\varepsilon))$ has the $N(0,1)$ distribution and is independent of $X$. In this case, we can guarantee that $\varepsilon = Q(X,e)$ by setting $Q(X,e) = Q_{\varepsilon|X}(\Phi(e))$, where $Q_{\varepsilon|X} = F_{\varepsilon|X}^{-1}$ is the conditional quantile function of $\varepsilon$ given $X$. In addition, Assumption \ref{as: heterogeneity} imposes certain smoothness conditions. In particular, it requires that the transformation function $e\mapsto Q(X,e)$, which generates the noise variable $\varepsilon$ from the $N(0,1)$ variable $e$, is smooth in the sense that it satisfies certain derivative bounds.

Assumption \ref{as: heterogeneity} is rather non-standard. It appears in our analysis because, as explained in Remark \ref{rem: proof} below, we rely upon the degrees of freedom formula for the Lasso estimator to establish some sparsity bounds. In turn, this formula, being a consequence of the Stein identity characterizing the standard Gaussian distribution, has a simple form whenever $\varepsilon\sim N(0,\sigma^2)$; see \cite{ZHT07} and \cite{TT12}. We extend this formula to the non-Gaussian case under the condition that the noise variable $\varepsilon$ is a smooth transformation of $e\sim N(0,1)$ as required by Assumption \ref{as: heterogeneity}. Note that Assumption \ref{as: heterogeneity} requires the noise variable $\varepsilon$ to be neither sub-Gaussian nor sub-exponential. It does require, however, that the support of $\varepsilon$ is $\mathbb R$. Note also that whenever $\varepsilon$ is independent of $X$, we can choose the function $Q(X,e)$ to be independent of $X$, i.e. $Q(X,e) =Q(e)$. One simple example of a distribution that satisfies Assumption \ref{as: heterogeneity} is that of $\varepsilon = e + e^3$ with $e\sim N(0,1)$. A more complicated example is $\varepsilon = (e + \gamma_1 e^3)/(1 + \gamma_2 e^2)$, where $\gamma_1,\gamma_2>0$ are such that $9\gamma_1^2 + \gamma_2^2 <10\gamma_1\gamma_2$. Figure 1 presents plots of three probability density functions satisfying Assumption \ref{as: heterogeneity}. Interestingly, the third one is bi-modal, which emphasizes the fact that Assumption \ref{as: heterogeneity} allows for a wide variety of distributions. Finally, note that Assumption \ref{as: heterogeneity} holds with $r=0$ if the conditional distribution of $\varepsilon$ given $X$ is Gaussian.

\begin{figure}[tbp]
\includegraphics[width=10cm]{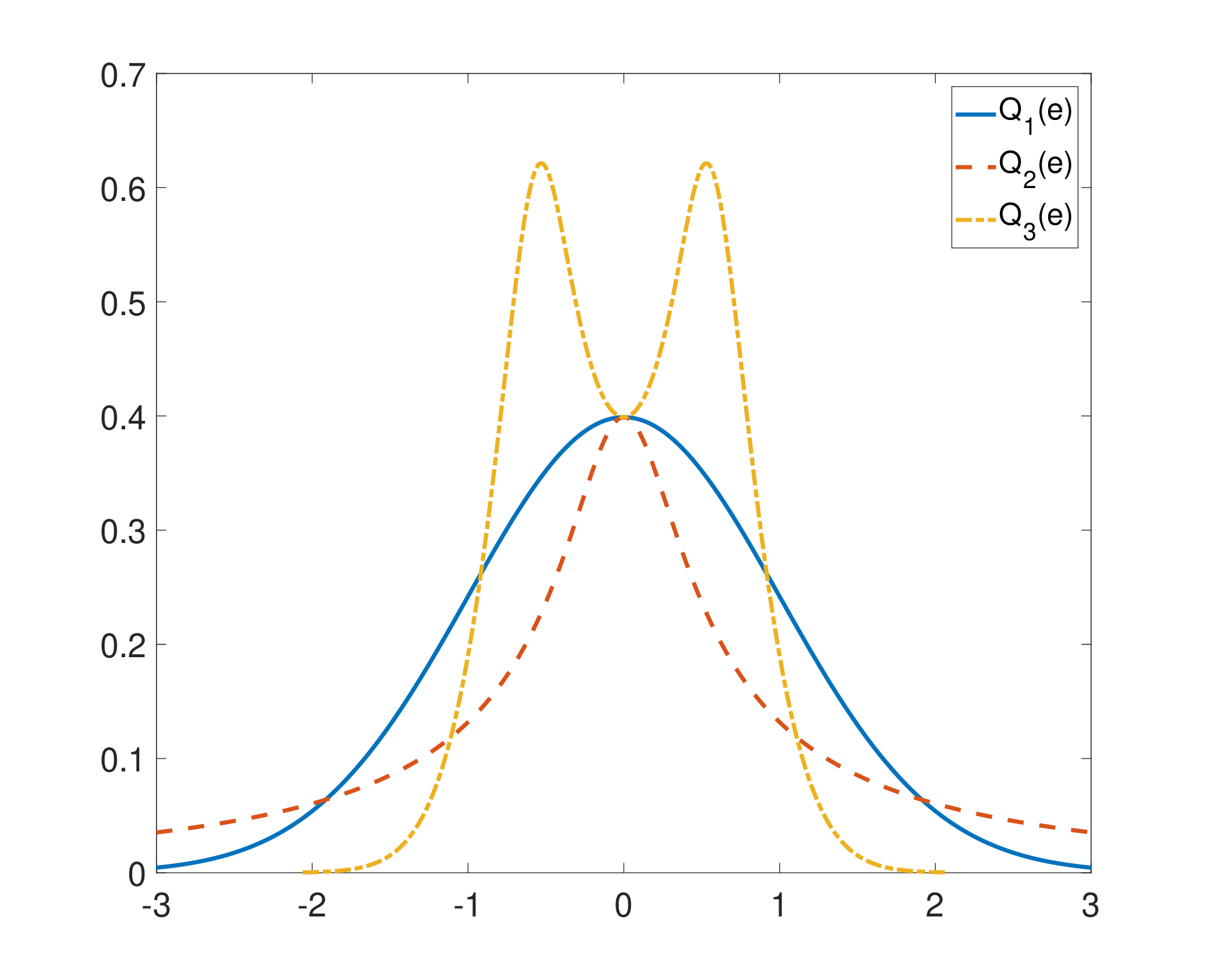}
\caption{The figure plots probability density functions of $\varepsilon = Q_j(e)$, $j=1,2,3$, where $e\sim N(0,1)$ and $Q_1(e) = e$, $Q_2(e) = e + e^3$, and $Q_3(e) = (e + e^3)/(1+2e^2)$. All three probability density functions are allowed by Assumption \ref{as: heterogeneity}.}
\label{fig:5}
\end{figure}

%The existence of a function $Q\colon \mathbb{R}^{p}\times 
%\mathbb{R}\rightarrow \mathbb{R}$ that is increasing with respect to its
%second argument and is such that $\varepsilon = Q(X,e)$ for some  standard Gaussian random variable $e$ that is independent of $X$ follows immediately from the Skorohod representation. Indeed, we can take $Q(X,e) = Q_X(\Phi(e))$, where $Q_X$ is the conditional quantile function of $\varepsilon$ given $X$ and $\Phi$ is the cdf of the $N(0,1)$ distribution. Assumption \ref{as: heterogeneity} thus only imposes the
%restriction that the conditional distribution of $\varepsilon$ given $X$ is smooth in the sense that the ``generating'' function $e\mapsto Q(X,e)$ satisfies the derivative bounds in \eqref{eq: skorohod representation}. Note that Assumption \ref{as: heterogeneity} holds with $r=0$ if the
%conditional distribution of $\varepsilon $ given $X$ is Gaussian.

\begin{assumption}[Candidate set]
\label{as: candidate set}The candidate set $\Lambda _{n}$ takes the
following form: $\Lambda _{n}=\{C_{1}a^{l}\colon l=0,1,2,\dots ;\ a^{l}\geq
c_{1}/n\}$.
\end{assumption}

It is known from \cite{BRT09} that the optimal rate of convergence of the
Lasso estimator is achieved when $\lambda$ is of order $(\log p / n)^{1/2}$.
Since under Assumption \ref{as: growth condition}, we have $\log p = o(n)$,
it follows that our choice of the candidate set $\Lambda_n$ in Assumption %
\ref{as: candidate set} makes sure that there are some $\lambda$'s in the
candidate set $\Lambda_n$ that would yield the Lasso estimator with the
optimal rate of convergence in the prediction norm. Note also that
Assumption \ref{as: candidate set} allows for a rather large candidate set $%
\Lambda_n$ of values of $\lambda$; in particular, the largest value, $C_1$,
can be set arbitrarily large and the smallest value, $c_1/n$, converges to
zero rather fast. In fact, the only two conditions that we need from
Assumption \ref{as: candidate set} is that $\Lambda_n$ contains a ``good''
value of $\lambda$, say $\bar \lambda_0$, such that the subsample Lasso
estimators $\widehat \beta_{-k}(\bar \lambda_0)$ satisfy the bound 
\eqref{eq: preliminary lasso bound} in Lemma \ref{lem: 1} with probability $1 - C
n^{-c} $ and that $|\Lambda_n|\leq C\log n$, where $c$ and $C$ are some
constants. Thus, we could for example set $\Lambda_n = \{a^l\colon l =\dots,
-2,-1,0,1,2,\dots; \ a^{-l} \leq n^{C_1}, \ a^l \leq n^{C_1}\}$.

\begin{assumption}[Dataset partition]
\label{as: dataset partition}The dataset partition $\{I_k\}_{k=1}^K$ is such that for all $k=1,\dots ,K$, we have $n_{k}/n\geq c_{1}$, where $n_k = |I_k|$.
\end{assumption}

Assumption \ref{as: dataset partition} is mild and is typically imposed in
the literature on $K$-fold cross-validation. This assumption ensures that
the subsamples $I_k$ are balanced in the sample size.

\section{Main Results\label{sec: main results}}

%Recall that for $b\in\mathbb R^p$, we use $\|b\|_{2,n} = (n^{-1}\sum_{i=1}^n (X_i' b)^2)^{1/2}$ to denote the prediction norm of $b$. 
Our first main result in this paper gives a non-asymptotic estimation error
bound for the cross-validated Lasso estimator $\widehat \beta(\widehat
\lambda) $ in the prediction norm.

\begin{theorem}[Prediction Norm Bound]
\label{thm: cross validation} Suppose that Assumptions \ref{as: covariates}
-- \ref{as: dataset partition} hold. Then for any $\alpha \in (0,1)$, 
\begin{equation*}
\Vert \widehat{\beta }(\widehat{\lambda })-\beta \Vert _{2,n}\leq \sqrt{%
\frac{Cs\log (p/\alpha )}{n}}\times \sqrt{\log (pn)+s^{-1}\log ^{r+1}n}
\end{equation*}%
with probability at least $1-\alpha -Cn^{-c}$,
where $c,C>0$ are constants depending only on $c_{1}$, $C_{1}$, $K$, $a$, $q$, and $r$.
\end{theorem}

\begin{remark}[Near-rate-optimality of cross-validated Lasso estimator in
prediction norm]
\label{rem: 1} The results in \cite{BRT09} imply that under the assumptions
of Theorem \ref{thm: cross validation}, setting $\lambda =\lambda ^{\ast
}=(C\log p/n)^{1/2}$ for sufficiently large constant $C$, which depends on
the distribution of $\varepsilon $, gives the Lasso estimator $\widehat{%
\beta }(\lambda ^{\ast })$ satisfying $\Vert \widehat{\beta }(\lambda ^{\ast
})-\beta \Vert _{2,n}=O_{P}((s\log p/n)^{1/2})$, and it follows from \cite%
{RT11} that this is the optimal rate of convergence (in the minimax sense)
for the estimators of $\beta $ in the model \eqref{eq: model}. Therefore,
Theorem \ref{thm: cross validation} implies that the cross-validated Lasso
estimator $\widehat{\beta }(\widehat{\lambda })$ has the fastest possible
rate of convergence in the prediction norm up to the small $(\log
(pn)+s^{-1}\log ^{r+1}n)^{1/2}$ factor. Note, however, that implementing the
cross-validated Lasso estimator does not require knowledge of the
distribution of $\varepsilon $, which makes this estimator attractive in
practice. In addition, simulation evidence suggests that $\widehat{\beta }(%
\widehat{\lambda })$ often outperforms $\widehat{\beta }(\lambda ^{\ast })$, which is one of the main reasons why cross-validation is typically recommended as a method to choose $\lambda 
$. The rate of convergence following from Theorem \ref{thm: cross validation}
is also very close to the oracle rate of convergence, $(s/n)^{1/2}$, that
could be achieved by the OLS estimator if we knew the set of covariates
having non-zero coefficients; see, for example, \cite{BCCK15}.\qed
\end{remark}

\begin{remark}[On the proof of Theorem \protect \ref{thm: cross validation}]
\label{rem: proof} One of the main steps in \cite{BRT09} is to show that outside
of the event 
\begin{equation}
\lambda <c\max_{1\leq j\leq p}\left \vert \frac{1}{n}\sum_{i=1}^{n}X_{ij}%
\varepsilon _{i}\right \vert ,  \label{eq: main event}
\end{equation}%
where $c>2$ is some constant, the Lasso estimator $\widehat{\beta }(\lambda
) $ satisfies the bound $\Vert \widehat{\beta }(\lambda )-\beta \Vert
_{2,n}\leq C\lambda \sqrt{s}$, where $C$ is a constant. Thus, to obtain the
Lasso estimator with a fast rate of convergence, it suffices to choose $%
\lambda $ such that it is small enough but the event \eqref{eq: main event}
holds with at most small probability. The choice $\lambda =\lambda ^{\ast }$
described in Remark \ref{rem: 1} satisfies these two conditions. The
difficulty with cross-validation, however, is that, as we demonstrate in
Section \ref{sec: simulations} via simulations, it typically yields a rather
small value of $\lambda $, so that the event \eqref{eq: main event} with $%
\lambda =\widehat{\lambda }$ holds with non-trivial (in fact, large)
probability even in large samples, and little is known about properties of
the Lasso estimator $\widehat{\beta }(\lambda )$ when the event 
\eqref{eq:
main event} does not hold, which is perhaps one of the main reasons why
there are only few results on the cross-validated Lasso estimator in the
literature. We therefore take a different approach. First, we use the fact
that $\widehat{\lambda }$ is the cross-validation choice of $\lambda $ to
derive bounds on $\Vert \widehat{\beta }_{-k}(\widehat{\lambda })-\beta
\Vert_2$ for the subsample Lasso estimators $\widehat{\beta }_{-k}(\widehat{%
\lambda })$ defined in \eqref{eq: subsample lasso}. Second, we use the
degrees of freedom formula of \cite%
{ZHT07} and \cite{TT12} to show that these estimators are sparse and to
derive bounds on $\Vert \widehat{\beta }_{-k}(\widehat{\lambda })-\beta
\Vert _{1}$ and $\Vert \widehat{\beta }_{-k}(\widehat{\lambda })-\beta \Vert
_{2,n}$. Third, we use the two point inequality stating that for all $b\in \mathbb{R}^{p}$ and $\lambda > 0$,
\begin{equation*}
\Vert \widehat{\beta }(\lambda )-b\Vert _{2,n}^{2}\leq \frac{1}{n}%
\sum_{i=1}^{n}(Y_{i}-X_{i}^{\prime }b)^{2}+\lambda \Vert b\Vert _{1}-\frac{1%
}{n}\sum_{i=1}^{n}(Y_{i}-X_{i}^{\prime }\widehat{\beta }(\lambda
))^{2}-\lambda \Vert \widehat{\beta }(\lambda )\Vert _{1},
\end{equation*}%
which can be found in \cite{G16}, with $\lambda = \widehat\lambda$ and $b=(K-1)^{-1}\sum_{k=1}^{K}(n-n_{k})%
\widehat{\beta }_{-k}(\widehat{\lambda })/n$, a convex combination of the
subsample Lasso estimators $\widehat{\beta }_{-k}(\widehat{\lambda })$, and
derive a bound for its right-hand side using the definition of estimators $%
\widehat{\beta }_{-k}(\widehat{\lambda })$ and bounds on $\Vert \widehat{%
\beta }_{-k}(\widehat{\lambda })-\beta \Vert_2 $ and $\Vert \widehat{\beta }%
_{-k}(\widehat{\lambda })-\beta \Vert _{1}$. Finally, we use the triangle
inequality to obtain a bound on $\Vert \widehat{\beta }(\widehat\lambda )-\beta
\Vert _{2,n}$ from the bounds on $\Vert \widehat{\beta }(\widehat\lambda )-b\Vert
_{2,n}$ and $\Vert \widehat{\beta }_{-k}(\widehat{\lambda })-\beta \Vert
_{2,n}$. The details of the proof can be found in Section \ref{sec: proofs estimation error}. 
%Third, we show that the Lasso estimator $\hb(\lambda)$ satisfies the following inequality for any $b\in\mathbb R^p$:
%$$
%\|\hb(\lambda) - b\|_{2,n}^2 \leq \frac{1}{n}\sum_{i=1}^n (Y_i - X_i'b)^2 + \lambda\|b\|_1 - \frac{1}{n}%\sum_{i=1}^n (Y_i - X_i'\hb(\lambda))^2 - \lambda\|\hb(\lambda)\|_1.
%$$
%Finally, we use this inequality with $b = (K-1)^{-1}\sum_{k=1}^K(n - n_k)\hb_{-k}(\widehat\lambda)/n$, a convex combination of the subsample Lasso estimators $\hb_{-k}(\widehat\lambda)$, and derive a bound for its right-hand side  using the definition of estimators $\hb_{-k}(\widehat\lambda)$ and bounds on $\|\hb_{-k}(\widehat\lambda) - \beta\|$ and $\|\hb_{-k}(\widehat\lambda) - \beta\|_1$.  The details of the proof, including a short proof of the two point inequality, can be found in Section \ref{sec: proofs}.
\qed
\end{remark}

%\begin{theorem}[Cross-validation; Gaussian high-dimensional case]\label{thm: cross validation gaussian high dimensional}
%Suppose that Assumptions \ref{as: covariates} -- \ref{as: dataset partition} hold. In addition, suppose that conditional distribution of $\varepsilon$ given $X$ is Gaussian. Then
%$$
%\|\widehat\beta(\widehat L) - \beta\|_{2,n}^2\lesssim \Big(\frac{s\log p}{n} + \frac{\log\log n}{n}\Big)^{3/4}\cdot \Big(T_n \Big(\frac{s\log p}{n} + \frac{\log\log n}{n}\Big)^{1/4} + \Big(T_n (\log^3 n)\cdot(\log^2 p)\Big)^{1/4}\Big)
%$$
%with probability $1-o(1)$ up-to a constant depending only on $c_1$ and $C_1$ where
%$$
%T_n = \Big(1 + \frac{p\log(p n)}{n}\Big)\wedge\Big((n^{2/q}M_n^2)\cdot (\log^4n)\cdot (\log p)\Big).
%$$
%\end{theorem}

%\begin{theorem}[Cross-validation; Non-Gaussian moderate-dimensional case]\label{thm: cross validation non gaussian moderate dimensional}
%Suppose that Assumptions \ref{as: covariates} -- \ref{as: dataset partition} hold. In addition, suppose that $p\log n/n\leq C_1$. Finally, suppose that for all $t\in\mathbb R$, we have $\log\Ep[\exp(t\varepsilon)\mid X] \leq C_1 t^2$. Then
%$$
%\|\widehat\beta(\widehat L) - \beta\|_{2,n}^2\lesssim \Big(\frac{s\log p}{n} + \frac{\log\log n}{n}\Big)^{1/2}
%$$
%with probability $1-o(1)$ up-to a constant depending only on $c_1$ and $C_1$ for $T_n$ defined in the statement of Theorem \ref{thm: cross validation gaussian high dimensional}
%\end{theorem}

Next, in order to obtain bounds on $\| \widehat \beta(\widehat \lambda) -
\beta \|_1$ and $\| \widehat \beta(\widehat \lambda) - \beta \|_2$, we
derive a sparsity bound for $\widehat \beta(\widehat \lambda)$, that is, we
show that the estimator $\widehat \beta(\widehat \lambda)$ has relatively
few non-zero components, at least with high-probability. Even though our
sparsity bound is not immediately useful in applications itself, it will
help us to translate the result in the prediction norm in Theorem \ref{thm:
cross validation} into the result in $L^1$ and $L^2$ norms in Theorem \ref%
{thm: l1 and l2 bounds}.

%In addition to the bound on the prediction norm of the estimation error of the cross-validated Lasso estimator given in Theorem \ref{thm: cross validation}, we derive in the next theorem a bound on the sparsity of the estimator. 

\begin{theorem}[Sparsity Bound]
\label{thm: sparsity bound} Suppose that Assumptions \ref{as: covariates} -- %
\ref{as: dataset partition} hold. Then for any $\alpha \in (0,1)$, 
\begin{equation*}
\| \widehat \beta(\widehat \lambda)\|_0 \leq Cs\times \frac{(\log^2p)(\log
n)(\log(p n) + s^{-1}\log^{r+1} n)}{\alpha}
\end{equation*}
with probability at least $1 - \alpha - Cn^{-c}$, where $c,C>0$ are constants depending only on $c_1$, $C_1$, $K$, $%
a $, $q$, and $r$.
\end{theorem}

\begin{remark}[On the sparsity bound]
\label{rem: sparsity} \cite{BC13} showed that outside of the event %
\eqref{eq: main event}, the Lasso estimator $\widehat \beta(\lambda)$
satisfies the bound $\| \widehat \beta(\lambda)\|_0 \leq C s$, for some
constant $C$, so that the number of covariates that have been mistakenly
selected by the Lasso estimator is at most of the same order as the number
of non-zero coefficients in the original model \eqref{eq: model}. As
explained in Remark \ref{rem: proof}, however, cross-validation typically
yields a rather small value of $\lambda$, so that the event 
\eqref{eq: main
event} with $\lambda = \widehat \lambda$ holds with non-trivial (in fact,
large) probability even in large samples, and it is typically the case that
smaller values of $\lambda$ lead to the Lasso estimators $\widehat
\beta(\lambda)$ with a larger number of non-zero coefficients. We therefore
should not necessarily expect that the inequality $\| \widehat
\beta(\widehat \lambda) \|_0 \leq C s$ holds with large probability. In
fact, it is well-known (from simulations) in the literature that the
cross-validated Lasso estimator typically satisfies $\| \widehat
\beta(\widehat \lambda) \|_0 \gg s $. Our theorem, however, shows that even
though the event \eqref{eq: main event} with $\lambda = \widehat \lambda$
may hold with large probability, the number of non-zero components in the
cross-validated Lasso estimator $\widehat \beta(\widehat \lambda)$ may
exceed $C s$ only by the relatively small $(\log^2p)(\log n)(\log(p n) +
s^{-1}\log^{r+1}) $ factor. \qed
\end{remark}

With the help of Theorems \ref{thm: cross validation} and \ref{thm: sparsity
bound}, we immediately obtain the following bounds on the $L^1$ and $L^2$
norms of the estimation error of the cross-validated Lasso estimator, which
is our second main result in this paper.

\begin{theorem}[$L^{1}$ and $L^{2}$ Norm Bounds]
\label{thm: l1 and l2 bounds} Suppose that Assumptions \ref{as: covariates}
-- \ref{as: dataset partition} hold. Then for any $\alpha \in (0,1)$, 
\begin{equation*}
\Vert \widehat{\beta }(\widehat{\lambda })-\beta \Vert _{2}\leq \sqrt{\frac{%
Cs\log (p/\alpha )}{n}}\times \sqrt{\log (pn)+s^{-1}\log ^{r+1}n}
\end{equation*}%
and 
\begin{equation*}
\Vert \widehat{\beta }(\widehat{\lambda })-\beta \Vert _{1}\leq \sqrt{\frac{%
Cs^{2}\log (p/\alpha )}{n}}\times \sqrt{\frac{(\log ^{2}p)(\log n)(\log
(pn)+s^{-1}\log ^{r+1}n)^{2}}{\alpha}}
\end{equation*}%
with probability at least $1-\alpha - Cn^{-c}$,
where $c,C>0$ are constants depending only on $c_{1}$, $C_{1}$, $K$, $a$%
, $q$, and $r$.
\end{theorem}

\begin{remark}[Near-rate-optimality of cross-validated Lasso estimator in $%
L^{1}\ $and $L^{2}$ norms]
Like in Remark \ref{rem: 1}, the results in \cite{BRT09} imply that under
the assumptions of Theorem \ref{thm: l1 and l2 bounds}, setting $\lambda
=\lambda ^{\ast }=(C\log p/n)^{1/2}$ for sufficiently large constant $C$
gives the Lasso estimator $\widehat{\beta }(\lambda ^{\ast })$ satisfying $%
\Vert \widehat{\beta }(\lambda ^{\ast })-\beta \Vert _{2}=O_{P}((s\log
p/n)^{1/2})$ and $\Vert \widehat{\beta }(\lambda ^{\ast })-\beta \Vert
_{1}=O_{P}((s^{2}\log p/n)^{1/2})$, and one can use the methods from \cite%
{RT11} to show that these rates are optimal. Therefore, the cross-validated
Lasso estimator $\widehat{\beta }(\widehat{\lambda })$ has the fastest
possible rate of convergence both in $L^{1}$ and in $L^{2}$ norms, up to
small logarithmic factors.\qed
\end{remark}

\begin{remark}[On the case with Gaussian noise]
Recall that whenever the conditional distribution of $\varepsilon$ given $X$
is Gaussian, we can take $r = 0$ in Assumption \ref{as: heterogeneity}.
Thus, it follows from Theorems \ref{thm: cross validation} and \ref{thm: l1
and l2 bounds} that, in this case, we have 
\begin{equation*}
\| \widehat \beta(\widehat \lambda) - \beta \|_{2,n}\vee \| \widehat \beta(\widehat \lambda) - \beta \|_{2} \leq \sqrt{\frac{C s\log(p/\alpha)}{n}}\times \sqrt{\log(p n)}
\end{equation*}
with probability at least $1 - \alpha - Cn^{-c}$ for any $\alpha \in(0,1)$ and some constants $c,C > 0$. Theorems \ref{thm: sparsity bound} and \ref{thm: l1 and l2 bounds} can also
be used to obtain the sparsity and $L^1$ norm bounds in this case as well.
However, the sparsity and $L^1$ norm bounds here can be improved using
results in \cite{BZ18}. In particular, assuming that the conditional
distribution of $\varepsilon$ given $X$ is $N(0,\sigma^2)$ for some constant 
$\sigma>0$, it follows from
Theorem 4.3 in \cite{BZ18} that for any $\lambda > 0$, 
\begin{equation*}
\text{Var}(\| \widehat \beta(\lambda) \|_0 \mid X_1^n) \leq {\mathrm{E}}[\|
\widehat \beta(\lambda) \|_0 \mid X_1^n] \left( 3 + 4\log \left( \frac{e p}{{%
\mathrm{E}}[\| \widehat \beta(\lambda) \|_0 \mid X_1^n]} \right) \right).
\end{equation*}
Combining this result and the same arguments as those in the proofs of Theorems \ref{thm: sparsity bound} and %
\ref{thm: l1 and l2 bounds}, with Chebyshev's inequality replacing Markov's inequality in the proof of Theorem \ref{thm: sparsity bound}, we have 
\begin{equation*}
\| \widehat \beta(\widehat \lambda) \|_0 \leq C s \times \frac{(\log^2
p)\log(p n)}{\sqrt \alpha}
\end{equation*}
and 
\begin{equation*}
\| \widehat \beta(\widehat \lambda) - \beta \|_1 \leq \sqrt{\frac{C s^2
\log(p/\alpha)}{n}}\times \frac{(\log p)\log(p n)}{\alpha^{1/4}}
\end{equation*}
with probability at least $1 - \alpha - Cn^{-c}$. \qed
\end{remark}

The near-rate-optimality of the cross-validated Lasso estimator in Theorem \ref{thm: cross validation} may be viewed as an in-sample prediction property since the prediction norm 
$$
\|\widehat\beta - \beta\|_{2,n} = \left(\frac{1}{n}\sum_{i=1}^{n}(X_{i}^{\prime }\widehat\beta - X_i'\beta)^2\right)^{1/2}
$$ evaluates estimation errors with respect to the observed data $X_1,\dots,X_n$. In addition, we can define an out-of-sample prediction norm 
$$
\|\widehat\beta - \beta\|_{p,2,n} = \Big(\Ep\Big[(X'\widehat\beta - X'\beta)^2\mid(X_i,Y_i)_{i=1}^n\Big]\Big)^{1/2},
$$
where $X$ is independent of $X_1,\dots,X_n$. Using Theorems \ref{thm: sparsity bound} and \ref{thm: l1 and l2 bounds}, we immediately obtain the following corollary on the estimation error of the cross-validated Lasso estimator in the out-of-sample prediction norm:
\begin{corollary}[Out-of-Sample Prediction Norm Bounds]\label{cor: out of sample prediction}
Suppose that Assumptions \ref{as: covariates} -- \ref{as: dataset partition} hold. Then for any $\alpha \in (0,1)$, 
\begin{equation*}
\Vert \widehat{\beta }(\widehat{\lambda })-\beta \Vert _{p,2,n}\leq \sqrt{\frac{%
Cs\log (p/\alpha )}{n}}\times \sqrt{\log (pn)+s^{-1}\log ^{r+1}n}
\end{equation*}
with probability at least $1-\alpha - Cn^{-c}$,
where $c,C>0$ are constants depending only on $c_{1}$, $C_{1}$, $K$, $a$%
, $q$, and $r$.
\end{corollary}

\section{General Sparsity Bounds\label{sec: sparsity bound}}

As we mentioned in Remark \ref{rem: proof}, our analysis of the
(full-sample) cross-validated Lasso estimator $\widehat{\beta }(\widehat{%
\lambda })$ requires understanding sparsity of the sub-sample
cross-validated Lasso estimators $\widehat{\beta }_{-k}(\widehat{\lambda })$%
, that is, we need a sparsity bound showing that $\Vert \widehat{\beta }%
_{-k}(\widehat{\lambda })\Vert _{0}$, $k=1,\dots ,K$, are sufficiently
small, at least with high probability. Unfortunately, existing sparsity
bounds are not good enough for our purposes because, as we discussed in
Remark \ref{rem: sparsity}, they only apply outside of the event 
\eqref{eq:
main event} and this event holds with non-trivial (in fact, large)
probability if we set $\lambda =\widehat{\lambda }$. We therefore develop here two novel sparsity bounds. The crucial feature of our bounds is that they apply for all values of $\lambda $, both large and small, independently of whether \eqref{eq: main event}
holds or not. Roughly speaking, the first bound shows, under mild conditions, that the Lasso estimator $\widehat{\beta }(\lambda )$ has to be sparse, at least with large probability, whenever it has small estimation error in the $L^2$ norm. The second bound shows, under somewhat stronger conditions, that the Lasso estimator $\widehat{\beta }(\lambda )$ has
to be sparse, at least on average, whenever it has small estimation error in the prediction norm. Both bounds turn out useful in our analysis.

\begin{theorem}[Sparsity Bound via Estimation Error in $L^2$ Norm]\label{thm: sparsity bound via l2 norm}
Suppose that Assumption \ref{as: heterogeneity} holds and let $\bar C>0$ be some constant. Then for all $\lambda > 0$ and $t\geq 1$,
\begin{multline*}
\Pr\left(\|\hb(\lambda)\|_0 \leq Cts\log^2(p n)\left(  \log^r n + \frac{ n(\log^2 n)\|\hb(\lambda) - \beta\|_2^2}{s\log(p n)}  \right)\mid X_1^n\right)\\ \leq 1 - \frac{2}{ts\log(pn)} - \frac{2}{n},
\end{multline*}
on the event $\sup_{\delta\in \mathcal S^p(s)}\|\delta\|_{2,n} \leq \bar C$, where $C>0$ is a constant depending only on $c_1$, $C_1$, $\bar C$, and $r$.
\end{theorem}

\begin{theorem}[Sparsity Bound via Estimation Error in Prediction Norm]
\label{lem: sparsity bound low dimensional} 
Suppose that Assumption \ref{as: heterogeneity} holds and let $\bar c, \bar C>0$ be some constants. Then for all $\lambda > 0$,
\begin{equation*}
{\mathrm{E}}[\Vert \widehat{\beta }(\lambda )\Vert _{0}\mid
X_{1}^{n}]\leq s+C(\log p)(nR_{n}(\lambda )^{2}+\log ^{r}n)
\end{equation*}
on the event 
\begin{equation}\label{eq: sparse eigenvalue sparsity theorem 2}
\bar c\leq \inf_{\delta\in \mathcal S^p(J_n(\lambda))}\|\delta\|_{2,n}\quad\text{and}\quad \max_{1\leq j\leq p}\sqrt{\frac{1}{n}\sum_{i=1}^n X_{i j}^2} \leq \bar C,
\end{equation}
where 
$$
J_n(\lambda) = n^{1/2 + c_1/8}(\sqrt n R_n(\lambda) + 1), \ R_{n}(\lambda )={\mathrm{E}}[\Vert \widehat{\beta }(\lambda )-\beta \Vert _{2,n}\mid X_{1}^{n}],
$$
and $C>0$ is a constant depending only on $c_1$, $C_1$, $\bar c$,  $\bar C$, and $r$.
\end{theorem}

%\begin{theorem}[General Sparsity Bound]
%\label{lem: sparsity bound low dimensional} Let $c_{\epsilon },C_{\epsilon
%}>0$ be some constants and suppose that Assumptions \ref{as: covariates}(a,b,d), %
%\ref{as: growth condition}(a), and \ref{as: heterogeneity} hold. Then there
%exist constants $c,C>0$ depending only on $c_{1}$, $C_{1}$, $c_{\epsilon }$, 
%$C_{\epsilon }$, $q$, and $r$ such that for all $\lambda >0$, the
%probability that both 
%\begin{equation*}
%{\mathrm{E}}[\Vert \widehat{\beta }(\lambda )\Vert _{0}\mid
%X_{1}^{n}]>s+C(\log p)(nR_{n}(\lambda )^{2}+\log ^{r}n)
%\end{equation*}%
%and 
%\begin{equation*}
%n^{2/q}\Delta _{n}^{2}M_{n}^{2}R_{n}(\lambda )\log p\leq C_{\epsilon
%}n^{-c_{\epsilon }}
%\end{equation*}%
%hold is at most $Cn^{-c}$, where $R_{n}(\lambda )={\mathrm{E}}[\Vert 
%\widehat{\beta }(\lambda )-\beta \Vert _{2,n}\mid X_{1}^{n}]$.
%\end{theorem}

%\appendix

%\appendix

\section{Simulations\label{sec: simulations}}

In this section, we present results of our simulation experiments. The
purpose of the experiments is to investigate finite-sample properties of the
cross-validated Lasso estimator. In particular, we are interested in (i)
comparing the estimation error of the cross-validated Lasso estimator in
different norms to the Lasso estimator based on other choices of $\lambda $;
(ii) studying sparsity properties of the cross-validated Lasso estimator;
and (iii) estimating probability of the event \eqref{eq: main event} for $%
\lambda =\widehat{\lambda }$, the cross-validation choice of $\lambda $.

We consider two data generating processes (DGPs). In both DGPs, we simulate
the vector of covariates $X=(X_{1},\dots ,X_{p})^{\prime }$ from the
Gaussian distribution with mean zero and variance-covariance matrix given by 
$E[X_{j}X_{k}]=\rho ^{|j-k|}$ for all $j,k=1,\dots ,p$ with $\rho =0.5$ and $%
0.75$. Also, we set $\beta =(1,-1,2,-2,0_{1\times (p-4)})^{\prime }$. We
simulate $\varepsilon $ from the standard Gaussian distribution in DGP1 and
from the uniform distribution on $[-3,3]$ in DGP2. For both DGPs, we take $%
\varepsilon $ to be independent of $X$. Further, for each DGP, we consider
samples of size $n=100$ and $400$. For each DGP and each sample size, we
consider $p=40$, $100$, and $400$. To construct the candidate set $\Lambda
_{n}$ of values of the penalty parameter $\lambda $, we use Assumption \ref%
{as: candidate set} with $a=0.9$, $c_{1}=0.005$ and $C_{1}=500$. Thus, the
set $\Lambda _{n}$ contains values of $\lambda $ ranging from $0.0309$ to $%
500$ when $n=100$ and from $0.0071$ to $500$ when $n=400$, that is, the set $%
\Lambda _{n}$ is rather large in both cases. In all experiments, we use
5-fold cross-validation ($K=5$). We repeat each experiment $5000$ times.

As a comparison to the cross-validated Lasso estimator, we consider the
Lasso estimators with $\lambda $ chosen according to \cite{SZ13} and \cite%
{PB18}, i.e.,
\begin{equation*}
\lambda =n^{-1/2}\sigma \sqrt{2\log p}\text{ and }\lambda =n^{-1/2}\sigma 
\sqrt{2\log (p/s)}
\end{equation*}%
respectively. These Lasso estimators achieve the optimal convergence rate
under the prediction norm (see, e.g., \cite{SZ13} and \cite{PB18}).\ The
noise level $\sigma $ and the true sparsity $s$ typically have to be
estimated from the data but for simplicity we assume that both $\sigma $ and 
$s$ are known, so we set $\sigma =1$ and $s=4$ in DGP1, and $\sigma =\sqrt{3}
$ and $s=4$ in DGP2. In what follows, these Lasso estimators are denoted as
SZ-Lasso and B-Lasso estimators respectively, and the cross-validated Lasso
estimator is denoted as CV-Lasso.

Figure \ref{fig:1} contains simulation results for DGP1 with $n=100$, $p=40$
and $\rho =0.75$. The first four (that is, the top-left, top-right,
middle-left and middle-right) panels of Figure \ref{fig:1} present the mean
of the estimation error of the Lasso estimators in the prediction, $L^{2}$, $%
L^{1}$, and out-of-sample prediction norms, respectively. The out-of-sample
prediction norm is defined as $\left \Vert b \right \Vert _{p,2,n}=(\Ep[(X'b)^2])^{1/2}$ for all $b\in\mathbb R^p$. In these panels, the dashed line represents the
mean of estimation error of the Lasso estimator as a function of $\lambda $
(we perform the Lasso estimator for each value of $\lambda $ in the
candidate set $\Lambda _{n}$; we sort the values in $\Lambda _{n}$ from the
smallest to the largest, and put the order of $\lambda $ on the horizontal
axis; we only show the results for values of $\lambda $ up to order 25 as
these give the most meaningful comparisons). This estimator is denoted as $%
\lambda $-Lasso. The solid, dotted and dashed-dotted horizontal lines
represent the mean of the estimation error of CV-Lasso, SZ-Lasso, and B-Lasso, respectively. 

From the top four panels of Figure \ref{fig:1}, we see that estimation error
of CV-Lasso is only slightly above the minimum of the estimation error over
all possible values of $\lambda $ not only in the prediction and $L^{2}$
norms but also in the $L^{1}$ norm. In comparison, SZ-Lasso and B-Lasso tend
to have larger estimation error in all four norms.

The bottom-left and bottom-right panels of Figure \ref{fig:1} depict the
histograms for the numbers of non-zero coefficients of the CV-Lasso
estimator and B-Lasso estimator respectively. Overall, these panels suggest
that the CV-Lasso estimator tends to select too many covariates: the number
of selected covariates with large probability varies between $5$ and $30$
even though there are only 4 non-zero coefficients in the true model. The
B-Lasso estimator is more sparse than the CV-Lasso estimator: it selects
around $5$ to $15$ covariates with large probability.

Figure \ref{fig:2} includes the simulation results for DGP1 when $n=100$, $%
p=400$ and $\rho =0.75$. The estimation errors of the Lasso estimators are
inflated when $p$ is much bigger than the sample size. The estimation error
of CV-Lasso under the prediction norm is increased from  $0.4481$ to $0.7616$
 when $p$ is increased from $40$ to $400$, although it remains the best
compared with SZ-Lasso and B-Lasso estimators. Similar phenomena are
observed for the estimation error under the $L^{2}$ norm and the
out-of-sample prediction norm. On the other hand, the estimation error of
the CV-Lasso is slightly larger than the SZ-Lasso and B-Lasso under the $%
L^{1}$ norm. For the sparsity of the Lasso estimators, the CV-Lasso is much
less sparse than the B-Lasso: it selects around $5$ to $50$ covariates with
large probability while the B-Lasso only selects $8$ to $22$ covariates with
large probability.

For all other experiments, the simulation results on the mean of estimation
error of the Lasso estimators can be found in Table 1. For simplicity, we
only report the minimum over $\lambda \in \Lambda _{n}$ of mean of the
estimation error of $\lambda $-Lasso and the mean of the estimation error of
B-Lasso in Table 1. The results in Table 1 confirm findings in Figure \ref%
{fig:1} and Figure \ref{fig:2}: the mean of the estimation error of CV-Lasso
is close to the minimum mean of the estimation errors of the $\lambda $%
-Lasso estimators under both DGPs for all combinations of $n$, $p$ and $\rho 
$ considered in all three norms. Their difference becomes smaller when the
sample size $n$ increases. The mean of the estimation error of B-Lasso is
larger than that of CV-Lasso in cases when $p$ is relatively small or the
regressors $X$ have strong correlation, while the B-Lasso has smaller
estimation error when $p$ is much larger than $n$ and the regressors $X$ are
weakly correlated. When the correlations of the regressors $X$ become
stronger and the largest eigenvalue of ${\mathrm{E}}[XX^{\prime }]$ becomes
bigger, the mean of the estimation error of the CV-Lasso estimator is
slightly enlarged and is much less effected compared with the B-Lasso
estimator. For example, in DGP1 with $n=100$ and $p=40$, the mean of
estimation error of CV-Lasso estimator increases\ $5.39\%$ when\ $\rho $ is
changed from $0.5$ to $0.75$ (and the largest eigenvalue of ${\mathrm{E}}%
[XX^{\prime }]$ increases from $2.97$ to $6.64$), while the B-Lasso
estimator has a $28\%$ increase.

Table 2 reports model selection results for the cross-validated Lasso
estimator. More precisely, the table shows probabilities for the number of
non-zero coefficients of the cross-validated Lasso estimator hitting
different brackets. Overall, the results in Table 2 confirm findings in
Figure \ref{fig:1} and Figure \ref{fig:2}: the cross-validated Lasso
estimator tends to select too many covariates. The probability of selecting
larger models tends to increase with $p$ but decreases with $n$.

Table 3 provides information on the finite-sample distribution of the ratio
of the maximum score $\max_{1\leq j\leq
p}|n^{-1}\sum_{i=1}^{n}X_{ij}\varepsilon _{i}|$ over $\widehat{\lambda }$,
the cross-validation choice of $\lambda $. More precisely, the table shows
probabilities for this ratio hitting different brackets. From Table 3, we
see that this ratio is above 0.5 with large probability in all cases and in
particular this probability exceeds 99\% in most cases. Hence, 
\eqref{eq:
main event} with $\lambda =\widehat{\lambda }$ holds with large probability,
meaning that deriving the rate of convergence of the cross-validated Lasso
estimator requires new arguments since existing arguments only work for the
case when \eqref{eq:
main event} does not hold; see discussion in Remark \ref{rem: proof} above.

\section{Proofs for Section \protect \ref{sec: main results}\label{sec:
proofs estimation error}}

In this section, we prove Theorems \ref{thm: cross validation}, \ref{thm:
sparsity bound}, and \ref{thm: l1 and l2 bounds} and Corollary \ref{cor: out of sample prediction}. Since the proofs are long, we start with a sequence of preliminary lemmas in Subsection \ref{sec:
preliminary lemmas estimation error} and give the actual proofs of the
theorems and the corollary in Subsections \ref{sec: proof of theorem 4.1}, \ref{sec: proof of
theorem 4.2}, \ref{sec: proof of theorem 4.3}, and \ref{sec: out of sample prediction}, respectively.

For convenience, we use the following additional notation. For $k=1,\dots ,K$%
, we denote 
\begin{equation*}
\Vert \delta \Vert _{2,n,k}=\bigg( \frac{1}{n_{k}}\sum_{i\in
I_{k}}(X_{i}^{\prime }\delta )^{2}\bigg) ^{1/2}\  \text{ and }\  \Vert \delta
\Vert _{2,n,-k}=\bigg( \frac{1}{n-n_{k}}\sum_{i\notin I_{k}}(X_{i}^{\prime
}\delta )^{2}\bigg) ^{1/2}
\end{equation*}%
for all $\delta \in \mathbb{R}^{p}$. We use $c$ and $C$ to denote strictly positive constants
that can change from place to place but that can be chosen to depend only on 
$c_{1}$, $C_{1}$, $K$, $a$, $q$, and $r$. We use the notation $a_{n}\lesssim
b_{n}$ if $a_{n}\leq Cb_{n}$. Moreover, for $\delta \in \mathbb{R}^{p}$ and $%
M\subset \{1,\dots ,p\}$, we use $\delta _{M}$ to denote the vector in $%
\mathbb{R}^{|M|}$ consisting of all elements of $\delta $ corresponding to
indices in $M$.

\subsection{Preliminary Lemmas\label{sec: preliminary lemmas estimation
error}}

Here, we collect preliminary lemmas that help to prove Theorems \ref{thm:
cross validation}--\ref{thm: l1 and l2 bounds}.

\begin{lemma}
\label{lem: eigenvalue bound} Suppose that Assumptions \ref{as: covariates} and \ref{as: growth condition} are satisfied and denote $\ell_n = \sqrt{s n^{1+c_1/2}}\log(p n)$. Then
\begin{equation}\label{eq: sparse eigenvalues bound general}
\sup_{\theta \in \mathcal{S}^{p}(\ell _{n})}\left \vert \frac{1}{n}%
\sum_{i=1}^{n}(X_{i}^{\prime }\theta )^{2}-{\mathrm{E}}[(X^{\prime
}\theta )^{2}]\right \vert \leq Cn^{-c}
\end{equation}
with probability at least $1-Cn^{-c}$, where $c$, $C>0$ are some constants
depending only on $c_{1}$, $C_{1}$, and $q$.
\end{lemma}

\begin{proof}
In this proof, $c$ and $C$ are strictly positive constants that depend only on $c_{1}$, $C_{1}$, and $q$ but their values can change from place to place. By Jensen's
inequality and the definition of $M_{n}$ in \eqref{eq: Mn definition}, 
\begin{align}
K_{n} &:=\left( {\mathrm{E}}\left[ \max_{1\leq i\leq n}\max_{1\leq j\leq
p}|X_{ij}|^{2}\right] \right) ^{1/2}  \notag \\
&\leq \left( {\mathrm{E}}\left[ \max_{1\leq i\leq n}\max_{1\leq j\leq
p}|X_{ij}|^{q}\right] \right) ^{1/q}\leq n^{1/q}M_{n}.  \label{l00}
\end{align}
Therefore, given that $\ell_n \leq C n^{1-c}$ by Assumptions \ref{as: covariates}(a) and \ref{as: growth condition}, which implies $\log\ell_n \leq C\log n$, it follows that
$$
\delta _{n}:=K_{n}\sqrt{\ell _{n}\log p/n}\left( 1+(\log \ell _{n})(\log
n)^{1/2}\right) \leq C n^{-c}
$$
by Assumption \ref{as: growth condition}; here, Assumption \ref{as: covariates}(a) is used only to verify that $M_n \geq c$. Also, denoting $\ell_{n,0} = n^{2/q + c_1}M_n^2 s \log^3(p n)$,
$$
\sup_{\theta \in \mathcal S^p(\ell_n)}\Ep[(X'\theta)^2] \leq \frac{2(\ell_n + \ell_{n,0})}{\ell_{n,0}}\sup_{\theta\in\mathcal S^p(\ell_{n,0})}\Ep[(X'\theta)^2] \leq \frac{C(\ell_n + \ell_{n,0})}{ \ell_{n,0}}
$$
by Lemma 9 in \cite{BC11} and Assumption \ref{as: covariates}(b). Thus,
$$
\delta_n \sup_{\theta \in \mathcal S^p(\ell_n)}\left(\Ep[(X'\theta)^2]\right)^{1/2} \leq C n^{-c}
$$
by Assumption \ref{as: growth condition}. Therefore, it follows from Lemma \ref{lem: rv} that 
\begin{equation*}
{\mathrm{E}}\left[ \sup_{\theta \in \mathcal{S}^{p}(\ell _{n})}\left
\vert \frac{1}{n}\sum_{i=1}^{n}(X_{i}^{\prime }\theta )^{2}-{\mathrm{E}}%
[(X^{\prime }\theta )^{2}]\right \vert \right] \leq Cn^{-c}.
\end{equation*}%
The asserted claim follows from combining this bound and Markov's inequality.
\end{proof}

\begin{lemma}
\label{lem: 1} Under Assumptions \ref{as: covariates}--\ref{as: dataset partition}, there exists $\bar{\lambda}_{0}=\bar{\lambda}_{n,0}\in \Lambda
_{n}$, possibly depending on $n$, such that for all $k=1,\dots ,K$, we have $\|\hb_{-k}(\bar\lambda_0)\|_0\lesssim s$ and, in addition,
\begin{equation}
\Vert \widehat{\beta }_{-k}(\bar{\lambda}_{0})-\beta \Vert
_2^{2}\lesssim \frac{s\log (pn)}{n}\  \text{ and }\  \Vert \widehat{%
\beta }_{-k}(\bar{\lambda}_{0})-\beta \Vert _{1}^{2}\lesssim \frac{s^{2}\log
(pn)}{n}  \label{eq: preliminary lasso bound}
\end{equation}%
with probability at least $1-Cn^{-c}$.
\end{lemma}

\begin{remark}
The result in this lemma is essentially well-known but we provide a short
proof here for completeness.
\end{remark}

\begin{proof}
Let $T=\{j\in \{1,\dots ,p\} \colon \beta _{j}\neq 0\}$
and $T^{c}=\{1,\dots ,p\} \backslash T$. Fix $k=1,\dots ,K$ and denote 
\begin{equation*}
Z_{k}=\frac{1}{n-n_{k}}\sum_{i\notin I_{k}}X_{i}\varepsilon _{i}
\end{equation*}%
and 
\begin{equation*}
\kappa _{k}=\inf \left \{ \frac{\sqrt{s}\Vert \delta \Vert _{2,n,-k}}{\Vert
\delta _{T}\Vert _{1}}\colon \  \delta \in \mathbb{R}^{p},\  \Vert \delta
_{T^{c}}\Vert _{1}<3\Vert \delta _{T}\Vert _{1}\right \} .
\end{equation*}%
To prove the asserted claims, we will apply Theorem 1 in \cite{BC11} that shows
that for any $\lambda \in \Lambda _{n}$, on the event $\lambda \geq 4\Vert
Z_{k}\Vert _{\infty }$, we have 
\begin{equation}\label{eq: bc bound main}
\Vert \widehat{\beta }_{-k}(\lambda )-\beta \Vert _{2,n,-k}\leq \frac{%
3\lambda \sqrt{s}}{2\kappa _{k}}.
\end{equation}%
To use this bound, we show that there exist $c>0$, $C>0$, and $\bar{\lambda}%
_{0}=\bar{\lambda}_{n,0}\in \Lambda _{n}$, possibly depending on $n$, such
that 
\begin{equation}
P(\kappa _{k}<c)\leq Cn^{-c},\ P\left( \bar{\lambda}_{0}<4\Vert Z_{k}\Vert
_{\infty }\right) \leq Cn^{-c},\  \bar{\lambda}_{0}\lesssim \left( \frac{\log
(pn)}{n}\right) ^{1/2}.  \label{eq: to prove 1}
\end{equation}%
To prove the first claim in \eqref{eq: to prove 1}, note that 
\begin{equation}
1\lesssim \Vert \delta \Vert _{2,n,-k}\lesssim 1
\label{eq: restricted eigenvalues}
\end{equation}%
with probability at least $1-Cn^{-c}$ uniformly over all $\delta \in \mathbb{%
R}^{p}$ such that $\Vert \delta \Vert _{2}=1$ and $\Vert \delta
_{T^{c}}\Vert _{0}\leq s\log n$ by Lemma \ref{lem: eigenvalue bound} and
Assumptions \ref{as: covariates}, \ref{as: growth condition} and %
\ref{as: dataset partition}. Hence, the first claim in 
\eqref{eq:
to prove 1} follows from Lemma 10 in \cite{BC11} applied with $m$ there
equal to $s\log n$ here.

To prove the second and the third claims in \eqref{eq: to prove 1}, note
that we have
$
\max_{1\leq j\leq p}\sum_{i\notin I_{k}}{\mathrm{E}}[|X_{ij}\varepsilon
_{i}|^{2}]\lesssim n
$
by Assumptions \ref{as: covariates}(b) and \ref{as: heterogeneity}. Also, 
\begin{equation*}
\Big({\mathrm{E}}\Big[\max_{1\leq i\leq n}\max_{1\leq j\leq
p}|X_{ij}\varepsilon _{i}|^{2}\Big]\Big)^{1/2}\leq \Big({\mathrm{E}}\Big[%
\max_{1\leq i\leq n}\max_{1\leq j\leq p}|X_{ij}\varepsilon _{i}|^{q}\Big]%
\Big)^{1/q}\lesssim n^{1/q}M_{n}.
\end{equation*}%
Thus, by Lemma \ref{lem: maximal inequality} and Assumption \ref{as: growth
condition}, 
\begin{equation*}
{\mathrm{E}}\Big[(n-n_{k})\Vert Z_{k}\Vert _{\infty }\Big]\lesssim \sqrt{%
n\log p}+n^{1/q}M_{n}\log p\lesssim \sqrt{n\log p}.
\end{equation*}%
Hence, applying Lemma \ref{lem: maximal inequality 2} with $t=(n\log
n)^{1/2} $ and $Z$ there replaced by $(n-n_{k})\Vert Z_{k}\Vert _{\infty }$
here and noting that $nM_{n}^{q}/(n\log n)^{q/2}\leq Cn^{-c}$ by Assumption %
\ref{as: growth condition} implies that 
\begin{equation*}
\Vert Z_{k}\Vert _{\infty }\lesssim \left( \frac{\log (pn)}{n}\right) ^{1/2}
\end{equation*}%
with probability at least $1-Cn^{-c}$. Hence, noting that $\log ^{4}(pn)\leq
Cn$ by Assumptions \ref{as: covariates}(a) and \ref{as: growth condition}, it follows from Assumption %
\ref{as: candidate set} that there exists $\bar{\lambda}_{0}\in \Lambda _{n}$
such that the second and the third claims in \eqref{eq: to prove
1} hold. By \eqref{eq: bc bound main}, this $\bar\lambda_0$ satisfies the following bound:
\begin{equation}\label{eq: prediction norm classic lambda}
P\left(\|\widehat\beta_{-k}(\bar\lambda_0) - \beta\|_{2,n,-k}^2 > \frac{Cs\log(p n)}{n}\right)\leq C n^{-c}.
\end{equation}
Now, to prove the asserted claims, note that using 
\eqref{eq: to
prove 1} and \eqref{eq: restricted eigenvalues} and applying Theorem 2 in 
\cite{BC11} with $m=s\log n$ there shows that $\Vert \widehat{\beta }_{-k}(%
\bar{\lambda}_{0})\Vert _{0}\lesssim s$ with probability at least $1-Cn^{-c}$. Hence,
\begin{align*}
\Vert \widehat{\beta }_{-k}(\bar{\lambda}_{0})-\beta \Vert _{1}^{2}
&\lesssim s\Vert \widehat{\beta }_{-k}(\bar{\lambda}_{0})-\beta \Vert
_{2}^{2}  \notag \\
&\lesssim s\Vert \widehat{\beta }_{-k}(\bar{\lambda}_{0})-\beta \Vert
_{2,n,-k}^{2}\lesssim \frac{s^{2}\log (pn)}{n}  \label{l1}
\end{align*}%
again with probability at least $1-Cn^{-c}$, where the second inequality follows from \eqref{eq: restricted eigenvalues},
and the third one from \eqref{eq: prediction norm classic lambda}. This gives all asserted claims and completes the proof of the lemma.
\end{proof}

\begin{lemma}
\label{lem: 2} Under Assumptions \ref{as: covariates}--\ref{as: dataset partition}, we have
for all $k=1,\dots ,K$ that 
\begin{equation*}
\Vert \widehat{\beta }_{-k}(\bar{\lambda}_{0})-\beta \Vert
_{2,n,k}^{2}\lesssim \frac{s\log (pn)}{n}
\end{equation*}%
with probability $1-Cn^{-c}$ for $\bar{\lambda}_{0}$ defined in Lemma \ref%
{lem: 1}.
\end{lemma}
\begin{remark}
We thank one of the anonymous referees for suggesting the proof below.
The suggestion relaxes the condition $s^{2}/n=o(1)$ in our early proof to $%
s/n=o(1)$, up to some log factors.\qed
\end{remark}

\begin{proof}
Fix $k=1,\dots,K$ and denote $\widehat\beta = \widehat\beta_{-k}(\bar\lambda_0)$. By Lemma \ref{lem: 1}, $\|\widehat\beta\|_0 \lesssim s$ with probability at least $1 - C n^{-c}$. Hence, 
\begin{align*}
\| \widehat{\beta }-\beta \|_{2,n,k}^{2} 
&\lesssim ( \widehat{\beta }-\beta)'\Ep[XX']( \widehat{\beta }-\beta) + \|\hb-\beta\|_2^2 
\lesssim \|\widehat{\beta }-\beta\|_2^2\lesssim \frac{s\log(pn)}{n}
\end{align*}
with probability at least $1 - C n^{-c}$, where the first inequality follows from Lemma \ref{lem: eigenvalue bound} and Assumption \ref{as: dataset partition}, the second from Assumption \ref{as: covariates}(b), and the third from Lemma \ref{lem: 1}. The asserted claim follows.
\end{proof}

\begin{lemma}
\label{lem: 3} Under Assumptions \ref{as: covariates}--\ref{as: dataset partition}, we have for all $k=1,\dots ,K$ that 
\begin{equation*}
\Vert \widehat{\beta }_{-k}(\widehat{\lambda })-\beta \Vert
_{2,n,k}^{2}\lesssim \frac{s\log (pn)}{n} + \frac{\log^{r+1}n}{n}
\end{equation*}%
with probability at least $1-Cn^{-c}$.
\end{lemma}

\begin{proof}
By the definition of $\widehat{\lambda }$ in \eqref{eq: cross validation}, 
\begin{equation*}
\sum_{k=1}^{K}\sum_{i\in I_{k}}(Y_{i}-X_{i}^{\prime }\widehat{\beta }_{-k}(%
\widehat{\lambda }))^{2}\leq \sum_{k=1}^{K}\sum_{i\in
I_{k}}(Y_{i}-X_{i}^{\prime }\widehat{\beta }_{-k}(\bar{\lambda}_{0}))^{2}
\end{equation*}%
for $\bar{\lambda}_{0}$ defined in Lemma \ref{lem: 1}. Therefore, 
\begin{align*}
\sum_{k=1}^{K}n_{k}\Vert \widehat{\beta }_{-k}(\widehat{\lambda })-\beta
\Vert _{2,n,k}^{2}& \leq \sum_{k=1}^{K}n_{k}\Vert \widehat{\beta }_{-k}(\bar{%
\lambda}_{0})-\beta \Vert _{2,n,k}^{2} \\
& \quad +2\sum_{k=1}^{K}\sum_{i\in I_{k}}\varepsilon _{i}X_{i}^{\prime }(%
\widehat{\beta }_{-k}(\widehat{\lambda })-\widehat{\beta }_{-k}(\bar{\lambda}%
_{0})).
\end{align*}%
Further, for all $k=1,\dots ,K$, denote $D_{k}=\{(X_{i},Y_{i})_{i\notin I_{k}},(X_i)_{i\in I_k}\}$
and 
\begin{equation*}
Z_{k}=\max_{\lambda \in \Lambda _{n}}\left \vert \frac{\sum_{i\in
I_{k}}\varepsilon _{i}X_{i}^{\prime }(\widehat{\beta }_{-k}(\lambda )-%
\widehat{\beta }_{-k}(\bar{\lambda}_{0}))}{\sqrt{n_{k}}\Vert \widehat{\beta }%
_{-k}(\lambda )-\widehat{\beta }_{-k}(\bar{\lambda}_{0})\Vert _{2,n,k}}%
\right \vert .
\end{equation*}%
Then by Lemma \ref{lem: maximal inequality} and Assumptions \ref{as:
heterogeneity} and \ref{as: candidate set}, we have that $\Ep[Z_k\mid D_k] \lesssim \sqrt{\log\log n} + M_{n,k}\log\log n$, where
\begin{align*}
M_{n,k} &= \left( {\mathrm{E}}\left[
\max_{\lambda \in \Lambda _{n}}\max_{i\in I_{k}}\frac{(\varepsilon
_{i}X_{i}^{\prime }(\widehat{\beta }_{-k}(\lambda )-\widehat{\beta }_{-k}(%
\bar{\lambda}_{0})))^{2}}{n_{k}\Vert \widehat{\beta }_{-k}(\lambda )-%
\widehat{\beta }_{-k}(\bar{\lambda}_{0})\Vert _{2,n,k}^{2}}\mid D_{k}\right]
\right) ^{1/2} \\
&\leq \left( {\mathrm{E}}\left[\max_{\lambda \in \Lambda _{n}}\max_{i\in I_{k}}\frac{|\varepsilon
_{i}X_{i}^{\prime }(\widehat{\beta }_{-k}(\lambda )-\widehat{\beta }_{-k}(%
\bar{\lambda}_{0}))|^{u}}{n_{k}^{u/2}\Vert \widehat{\beta }_{-k}(\lambda )-%
\widehat{\beta }_{-k}(\bar{\lambda}_{0})\Vert _{2,n,k}^{u}}\mid D_{k}\right]
\right) ^{1/u} \\
& \leq \left( {\mathrm{E}}\left[\sum_{\lambda \in \Lambda _{n}}\sum_{i\in I_{k}}\frac{|\varepsilon
_{i}X_{i}^{\prime }(\widehat{\beta }_{-k}(\lambda )-\widehat{\beta }_{-k}(%
\bar{\lambda}_{0}))|^{u}}{n_{k}^{u/2}\Vert \widehat{\beta }_{-k}(\lambda )-%
\widehat{\beta }_{-k}(\bar{\lambda}_{0})\Vert _{2,n,k}^{u}}\mid D_{k}\right]
\right) ^{1/u}
\end{align*} 
for any $u\geq 2$. In turn, the last expression is bounded from above by
$$
C\left((u(1 + r))^{u(1 + r)/2}\log n\right)^{1/u}
$$
since (i) $\Ep[|\varepsilon_i|^u\mid D_k] \leq C^u\Ep[(1 + |e| + |e|^{r+1})^u] \leq C^u(u(1+r))^{u(1+r)/2}$ under Assumption \ref{as: heterogeneity} and (ii) for any sequence $(a_i)_{i\in I_k}$ in $\mathbb R$, we have 
$
\sum_{i\in I_k}|a_i|^u \leq (\sum_{i\in I_k}|a_i|^2)^{u/2}.
$
Using this bound with $u = 3$ (for example) gives $\Ep[Z_k\mid D_k] \lesssim \sqrt{\log n}$. In addition, using this bound with $u = \log n$, it follows from Lemma \ref{lem: maximal inequality 2} that
$$
\Pr\Big(Z_k > 2\Ep[Z_k\mid D_k] + C\sqrt{\log^{r+1} n}\mid D_k\Big)\leq Cn^{-c}.
$$
Hence, $Z_k \lesssim \sqrt{\log^{r+1} n}$ with probability at least $1-Cn^{-c}$, and so, with the same probability,
\begin{equation*}
\left \vert \sum_{i\in I_{k}}\varepsilon _{i}X_{i}^{\prime }(\widehat{\beta }%
_{-k}(\widehat{\lambda })-\widehat{\beta }_{-k}(\bar{\lambda}_{0}))\right
\vert \lesssim \sqrt{n\log^{r+1} n}\Vert \widehat{\beta }_{-k}(\widehat{%
\lambda })-\widehat{\beta }_{-k}(\bar{\lambda}_{0})\Vert _{2,n,k}.
\end{equation*}%
Therefore, since $n_{k}/n\geq c_{1}$ by Assumption \ref{as: dataset
partition}, we have with the same probability that 
\begin{align*}
\sum_{k=1}^{K}\Vert \widehat{\beta }_{-k}(\widehat{\lambda })-\beta \Vert
_{2,n,k}^{2}& \lesssim \sum_{k=1}^{K}\Vert \widehat{\beta }_{-k}(\bar{\lambda%
}_{0})-\beta \Vert _{2,n,k}^{2} \\
& \quad +\sqrt{\frac{\log^{r+1} n}{n}}\sum_{k=1}^{K}\Vert \widehat{\beta }%
_{-k}(\widehat{\lambda })-\widehat{\beta }_{-k}(\bar{\lambda}_{0})\Vert
_{2,n,k},
\end{align*}%
and thus, by the triangle inequality, 
\begin{align*}
\Vert \widehat{\beta }_{-\widehat{k}}(\widehat{\lambda })-\widehat{\beta }_{-%
\widehat{k}}(\bar{\lambda}_{0})\Vert _{2,n,\widehat{k}}^{2}& \lesssim
\sum_{k=1}^{K}\Vert \widehat{\beta }_{-k}(\bar{\lambda}_{0})-\beta \Vert
_{2,n,k}^{2} \\
& \quad +\sqrt{\frac{\log^{r+1} n}{n}}\Vert \widehat{\beta }_{-\widehat{k}}(%
\widehat{\lambda })-\widehat{\beta }_{-\widehat{k}}(\bar{\lambda}_{0})\Vert
_{2,n,\widehat{k}},
\end{align*}%
where $\widehat{k}$ is a value of $k=1,\dots ,K$ that maximizes $\Vert \widehat{\beta 
}_{-k}(\widehat{\lambda })-\widehat{\beta }_{-k}(\bar{\lambda}_{0})\Vert
_{2,n,k}$. Therefore, by Lemma \ref{lem: 2},
\begin{equation*}
\Vert \widehat{\beta }_{-\widehat{k}}(\widehat{\lambda })-\widehat{\beta }_{-%
\widehat{k}}(\bar{\lambda}_{0})\Vert _{2,n,\widehat{k}}^{2}\lesssim \frac{%
s\log (pn)}{n}+\sqrt{\frac{\log^{r+1} n}{n}}\Vert \widehat{\beta }_{-\widehat{k%
}}(\widehat{\lambda })-\widehat{\beta }_{-\widehat{k}}(\bar{\lambda}%
_{0})\Vert _{2,n,\widehat{k}},
\end{equation*}%
and thus, for all $k=1,\dots ,K$, 
\begin{equation*}
\Vert \widehat{\beta }_{-k}(\widehat{\lambda })-\widehat{\beta }_{-k}(\bar{%
\lambda}_{0})\Vert _{2,n,k}^{2}\leq \Vert \widehat{\beta }_{-\widehat{k}}(%
\widehat{\lambda })-\widehat{\beta }_{-\widehat{k}}(\bar{\lambda}_{0})\Vert
_{2,n,\widehat k}^{2}\lesssim \frac{s\log (pn)}{n} + \frac{\log^{r+1}n}{n},
\end{equation*}%
again with probability at least $1-Cn^{-c}$. The asserted claim now follows from combining this bound with the triangle
inequality and Lemma \ref{lem: 2}. This completes the proof of the lemma.
\end{proof}

\begin{lemma}
\label{lem: 4} Under Assumptions \ref{as: covariates} -- \ref{as: dataset
partition}, we have for all $k=1,\dots ,K$ that 
\begin{equation*}
\Vert \widehat{\beta }_{-k}(\widehat{\lambda })-\beta \Vert _{2}^{2}\lesssim 
\frac{s\log (pn)}{n} + \frac{\log^{r+1}n}{n}
\end{equation*}%
with probability at least $1-Cn^{-c}$.
\end{lemma}

\begin{proof}
Fix $k=1,\dots ,K$. For $\lambda \in \Lambda _{n}$, let $\delta _{\lambda }=(\widehat{\beta }_{-k}(\lambda )-\beta )/\Vert \widehat{\beta }_{-k}(\lambda
)-\beta \Vert _{2}$ and observe that conditional on $(X_{i},Y_{i})_{i\notin I_{k}}$, $(\delta _{\lambda })_{\lambda \in \Lambda _{n}}$ is non-stochastic. Hence, by Lemma \ref{lem: lasso continuity}, Assumptions \ref{as: covariates}(a) and \ref{as: dataset partition} and Chebyshev's inequality applied conditional on $(X_i,Y_i)_{i\notin I_k}$, for any $\lambda\in\Lambda_n$,
$
(n_k)^{-1}\sum_{i\in I_k}(X_i'\delta_\lambda)^2 \geq c
$
with probability at least $1 - C n^{-c}$, and so
$
\| \hb_{-k}(\lambda) - \beta \|_2^2 \leq C\| \hb_{-k}(\lambda) - \beta \|_{2,n,k}^2
$
with the same probability. Therefore, by Assumption \ref{as: candidate set} and the union bound,
$
\| \hb_{-k}(\widehat\lambda) - \beta \|_2^2 \leq C\| \hb_{-k}(\widehat\lambda) - \beta \|_{2,n,k}^2
$
with probability at least $1 - C n^{-c}$. The asserted claim follows from combining this inequality and Lemma \ref{lem: 3}. 
\end{proof}

%\begin{lemma}
%There exists a constant $C$ such that for all $k=1,\dots,K$, all eigenvalues of the matrix $(n - n_k)^{-1}\sum_{i=1}^nX_i X_i'$ are bounded from above by $C(1 + p\log p/n)$ with probability $1 - o(1)$.
%\end{lemma}

\begin{lemma}
\label{lem: 4.51} Fix $k=1,\dots ,K$ and denote 
\begin{equation}
\Lambda _{n,k}(X_{1}^{n},T)=\left \{ \lambda \in \Lambda _{n}\colon {\mathrm{%
E}}[\Vert \widehat{\beta }_{-k}(\lambda )-\beta \Vert _{2,n,-k}\mid
X_{1}^{n}]\leq T\right \} ,\quad T>0.  \label{eq: lambda set definition}
\end{equation}%
Then under Assumptions \ref{as: covariates} -- \ref{as: dataset partition},
we have that $\widehat{\lambda }\in \Lambda _{n,k}(X_{1}^{n},T_{n})$ with
probability at least $1-Cn^{-c}$, where 
\begin{equation}\label{eq: definition of tn relevant risk}
T_{n}=C\left(\frac{s\log (pn)}{n} + \frac{\log^{r+1}n}{n}\right)^{1/2}.
\end{equation}
\end{lemma}

\begin{proof}
Fix $k=1,\dots ,K$ and note that by Assumption \ref{as: covariates}(b) and Lemma \ref{lem: eigenvalue bound}, 
$
\sup_{\delta\in\mathcal S^p(s)}\|\delta\|_{2,n}\leq C
$ 
with probability at least $1 - C n^{-c}$. Hence, by Lemma \ref{lem: 4}, Theorem \ref{thm: sparsity bound via l2 norm}, Assumption \ref{as: candidate set}, and the union bound,
$
\| \hb_{-k}(\widehat\lambda) \|_0 \leq n^{c_1/4}s\log^2(p n)
$
with probability at least $1 - C n^{-c}$. Further, by Assumption \ref{as: growth condition},
$
n^{c_1/4} s\log^2(p n) \leq \sqrt{s n^{1 + c_1/2}}\log(p n)
$
for all $n\geq n_0$ with $n_0$ depending only on $c_1$ and $C_1$, and so, by Assumption \ref{as: covariates}(b) and Lemma \ref{lem: eigenvalue bound},
$
\| \hb_{-k}(\widehat\lambda) - \beta \|_{2,n,-k}^2 \lesssim \| \hb_{-k}(\widehat\lambda) - \beta \|_2^2
$
with probability at least $1 - C n^{-c}$. Combining this bound with Lemma \ref{lem: 4} now gives
\begin{equation}\label{eq: lemma 86 almost there}
\Pr\left(\| \hb_{-k}(\widehat\lambda) - \beta \|_{2,n,-k}^2 > C\left(\frac{s\log(p n)}{n} + \frac{\log^{r+1}n}{n}\right)\right) \leq C n^{-c}.
\end{equation}
Now, 
\begin{align*}
& \Pr \left( \widehat{\lambda }\notin \Lambda _{n,k}(X_{1}^{n},T_{n})\right) \leq \Pr \left( \Vert \widehat{\beta }_{-k}(\widehat{\lambda }%
)-\beta \Vert _{2,n,-k}>T_{n}/2\right) \\
& \quad +\Pr \left( \max_{\lambda \in \Lambda _{n}}\left \vert \Vert 
\widehat{\beta }_{-k}(\lambda )-\beta \Vert _{2,n,-k}-{\mathrm{E}}[\Vert 
\widehat{\beta }_{-k}(\lambda )-\beta \Vert _{2,n,-k}\mid X_{1}^{n}]\right
\vert >T_{n}/2\right),
\end{align*}
where the first and the second terms on the right-hand side are at most $Cn^{-c}$ by \eqref{eq: lemma 86 almost there} and Lemma \ref{lem: 0.7} applied with $\kappa = \log n$, respectively, as long the constant $C$ in the definition of $T_n$ is large enough. The asserted claim follows.
\end{proof}

\begin{lemma}
\label{lem: chatterjee} For all $\lambda \in \Lambda _{n}$ and $b\in \mathbb{%
R}^{p}$, we have 
\begin{equation*}
\Vert \widehat{\beta }(\lambda )-b\Vert _{2,n}^{2}\leq \frac{1}{n}%
\sum_{i=1}^{n}(Y_{i}-X_{i}^{\prime }b)^{2}+\lambda \Vert b\Vert _{1}-\frac{1%
}{n}\sum_{i=1}^{n}(Y_{i}-X_{i}^{\prime }\widehat{\beta }(\lambda
))^{2}-\lambda \Vert \widehat{\beta }(\lambda )\Vert _{1}.
\end{equation*}
\end{lemma}

\begin{proof}
The result in this lemma is sometimes referred to as the two point
inequality; see Section 2.4 in \cite{G16}, where the proof is also provided. 
\end{proof}

\subsection{Proof of Theorem \protect \ref{thm: cross validation}\label{sec:
proof of theorem 4.1}}
Throughout the proof, we can assume that $\alpha\in [1/n, e^{-2}]$ since the results for $\alpha > e^{-2}$ and $\alpha < 1/n$ follow from the cases $\alpha = e^{-2}$ and $\alpha = 1/n$, respectively, with suitably increased constant $C$. We proceed in three steps. In the first step, for any given $\lambda > 0$, we use Lemma \ref{lem: chatterjee} to provide an upper bound on the conditional median of $n\Vert \widehat{\beta }(\lambda )-\beta \Vert
_{2,n}^{2}$ given $X_{1}^{n}$ via some functionals of subsample estimators $\hb_{-k}(\lambda)$. In the second step, we derive bounds on these functionals for relevant values of $\lambda$ with the help of Theorem \ref{lem: sparsity bound low dimensional}. In the third step, we use Lemma \ref{lem: 4.51} to show that $\widehat\lambda$ belongs to the relevant set with high probability and Lemma \ref{lem: 0.7} to replace conditional medians by conditional expectations and complete the proof.

\medskip

\noindent \textbf{Step 1.} For any random variable $Z$ and any number $%
\alpha $, let $Q_{\alpha }(Z\mid X_{1}^{n})$ denote the $\alpha $th quantile
of the conditional distribution of $Z$ given $X_{1}^{n}$. In this step, we
show that for any $\lambda >0$, 
\begin{align}
& Q_{1/2}(n\Vert \widehat{\beta }(\lambda )-\beta \Vert _{2,n}^{2}\mid
X_{1}^{n})  \notag \\
& \quad \lesssim \sum_{k=1}^{K}Q_{1-1/(16K)}\left( \sum_{i\notin
I_{k}}(X_{i}^{\prime }(\widehat{\beta }_{-k}(\lambda )-\beta ))^{2}\mid
X_{1}^{n}\right)  \notag \\
& \qquad +\sum_{k=1}^{K}Q_{1-1/(16K)}\left( \sum_{i\in I_{k}}(X_{i}^{\prime }(%
\widehat{\beta }_{-k}(\lambda )-\beta ))^{2}\mid X_{1}^{n}\right)  \notag \\
& \qquad +\sum_{k=1}^{K}Q_{1-1/(16K)}\left( \left \vert \sum_{i\notin
I_{k}}\varepsilon _{i}X_{i}^{\prime }(\widehat{\beta }_{-k}(\lambda )-\beta
)\right \vert \mid X_{1}^{n}\right)  \notag \\
& \qquad +\sum_{k=1}^{K}Q_{1-1/(16K)}\left( \left \vert \sum_{i\in
I_{k}}\varepsilon _{i}X_{i}^{\prime }(\widehat{\beta }_{-k}(\lambda )-\beta
)\right \vert \mid X_{1}^{n}\right) .  \label{eq: step 1 statement}
\end{align}%
To do so, fix any $\lambda >0$ and denote 
\begin{equation}
b(\lambda )=\frac{1}{K-1}\sum_{k=1}^{K}\frac{n-n_{k}}{n}\widehat{\beta }%
_{-k}(\lambda ).  \label{eq: def b}
\end{equation}%
Then 
\begin{align*}
& \sum_{i=1}^{n}(Y_{i}-X_{i}^{\prime }\widehat{\beta }(\lambda
))^{2}+n\lambda \Vert \widehat{\beta }(\lambda )\Vert _{1} \\
& \quad =\frac{1}{K-1}\sum_{k=1}^{K}\left( \sum_{i\notin
I_{k}}(Y_{i}-X_{i}^{\prime }\widehat{\beta }(\lambda ))^{2}+(n-n_{k})\lambda
\Vert \widehat{\beta }(\lambda )\Vert _{1}\right) \\
& \quad \geq \frac{1}{K-1}\sum_{k=1}^{K}\left( \sum_{i\notin
I_{k}}(Y_{i}-X_{i}^{\prime }\widehat{\beta }_{-k}(\lambda
))^{2}+(n-n_{k})\lambda \Vert \widehat{\beta }_{-k}(\lambda )\Vert
_{1}\right) \\
& \quad \geq \frac{1}{K-1}\sum_{k=1}^{K}\sum_{i\notin
I_{k}}(Y_{i}-X_{i}^{\prime }\widehat{\beta }_{-k}(\lambda ))^{2}+n\lambda
\Vert b(\lambda )\Vert _{1},
\end{align*}%
where the second line follows from the definition of $\widehat{\beta }%
_{-k}(\lambda )$'s and the third one from the triangle inequality. Also, 
\begin{align*}
& \frac{1}{K-1}\sum_{k=1}^{K}\sum_{i\notin I_{k}}(Y_{i}-X_{i}^{\prime }%
\widehat{\beta }_{-k}(\lambda ))^{2} \\
& \geq \frac{1}{K-1}\sum_{k=1}^{K}\sum_{i\notin I_{k}}\left(
(Y_{i}-X_{i}^{\prime }b(\lambda ))^{2}+2(Y_{i}-X_{i}^{\prime }b(\lambda
))(X_{i}^{\prime }b(\lambda )-X_{i}^{\prime }\widehat{\beta }_{-k}(\lambda
))\right) \\
& =\sum_{i=1}^{n}(Y_{i}-X_{i}^{\prime }b(\lambda ))^{2}+\frac{2}{K-1}%
\sum_{k=1}^{K}\sum_{i\notin I_{k}}(Y_{i}-X_{i}^{\prime }b(\lambda
))(X_{i}^{\prime }b(\lambda )-X_{i}^{\prime }\widehat{\beta }_{-k}(\lambda
)).
\end{align*}%
Thus, by Lemma \ref{lem: chatterjee}, 
\begin{equation*}
n\Vert \widehat{\beta }(\lambda )-b(\lambda )\Vert _{2,n}^{2}\leq \frac{2}{%
K-1}\sum_{k=1}^{K}\sum_{i\notin I_{k}}(Y_{i}-X_{i}^{\prime }b(\lambda
))(X_{i}^{\prime }\widehat{\beta }_{-k}(\lambda )-X_{i}^{\prime }b(\lambda
)).
\end{equation*}%
Substituting here $Y_{i}=X_{i}^{\prime }\beta +\varepsilon _{i}$, $i=1,\dots
,n$, and the definition of $b(\lambda )$ in \eqref{eq: def b} and using the
triangle inequality gives 
\begin{align}
n\Vert \widehat{\beta }(\lambda )-\beta \Vert _{2,n}^{2}& \lesssim n\Vert 
\widehat{\beta }(\lambda )-b(\lambda )\Vert _{2,n}^{2}+n\Vert b(\lambda
)-\beta \Vert _{2,n}^{2}  \notag \\
& \lesssim \sum_{k=1}^{K}\sum_{i\notin I_{k}}(X_{i}^{\prime }(\widehat{\beta 
}_{-k}(\lambda )-\beta ))^{2}+\sum_{k=1}^{K}\sum_{i\in I_{k}}(X_{i}^{\prime
}(\widehat{\beta }_{-k}(\lambda )-\beta ))^{2}  \notag \\
& +\sum_{k=1}^{K}\left \vert \sum_{i\notin I_{k}}\varepsilon
_{i}X_{i}^{\prime }(\widehat{\beta }_{-k}(\lambda )-\beta )\right \vert
+\sum_{k=1}^{K}\left \vert \sum_{i\in I_{k}}\varepsilon _{i}X_{i}^{\prime }(%
\widehat{\beta }_{-k}(\lambda )-\beta )\right \vert .
\label{eq: step 1 intermediate}
\end{align}%
The claim of this step, inequality \eqref{eq: step 1 statement}, follows
from \eqref{eq: step 1 intermediate} and Lemma \ref{lem: quantile properties}%
.

\medskip

\noindent \textbf{Step 2.} Denote 
\begin{equation}
\Lambda _{n}(X_{1}^{n},T)=\cap _{k=1}^{K}\Lambda _{n,k}(X_{1}^{n},T),\quad
T>0,  \label{l5}
\end{equation}%
for $\Lambda _{n,k}(X_{1}^{n},T)$ defined in 
\eqref{eq: lambda set
definition} of Lemma \ref{lem: 4.51}. In this step, we show that 
\begin{multline}
\Pr \Big(\max_{\lambda \in \Lambda _{n}(X_{1}^{n},T_{n})}Q_{1/2}(n\Vert 
\widehat{\beta }(\lambda )-\beta \Vert _{2,n}^{2}\mid X_{1}^{n})
\label{eq: step 2 statement} \\
>C(\log p)(s\log (pn)+\log ^{r+1}n)\Big)\leq Cn^{-c}
\end{multline}%
for $T_{n}$ defined in \eqref{eq: definition of tn relevant risk} of Lemma \ref{lem: 4.51}.

To do so, we apply the result in Step 1 and bound all terms on the
right-hand side of \eqref{eq: step 1 statement} in turn. To start, fix $%
k=1,\dots ,K$. Then for any $\lambda \in \Lambda _{n}(X_{1}^{n},T_{n})$, 
\begin{align}
& Q_{1-1/(16K)}\bigg( \sum_{i\notin I_{k}}(X_{i}^{\prime }(\widehat{\beta }%
_{-k}(\lambda )-\beta ))^{2}\mid X_{1}^{n}\bigg)  \notag \\
%& \quad =Q_{1-1/(16K)}\bigg( \bigg( \sum_{i\notin I_{k}}(X_{i}^{\prime }(%
%\widehat{\beta }_{-k}(\lambda )-\beta ))^{2}\bigg) ^{1/2}\mid
%X_{1}^{n}\bigg) ^{2}  \notag \\
& \quad  \leq nQ_{1-1/(16K)}\Big(\Vert \widehat{\beta }_{-k}(\lambda )-\lambda
\Vert _{2,n,-k}\mid X_{1}^{n}\Big)^{2}  \notag \\
& \quad \leq 16Kn\Big({\mathrm{E}}[\Vert \widehat{\beta }_{-k}(\lambda
)-\beta \Vert _{2,n,-k}\mid X_{1}^{n}]\Big)^{2}\lesssim
s\log (pn) + \log^{r+1}n,  \label{eq: step 2 - 1}
\end{align}%
where the third line follows from Markov's inequality and the definition of $T_n$.

Next, since Assumption \ref{as: covariates}(a) implies that $\Ep[(X'\delta)^2]\geq c$ for all $\delta\in\mathcal S^p(n+s)$, it follows from Lemma \ref{lem: eigenvalue bound} and Assumptions \ref{as: covariates} and \ref{as: growth condition} that
$$
c\leq \inf_{\delta\in \mathcal S^p(n^{1/2 + c_1/8}(\sqrt n T_n + 1))}\|\delta\|_{2,n,-k}\quad\text{and}\quad \max_{1\leq j\leq p}\sqrt{\frac{1}{n-n_k}\sum_{i\notin I_k} X_{i j}^2} \leq C
$$
with probability at least $1 - C n^{-c}$. Thus, by Theorem \ref{lem: sparsity bound low dimensional} and the union bound,
\begin{align}
\max_{\lambda \in \Lambda _{n}(X_{1}^{n},T_{n})}{\mathrm{E}}[\Vert \widehat{%
\beta }_{-k}(\lambda )\Vert _{0}\mid X_{1}^{n}]& \lesssim s+(\log
p)(nT_{n}^{2}+\log ^{r}n) \nonumber\\
& \lesssim (\log p)(s\log (pn)+\log ^{r+1}n)\label{eq: getting there sparsity bound 1}
\end{align}%
with probability at least $1-Cn^{-c}$. Thus, by Markov's inequality, 
Lemma \ref{lem: eigenvalue bound}, and Assumptions \ref{as: covariates}, \ref{as: growth condition}, and \ref{as: dataset partition}, 
\begin{align}
& \max_{\lambda \in \Lambda _{n}(X_{1}^{n},T_{n})}Q_{1-1/(16K)}\Big(%
\sum_{i\in I_{k}}(X_{i}^{\prime }(\widehat{\beta }_{-k}(\lambda )-\beta
))^{2}\mid X_{1}^{n}\Big)  \notag \\
& \lesssim \max_{\lambda \in \Lambda _{n}(X_{1}^{n},T_{n})}Q_{1-1/(32K)}\Big(%
n\Vert \widehat{\beta }_{-k}(\lambda )-\beta )\Vert _{2,n,-k}^{2}\mid
X_{1}^{n}\Big)\nonumber \\
& \lesssim s\log (pn) + \log^{r+1}n,  \label{eq: step 2 - 2}
\end{align}%
with probability at least $1 - C n^{-c}$, where the last inequality follows from the same argument as that in  \eqref{eq: step 2 - 1}.

Next, by Markov's inequality and the definition of $M_n$ in \eqref{eq: Mn definition}, 
$
\sum_{i\notin I_k}\|X_i\|_{\infty}^q \leq n^{1 + c_1}M_n^q
$ 
with probability at least $1 - C n^{-c}$. Hence, by Lemma \ref{lem: maximal inequality} and Assumptions \ref{as: growth condition} and \ref{as: heterogeneity},
$$
\Ep\left[\left\| \sum_{i\notin I_k}\varepsilon_i X_i \right\|_{\infty} | X_1^n\right] \lesssim \sqrt{n\log p} + n^{(1 + c_1)/q}M_n \log p\lesssim \sqrt{n\log p}
$$
with probability at least $1 - C n^{-c}$. Therefore, by Proposition A.1.6 in \cite{VW96},
\begin{align*}
\Ep\left[\left\| \sum_{i\notin I_k}\varepsilon_i X_i \right\|_{\infty}^4 | X_1^n\right]
 &\lesssim \left(\Ep\left[\left\| \sum_{i\notin I_k}\varepsilon_i X_i \right\|_{\infty} | X_1^n\right]\right)^4 + \Ep\left[ \max_{i\notin I_k}\|\varepsilon_i X_i\|_{\infty}^4 | X_1^n\right]\\
 &\lesssim \Big(\sqrt{n\log p} + n^{(1 + c_1)/q}M_n\Big)^4 \lesssim (n\log p)^2
\end{align*}
with probability at least $1 - C n^{-c}$. Thus, proceeding as in the proof of Theorem \ref{lem: sparsity bound low dimensional}, getting from \eqref{eq: thm 52 beginning} to \eqref{eq: thm 52 end}, with $\psi_i$'s replaced by $\varepsilon_i$'s, we obtain 
\begin{align*}
&\max_{\lambda \in \Lambda_n(X_1^n, T_n)} {\mathrm{E}}\Big[\Big| %
\sum_{i\notin I_k} \varepsilon_i X_i^{\prime }(\widehat \beta_{-k}(\lambda)
- \beta) \Big| \mid X_1^n \Big] \\
&\quad \lesssim \sqrt{n\log p}\Big(T_n + \sqrt{\frac{\log^r n}{n}}\Big)%
\max_{\lambda \in \Lambda_n(X_1^n, T_n)}({\mathrm{E}}[ \| \widehat
\beta_{-k}(\lambda) \|_0 + s\mid X_1^n ])^{1/2} \\
&\qquad + \sqrt n T_n + \sqrt{\log^r n} + \max_{\lambda \in \Lambda_n(X_1^n,
T_n)} {\mathrm{E}}[ \| \widehat \beta_{-k}(\lambda) \|_0 + s\mid X_1^n ] \\
&\quad \lesssim s + (\log p)(n T_n^2+ \log^r n) \lesssim (\log p)(s\log(p n)
+ \log^{r+1} n)
\end{align*}
with probability at least $1 - C n^{-c}$, where the second inequality
follows from \eqref{eq: getting there sparsity bound 1}. Hence, by
Markov's inequality, 
\begin{align*}
&\max_{\lambda \in \Lambda_n(X_1^n, T_n)}Q_{1 - 1/(16 K)}\Big(\Big|%
\sum_{i\notin I_k} \varepsilon_i X_i^{\prime }(\widehat \beta_{-k}(\lambda)
- \beta)\Big| \mid X_1^n \Big) \\
&\quad \leq 16 K \max_{\lambda \in \Lambda_n(X_1^n, T_n)} {\mathrm{E}}\Big[%
\Big| \sum_{i\notin I_k} \varepsilon_i X_i^{\prime }(\widehat
\beta_{-k}(\lambda) - \beta) \Big| \mid X_1^n \Big] \\
&\quad \lesssim (\log p)(s\log(p n) + \log^{r+1} n)
\end{align*}
with probability at least $1 - C n^{-c}$.

Finally, by Markov's inequality, for any $A_{1},A_2,\lambda>0$, 
\begin{align*}
& \Pr \Big(\Big|\sum_{i\in I_{k}}\varepsilon _{i}X_{i}^{\prime }(\widehat{%
\beta }_{-k}(\lambda )-\beta )\Big|>\sqrt{A_{1}A_{2}(s\log (pn) + \log^{r+1}n)}\mid X_{1}^{n}%
\Big) \\
& \quad \leq \Pr \Big(\sum_{i\in I_{k}}(X_{i}^{\prime }(\widehat{\beta }%
_{-k}(\lambda )-\beta ))^{2}>A_{2}(s\log (pn) + \log^{r+1}n)\mid X_{1}^{n}\Big) \\
& \quad \quad +{\mathrm{E}}\Big[1\Big \{ \sum_{i\in I_{k}}(X_{i}^{\prime }(%
\widehat{\beta }_{-k}(\lambda )-\beta ))^{2}\leq A_{2}(s\log (pn) + \log^{r+1}n)\Big \} \\
& \quad \quad \quad \times \frac{{\mathrm{E}}[|\sum_{i\in I_{k}}\varepsilon
_{i}X_{i}^{\prime }(\widehat{\beta }_{-k}(\lambda )-\beta )|^{2}\mid
X_{1}^{n},(Y_{i})_{i\notin I_{k}}]}{A_{1}A_{2}(s\log (pn) + \log^{r+1}n)}\mid X_{1}^{n}\Big]
\\
& \quad \lesssim \Pr \Big(\sum_{i\in I_{k}}(X_{i}^{\prime }(\widehat{\beta }%
_{-k}(\lambda )-\beta ))^{2}>A_{2}(s\log (pn) + \log^{r+1}n)\mid X_{1}^{n}\Big)+1/A_{1}.
\end{align*}%
Choosing both $A_{1}$ and $A_{2}$ here large enough and using the same
argument as that in \eqref{eq: step 2 - 2} shows that 
\begin{align*}
& \max_{\lambda \in \Lambda _{n}(X_{1}^{n},T_{n})}Q_{1-1/(16K)}\Big(\Big|%
\sum_{i\in I_{k}}\varepsilon _{i}X_{i}^{\prime }(\widehat{\beta }%
_{-k}(\lambda )-\beta )\Big|\mid X_{1}^{n}\Big) \\
& \quad \lesssim \sqrt{s\log (pn)+\log^{r+1}n}\lesssim s\log (pn) + \log^{r+1}n
\end{align*}%
with probability at least $1-Cn^{-c}$. Combining all inequalities presented
above together and using Step 1 gives \eqref{eq: step 2 statement}, which is the asserted
claim of this step.

\medskip

\noindent \textbf{Step 3.} Here we complete the proof. To do so, note that
by Lemma \ref{lem: 0.7} applied with $\kappa = 2$, for any $\lambda > 0$, 
\begin{equation*}
\Big| \| \widehat \beta(\lambda) - \beta \|_{2,n} - {\mathrm{E}}[ \|
\widehat \beta(\lambda) - \beta \|_{2,n} \mid X_1^n] \Big| \lesssim \sqrt{%
\frac{\log^r n}{n}}
\end{equation*}
with probability at least $3/4$, which implies that 
\begin{equation*}
\Big| Q_{1/2}(\| \widehat \beta(\lambda) - \beta \|_{2,n}\mid X_1^n) - {%
\mathrm{E}}[\| \widehat \beta(\lambda) - \beta \|_{2,n}\mid X_1^n] \Big| %
\lesssim \sqrt{\frac{\log^r n}{n}}.
\end{equation*}
Combining this inequality with \eqref{eq: step 2 statement} in Step 2 shows
that 
\begin{multline}  \label{eq: expectation bound for good lambda}
\Pr \bigg(\max_{\lambda \in \Lambda_n(X_1^n,T_n)}{\mathrm{E}}[\| \widehat
\beta(\lambda) - \beta \|_{2,n}\mid X_1^n] \\
> \sqrt{\frac{C s\log p}{n}}\times \sqrt{\log(p n) + s^{-1}\log^{r+1} n}\bigg) %
\leq C n^{-c}.
\end{multline}
Also, applying Lemma \ref{lem: 0.7} with $\kappa = \log(1/\alpha)\leq \log n$ and 
$$
t = \left(\frac{\tilde C\log(1/\alpha)\log^{r+1}n}{n}\right)^{1/2}
$$
with sufficiently large $\tilde C$, which can be chosen to depend only on $C_1$ and $r$, it follows that for any $\lambda > 0$,
\begin{align*}
&\Pr \Big(\Big|\Vert \widehat{\beta }(\lambda )-\beta \Vert _{2,n}-{\mathrm{E}%
}[\Vert \widehat{\beta }(\lambda)-\beta \Vert _{2,n}\mid X_{1}^{n}]\Big|>t\Big) \leq \left(\frac{C}{\tilde C\log n}\right)^{\frac{\log(1/\alpha)}{2}} \leq \frac{\alpha}{|\Lambda_n|}
\end{align*}
since $\alpha \leq e^{-2}$. Combining these inequalities and using the union bound, we obtain
\begin{multline}  \label{eq: almost there}
\Pr \bigg(\max_{\lambda \in \Lambda_n(X_1^n,T_n)}\| \widehat \beta(\lambda)
- \beta \|_{2,n} \\
> \sqrt{\frac{C s\log (p/\alpha)}{n}}\times \sqrt{\log(p n) + s^{-1}\log^{r+1}
n}\bigg) \leq \alpha + C n^{-c}.
\end{multline}
Finally, by Lemma \ref{lem: 4.51} and the union bound, 
\begin{equation}  \label{eq: there}
\Pr(\widehat \lambda \in \Lambda_n(X_1^n,T_n)) \geq 1 - Cn^{-c}.
\end{equation}
Combining the last two inequalities gives the asserted claim and completes
the proof of the theorem. \qed

\subsection{Proof of Theorem \protect \ref{thm: sparsity bound}\label{sec:
proof of theorem 4.2}}

Define $\Lambda _{n}(X_{1}^{n},T_n)$ as in Step 2 of the proof of
Theorem \ref{thm: cross validation}. Then by Assumptions \ref{as: covariates} and \ref{as: growth condition}, Lemma \ref{lem: eigenvalue bound}, Theorem \ref{lem: sparsity bound low dimensional}, and \eqref{eq: expectation bound for good
lambda} in the proof of Theorem \ref{thm: cross validation},
\begin{equation*}
\max_{\lambda \in \Lambda _{n}(X_{1}^{n},T_{n})}{\mathrm{E}}[\Vert \widehat{%
\beta }(\lambda )\Vert _{0}\mid X_{1}^{n}]\lesssim s(\log ^{2}p)(\log
(pn)+s^{-1}\log ^{r+1}n)
\end{equation*}%
with probability at least $1-Cn^{-c}$. Thus, by Markov's inequality, the
union bound, and Assumption \ref{as: candidate set}, for any $\bar{s}>0$, 
\begin{align*}
&\Pr \left( \max_{\lambda \in \Lambda _{n}(X_{1}^{n},T_{n})}\Vert \widehat{%
\beta }(\lambda )\Vert _{0}>\bar{s}\mid X_{1}^{n}\right) \\
&\qquad \lesssim s(\log
^{2}p)(\log n)(\log (pn)+s^{-1}\log ^{r+1}n)/\bar{s}
\end{align*}%
with probability at least $1-Cn^{-c}$. The asserted claim of the theorem
follows from combining this bound with \eqref{eq: there} in the proof of
Theorem \ref{thm: cross validation} and substituting 
\begin{equation*}
\bar{s}=Cs(\log ^{2}p)(\log n)(\log (pn)+s^{-1}\log ^{r+1}n)/\alpha
\end{equation*}%
with a sufficiently large constant $C>0$. This completes the proof of the theorem. \qed

\subsection{Proof of Theorem \protect \ref{thm: l1 and l2 bounds}\label{sec:
proof of theorem 4.3}}
Applying Theorem \ref{thm: sparsity bound} with $\alpha =n^{-c_{1}/4}$ shows that 
\begin{equation*}
\Vert \widehat{\beta }(\widehat{\lambda })\Vert _{0}\lesssim sn^{c_{1}/4}(\log
^{2}p)(\log n)(\log (pn)+s^{-1}\log ^{r+1}n)
\end{equation*}
with probability at least $1-Cn^{-c}$. Thus, by
Lemma \ref{lem: eigenvalue bound} and Assumptions \ref{as: covariates}(a) and \ref{as: growth condition}, 
$
\Vert \widehat{\beta }(\widehat{\lambda })-\beta \Vert _{2}\lesssim \Vert 
\widehat{\beta }(\widehat{\lambda })-\beta \Vert _{2,n}
$
with probability at least $1-Cn^{-c}$. The
asserted claim regarding $\Vert \widehat{\beta }(\widehat{\lambda })-\beta
\Vert _{2}$ follows from this bound and Theorem \ref{thm: cross validation}.

Also, by the Cauchy-Schwarz and triangle inequalities, 
\begin{equation*}
\| \widehat \beta(\widehat \lambda) - \beta \|_1 \leq \sqrt{\| \widehat
\beta(\widehat \lambda) - \beta \|_0}\| \widehat \beta(\widehat \lambda) -
\beta \|_2 \leq \sqrt{\| \widehat \beta(\widehat \lambda)\|_0 + s}\|
\widehat \beta(\widehat \lambda) - \beta \|_2.
\end{equation*}
The asserted claim regarding $\| \widehat \beta(\widehat \lambda) - \beta
\|_1$ follows from this bound, Theorem \ref{thm: sparsity bound}, and the
asserted claim regarding $\| \widehat \beta(\widehat \lambda) - \beta \|_2$.
This completes the proof of the theorem. \qed

%\section{Proofs for Section \ref{sec: bootstrap after cross validation}}
%To be provided.

\subsection{Proof of Corollary \ref{cor: out of sample prediction}}\label{sec: out of sample prediction}
 By Assumptions \ref{as: covariates}(a) and \ref{as: growth condition} and Theorem \ref{thm: sparsity bound},
\begin{equation*}
\|\widehat\beta(\widehat\lambda) - \beta\|_{p,2,n} \leq C\Vert \widehat{\beta }(\widehat{\lambda 
})-\beta \Vert _{2}
\end{equation*}%
with probability at least $1-Cn^{-c}$. Hence, by Theorem \ref{thm: l1 and l2 bounds},
\begin{equation*}
\|\widehat\beta(\widehat\lambda) - \beta\|_{p,2,n} \leq \sqrt{\frac{Cs\log (p/\alpha )}{n}}%
\times \sqrt{\log (pn)+s^{-1}\log ^{r+1}n},
\end{equation*}%
with probability at least $1-\alpha - Cn^{-c}$. The asserted claim follows.
\qed

\section{Proofs for Section \ref{sec: sparsity bound}\label{sec: proofs sparsity}}

In this section, we prove Theorems \ref{thm: sparsity bound via l2 norm} and \ref{lem: sparsity bound low dimensional}. Since the proofs are long, we start with a sequence of preliminary lemmas.

%Moreover, for $\lambda\in\Lambda_n$ and $e = (e_1,\dots,e_n)\in\R^n$, denote
%$$
%\hb^e(\lambda) = \argmin_{b\in\R^p}\left(\frac{1}{n}\sum_{i=1}^n (e_i + X_i'(\beta - b))^2 + \lambda \|b\|_1\right). 
%$$
%Note that $\hb(\lambda) = \hb^e(\lambda)$ if we substitute $e = (\eps_1,\dots,\eps_n)$.

\subsection{Preliminary Lemmas\label{sec: preliminary lemmas sparsity}}

%Here, we collect preliminary lemmas that help to prove Theorem \ref{lem: sparsity bound low dimensional} from the main text.

\begin{lemma}
\label{lem: lasso continuity} For all $\lambda >0$, the Lasso estimator $%
\widehat{\beta }(\lambda )$ given in \eqref{eq: lasso} based on the data $%
(X_{i},Y_{i})_{i=1}^{n}=(X_{i},X_{i}^{\prime }\beta +\varepsilon
_{i})_{i=1}^{n}$ has the following property: the function $(\varepsilon
_{i})_{i=1}^{n}\mapsto (X_{i}^{\prime }\widehat{\beta }(\lambda ))_{i=1}^{n}$
mapping $\mathbb{R}^{n}$ to $\mathbb{R}^{n}$ for any fixed value of $%
X_{1}^{n}=(X_{1},\dots ,X_{n})$ is well-defined and is Lipschitz-continuous
with Lipschitz constant one with respect to Euclidean norm. Moreover, there always exists a Lasso estimator 
$\widehat{\beta }(\lambda )$ such that $\Vert \widehat{\beta }(\lambda
)\Vert _{0}\leq n$ almost surely. Finally, $\hb(\lambda)$ is unique almost surely whenever the distribution of $X$ is absolutely continuous with respect to the Lebesgue measure on $\mathbb R^p$. 
\end{lemma}

\begin{proof}
All the asserted claims in this lemma can be found in the literature. Here
we give specific references for completeness. The fact that the function $%
(\varepsilon _{i})_{i=1}^{n}\mapsto (X_{i}^{\prime }\widehat{\beta }(\lambda
))_{i=1}^{n}$ is well-defined follows from Lemma 1 in \cite{T13}, which
shows that even if the solution $\widehat{\beta }(\lambda )$ of the
optimization problem \eqref{eq:
lasso} is not unique, $(X_{i}^{\prime }\widehat{\beta }(\lambda ))_{i=1}^{n}$
is the same across all solutions. The Lipschitz property then follows from
Proposition 2 in \cite{BT17}. Moreover, by discussion in Section 2.1 in \cite%
{T13}, there always exists a Lasso solution, say $\widehat{\beta }(\lambda )$,
taking the form in (10) of \cite{T13}, and such a solution satisfies $\Vert 
\widehat{\beta }(\lambda )\Vert _{0}\leq n$. Finally, the last claim follows from Lemma 4 in \cite{T13}.
\end{proof}

\begin{lemma}
\label{lem: 0.7} Suppose that Assumption \ref{as: heterogeneity}\ holds. Then for
all $\kappa \geq 1$, $n\geq e^{\kappa }$, and $\lambda >0$, we have 
\begin{equation}
\Pr \Big(\Big|\Vert \widehat{\beta }(\lambda )-\beta \Vert _{2,n}-{\mathrm{E}%
}[\Vert \widehat{\beta }-\beta \Vert _{2,n}\mid X_{1}^{n}]\Big|>t\Big)\leq
\left( \frac{C\kappa \log ^{r}n}{t^{2}n}\right) ^{\kappa /2}
\label{eq: prediction norm concentration 1}
\end{equation}%
for some constant $C>0$ depending only on $C_{1}$ and $r$.
\end{lemma}

\begin{proof}
Fix $\kappa \geq 1$, $n\geq e^{\kappa}$, and $\lambda > 0$. Also, let $\xi$
be a $N(0,1)$ random variable that is independent of the data and let $C$ be
a positive constant that depends only on $C_1$ and $r$ but whose value can change
from place to place. Then by Lemma \ref{lem: lasso continuity}, the function 
$(\varepsilon_i)_{i=1}^n\mapsto (X_i^{\prime }\widehat
\beta(\lambda))_{i=1}^n $ is Lipschitz-continuous with Lipschitz constant
one, and so is 
\begin{equation*}
(\varepsilon_i)_{i=1}^n\mapsto \left(\sum_{i=1}^n (X_i^{\prime }(\widehat
\beta(\lambda) - \beta))^2\right)^{1/2} = \sqrt n\| \widehat \beta - \beta
\|_{2,n}.
\end{equation*}
Therefore, applying Lemma \ref{lem: nongaussian concentration} with $u(x) =
(x\vee 0)^{\kappa}$ and using Markov's inequality and Assumption \ref{as:
heterogeneity} shows that for any $t>0$, 
\begin{align*}
&\Pr \Big( \| \widehat \beta(\lambda) - \beta \|_{2,n} - {\mathrm{E}}[ \|
\widehat \beta - \beta \|_{2,n}\mid X_1^n] > t \mid X_1^n\Big) \\
&\qquad \leq \Big( \frac{C_1\pi}{2 t \sqrt n} \Big)^{\kappa}{\mathrm{E}}\Big[%
\max_{1\leq i\leq n}( 1 + |e_i|^r )^{\kappa}\Big]{\mathrm{E}}[|\xi|^{\kappa}]
 \leq \Big( \frac{C}{t \sqrt n} \Big)^{\kappa}{\mathrm{E}}\Big[%
\max_{1\leq i\leq n}|e_i|^{r \kappa}\Big]{\mathrm{E}}[|\xi|^{\kappa}] \\
&\qquad \leq \Big( \frac{C}{t \sqrt n} \Big)^{\kappa}\Big({\mathrm{E}}\Big[%
\max_{1\leq i\leq n}|e_i|^{r\log n}\Big]\Big)^{\kappa/\log n}{\mathrm{E}}%
[|\xi|^{\kappa}] \\
&\qquad \leq \Big( \frac{C (r\log n)^{r/2}\sqrt{\kappa}}{t \sqrt
n} \Big)^{\kappa} = \Big( \frac{C\sqrt{\kappa \log^r n}}{t \sqrt n} \Big)%
^{\kappa} = \Big( \frac{C\kappa \log^r n}{t^2 n} \Big)^{\kappa/2}.
\end{align*}
This gives one side of the bound \eqref{eq: prediction norm concentration 1}%
. Since the other side follows similarly, the proof is complete.
\end{proof}

\begin{lemma}
\label{lem: 0.5} Suppose that Assumption \ref{as: heterogeneity}\ holds and let 
$Q^{-1}\colon \mathbb{R}^{p}\times \mathbb{R}\rightarrow \mathbb{R}$ be the
inverse of $Q\colon \mathbb{R}^{p}\times \mathbb{R}\rightarrow \mathbb{R}$
with respect to the second argument. Then for all $\lambda > 0$,
\begin{equation}
{\mathrm{E}}[\Vert \widehat{\beta }(\lambda)\Vert _{0}\mid
X_{1}^{n}]=\sum_{i=1}^{n}{\mathrm{E}}[\psi _{i}X_{i}^{\prime }(\widehat{%
\beta }(\lambda)-\beta )\mid X_{1}^{n}],  \label{eq: ZHT}
\end{equation}%
where 
\begin{equation*}
\psi _{i}=\frac{e_{i}}{Q_{2}(X_{i},e_{i})}+\frac{Q_{22}(X_{i},e_{i})}{%
Q_{2}(X_{i},e_{i})^{2}}\text{ and }e_{i}=Q^{-1}(X_{i},\varepsilon
_{i})
\end{equation*}
for all $i = 1,\dots,n$.
In addition,
\begin{align*}
&\Ep\left[\left( \|\widehat\beta(\lambda)\|_0 - \sum_{i=1}^n \psi_i X_i'(\widehat\beta - \beta) \right)^2 \mid X_1^n\right] \\
&\quad = \sum_{i=1}^n \Ep\left[\gamma_i ( X_i'(\widehat\beta(\lambda) - \beta) )^2 \mid X_1^n\right] + \Ep[ \|\widehat\beta(\lambda)\|_0 \mid X_1^n],
\end{align*}
where
$$
\gamma_i = \frac{1}{Q_2(X_i,e_i)^2} - \frac{e_i Q_{22}(X_i,e_i)}{Q_2(X_i,e_i)^3} + \frac{Q_{222}(X_i,e_i)}{Q_2(X_i,e_i)^3} - \frac{2Q_{22}(X_i,e_i)^2}{Q_2(X_i,e_i)^4}
$$
for all $i = 1,\dots,n$.
Moreover,
%\begin{align*}
%\mathrm{Var}\left( \sum_{i=1}^n \psi_i X_i'(\widehat\beta(\lambda) - \beta) \mid X_1^n \right)
% & \leq C\Big((n\log^2 n)\sqrt{\Ep[ \|\widehat\beta(\lambda) - \beta \|_{2,n}^4 \mid X_1^n]} \\
% &\quad + \Ep[ \|\widehat\beta(\lambda)\|_0 \vee 1 \mid X_1^n](\log p)(\log^r n)\Big),
%\end{align*}
\begin{align*}
\mathrm{Var}\left( \sum_{i=1}^n \psi_i X_i'(\widehat\beta(\lambda) - \beta) \mid X_1^n \right)
 & \leq 2\sum_{i=1}^n\Ep\left[\Big(\gamma_i Q_2(X_i,e_i)X_i'(\hb - \beta)\Big)^2 \mid X_1^n\right]\\
 & \quad + C\Ep\Big[ \|\widehat\beta(\lambda)\|_0 + 1 \mid X_1^n\Big](\log p)(\log^r n),
\end{align*}
where $C>0$ is a constant depending only on $c_1$, $C_1$, and $r$.
\end{lemma}

\begin{remark}
Here, the inverse $Q^{-1}$ exists because by Assumption \ref{as: heterogeneity}, $Q$ is strictly increasing and continuous with respect to
its second argument.\qed
\end{remark}

\begin{proof}
This lemma extends some of the results in \cite{TT12} and \cite{BZ18} to the non-Gaussian case. All arguments in the proof are conditional on $X_1,\dots,X_n$ but we drop the conditioning sign for brevity of notation. Also, we use $C$ to denote a positive constant that depends only on $c_1$, $C_1$ and $r$ but whose value can change from place to place.

Fix $\lambda >0$ and denote $\widehat\beta = \widehat\beta(\lambda)$ and $\widehat T = \{j\in\{1,\dots,p\}\colon \widehat\beta_j \neq 0\}$. For all $i = 1,\dots,n$ we will use $X_{i\widehat T}$ to denote the sub-vector of $X_i$ in $\mathbb R^{|\widehat T|}$ corresponding to indices in $\widehat T$. By results in \cite{TT12}, we then have
\begin{equation}\label{eq: TT}
\frac{\partial X_i'(\widehat\beta - \beta)}{\partial \varepsilon_j} = X_{i\widehat T}'\left(\sum_{l=1}^n X_{l\widehat T}X_{l\widehat T}'\right)^{-1}X_{j\widehat T},\quad i,j=1,\dots,n;
\end{equation}
see, in particular, the proof of Theorem 1 there. Taking the sum over $i = j = 1,\dots,n$ and applying the trace operator on the right-hand side of this identity gives
\begin{equation}
\Vert \widehat{\beta }\Vert _{0} = |\widehat T| =\sum_{i=1}^{n} \frac{\partial (X_{i}^{\prime }(%
\widehat{\beta }-\beta ))}{\partial \varepsilon _{i}}. \label{eq: TT identity}
\end{equation}%
Also, for all $%
i=1,\dots ,n$, under Assumption \ref{as: heterogeneity} (and conditional on $X_{1}^{n}$), the random variable $\varepsilon _{i}$ is absolutely continuous
with respect to Lebesgue measure on $\mathbb{R}$ with continuously
differentiable pdf $\chi _{i}$ defined by 
\begin{equation*}
\chi _{i}(Q(X_{i},e))=\frac{\phi (e)}{Q_{2}(X_{i},e)},\quad e\in \mathbb{R},
\end{equation*}%
where $\phi $ is the pdf of the $N(0,1)$ distribution. Taking the derivative over $e$ here gives
\begin{equation*}
\chi _{i}^{\prime }(Q(X_{i},e))Q_{2}(X_{i},e)=-\frac{e\phi (e)}{%
Q_{2}(X_{i},e)}-\frac{\phi (e)Q_{22}(X_{i},e)}{Q_{2}(X_{i},e)^{2}},\quad
e\in \mathbb{R},
\end{equation*}%
and so 
\begin{equation*}
\frac{\chi _{i}^{\prime }(\varepsilon _{i})}{\chi _{i}(\varepsilon _{i})}=%
\frac{\chi _{i}^{\prime }(Q(X_{i},e_{i}))}{\chi _{i}(Q(X_{i},e_{i}))}=-\frac{%
e_{i}}{Q_{2}(X_{i},e_{i})}-\frac{Q_{22}(X_{i},e_{i})}{%
Q_{2}(X_{i},e_{i})^{2}}=-\psi _{i}.
\end{equation*}%
Therefore, by Lemma \ref{lem: integration by parts stein}, whose application is justified by Assumption \ref{as: heterogeneity} and Lemma \ref{lem: lasso continuity},
\begin{equation}
{\mathrm{E}}\left[ \frac{\partial (X_{i}^{\prime }(\widehat{\beta }-\beta ))}{\partial \varepsilon _{i}}\right] ={\mathrm{E}}%
[\psi _{i}X_{i}^{\prime }(\widehat{\beta }-\beta )],\quad i=1,\dots ,n.  \label{eq: generalized stein identity}
\end{equation}%
Combining \eqref{eq: TT identity} and \eqref{eq: generalized stein identity}
gives the first asserted claim.

To prove the second asserted claim, we proceed along the lines in the proof of Theorem 1.1 in \cite{BZ18}. Specifically, let $f_1,\dots,f_n$ be twice continuously differentiable functions mapping $\mathbb R^n$ to $\mathbb R$ with bounded first and second derivatives. Also, let $\bar\varepsilon = (\varepsilon_1,\dots,\varepsilon_n)'$. Then, it follows from Lemma \ref{lem: integration by parts stein} that
\begin{align*}
&\Ep\left[ \psi_i f_i(\bar\varepsilon)\left( \sum_{j=1}^n\psi_j f_j(\bar\varepsilon) - \sum_{l=1}^n\frac{\partial f_l(\bar\varepsilon)}{\partial \varepsilon_l} \right) \right] \\
&\quad = \Ep\left[ \frac{\partial f_i(\bar\varepsilon)}{\partial \varepsilon_i}\left( \sum_{j=1}^n\psi_j f_j(\bar\varepsilon) - \sum_{l=1}^n\frac{\partial f_l(\bar\varepsilon)}{\partial \varepsilon_l} \right) \right]\\
&\quad\quad + \Ep\left[ f_i(\bar\varepsilon)\left( \gamma_i f_i(\bar\varepsilon) + \sum_{j=1}^n \psi_j \frac{\partial f_j(\bar\varepsilon)}{\partial \varepsilon_i} - \sum_{l=1}^n \frac{\partial^2 f_l(\bar\varepsilon)}{\partial \varepsilon_l\partial \varepsilon_i} \right) \right]
\end{align*}
for all $i=1,\dots,n$ and, in addition,
$$
\Ep\left[ \psi_j f_i(\bar\varepsilon)\frac{\partial f_j(\bar\varepsilon)}{\partial \varepsilon_i} \right]
 = \Ep\left[ \frac{\partial f_i(\bar\varepsilon)}{\partial \varepsilon_j} \frac{\partial f_j(\bar\varepsilon)}{\partial \varepsilon_i} + f_i(\bar\varepsilon) \frac{\partial^2 f_j(\bar\varepsilon)}{\partial\varepsilon_i\partial\varepsilon_j} \right]
$$
for all $j=1,\dots,n$. Combining these results, rearranging the terms, and taking the sum over $i = 1,\dots,n$, we obtain
$$
\Ep\left[\left( \sum_{i=1}^n \psi_i f_i(\bar\varepsilon) - \sum_{i=1}^n \frac{\partial f_i(\bar\varepsilon)}{\partial \varepsilon_i} \right)^2\right] = \sum_{i=1}^n \Ep[\gamma_i f_i(\bar\varepsilon)^2] +\sum_{i,j=1}^n\Ep\left[ \frac{\partial f_i(\bar\varepsilon)}{\partial \varepsilon_j}\frac{\partial f_j(\bar\varepsilon)}{\partial \varepsilon_i} \right],
$$
and since all second-order derivatives cancell out, it follows from a convolution argument that the same identity holds for any Lipschitz functions $f_1,\dots,f_n$; see Appendix A of \cite{BZ18} for details. We now substitute $f_i(\bar\varepsilon) = X_i'(\widehat\beta - \beta)$ for all $i=1,\dots,n$ in this identity and note that
$$
\sum_{i,j=1}^n \frac{\partial f_i(\bar\varepsilon)}{\partial \varepsilon_j}\frac{\partial f_j(\bar\varepsilon)}{\partial \varepsilon_i} = \|\widehat\beta\|_0
$$
by \eqref{eq: TT} in this case. This gives the second asserted claim.

To prove the third asserted claim, we have by the Gaussian Poincare inequality, Theorem 3.20 in \cite{BLM13} that
\begin{align*}
&\mathrm{Var}\left( \sum_{i=1}^n \psi_iX_i'(\hb - \beta) \right) \leq \sum_{j=1}^n\Ep\left[\left( \frac{\partial}{\partial e_j}\sum_{i=1}^n \psi_i X_i'(\hb - \beta) \right)^2\right]\\
&\quad \leq 2\sum_{j=1}^n\Ep\left[\left(\frac{\partial\psi_j}{\partial e_j}X_j'(\hb - \beta)\right)^2\right]
+ 2\sum_{j=1}^n\Ep\left[\left(\sum_{i=1}^n\psi_i\frac{\partial X_i'(\hb - \beta)}{\partial e_j}\right)^2\right].
\end{align*}
Here, the first term on the right-hand side is equal to
\begin{align*}
&2\sum_{j=1}^n\Ep\left[\left( \gamma_j Q_2(X_j,e_j)X_j'(\hb - \beta) \right)^2\right].
%&\quad \leq 2\Ep\left[ \max_{1\leq j\leq n}\Big(\gamma_jQ_2(X_j,e_j)\Big)^2 \sum_{j=1}^n (X_j'(\hb - \beta))^2 \right]\\
%&\quad \leq C(n\log^2 n)\sqrt{\Ep\left[\|\hb - \beta\|_{2,n}^4\right]},
\end{align*}
%where the second inequality follows from the H\"{o}lder's inequality and Assumption \ref{as: heterogeneity}. 
Also, by \eqref{eq: TT}, the second term is equal to
\begin{align*}
&2\sum_{j=1}^n \Ep\left[ \left( \sum_{i=1}^n\psi_i Q_2(X_j,e_j) X_{i\widehat T}'\left(\sum_{l=1}^n X_{l\widehat T}X_{l\widehat T}'\right)^{-1}X_{j\widehat T} \right)^2 \right]\\
&\quad = 2\Ep\left[ \max_{1\leq j\leq p} Q_2(X_j,e_j)^2\sum_{j=1}^n\left( \sum_{i=1}^n \psi_i X_{i\widehat T}'\left(\sum_{l=1}^n X_{l\widehat T}X_{l\widehat T}'\right)^{-1}X_{j\widehat T} \right)^2 \right]\\
&\quad = 2\Ep\left[ \max_{1\leq j\leq p}Q_2(X_j,e_j)^2 \sum_{i=1}^n \psi_iX_{i\widehat T}'\left(\sum_{l=1}^n X_{l\widehat T}X_{l\widehat T}'\right)^{-1}\sum_{j=1}^n X_{j\widehat T}\psi_j \right].
\end{align*} 
Next, observe that
$$
\sum_{i=1}^n \psi_iX_{i\widehat T}'\left(\sum_{l=1}^n X_{l\widehat T}X_{l\widehat T}'\right)^{-1}\sum_{j=1}^n X_{j\widehat T}\psi_j
$$
is equal to $\| P_{\widehat T}\bar\psi \|_2^2$, where $\bar\psi = (\psi_1,\dots,\psi_n)'$ and $P_{\widehat T}$ is the matrix projecting on $(X_{1\widehat T},\dots,X_{n\widehat T})'$. In turn, we can bound $\Ep[\| P_{\widehat T}\bar\psi \|_2^2]$ using arguments from the proof of Theorem 4.3 in \cite{BZ18}. In particular, for any $M\subset\{1,\dots,p\}$, letting $P_M$ denote the matrix projecting on $(X_{1M},\dots,X_{nM})'$, we have $\Ep[\|P_M \bar\psi\|_2^2] \leq C|M|$, and so, by Assumption \ref{as: heterogeneity} and the Hanson-Wright inequality, Theorem 1.1 in \cite{RV13},
$$
P\Big( \| P_M\bar\psi \|_2^2 > C(|M| + x) \Big)\leq e^{-x}
$$
for all $x>0$. Thus, applying the union bound twice,
$$
\Pr\left(\max_{M\subset\{1,\dots,p\}} \left( \| P_M\bar\psi \|_2^2 - C\left( |M| + \log{p\choose|M|} + \log p + x \right) \right) > 0 \right) \leq e^{-x},
$$
and so
$$
\Pr\left(\max_{M\subset\{1,\dots,p\}} \left( \| P_M\bar\psi \|_2^2 - C\Big( (|M| + 1)\log p + x \Big) \right) > 0 \right) \leq e^{-x}.
$$
By Fubini's theorem and simple calculations, we then have
\begin{equation}\label{eq: BZ projection 1}
\Ep\left[ \| P_{\widehat T}\bar\psi \|_2^2 \right] \leq C\Ep\left[\| \hb \|_0 + 1\right]\log p.
\end{equation}
Also,
\begin{equation}\label{eq: BZ projection 2}
\Ep\left[ \| P_{\widehat T}\bar\psi \|_2^4 \right] \leq  \Ep\left[\| \bar\psi \|_2^4\right]  \leq Cn^2.
\end{equation}
Hence, for a sufficiently large constant $A$ that can be chosen to depend on $c_1$, $C_1$ and $r$ only,
\begin{align*}
&\Ep\left[ \max_{1\leq j\leq p}Q_2(X_j,e_j)^2 \sum_{i=1}^n \psi_iX_{i\widehat T}'\left(\sum_{l=1}^n X_{l\widehat T}X_{l\widehat T}'\right)^{-1}\sum_{j=1}^n X_{j\widehat T}\psi_j \right]\\
&\quad = \Ep\left[\max_{1\leq j\leq p}Q_2(X_j,e_j)^2\| P_{\widehat T}\bar\psi \|_2^21\left\{ \max_{1\leq j\leq p}Q_2(X_j,e_j)^2 \leq A\log^r n \right\}\right]\\
&\quad\quad + \Ep\left[\max_{1\leq j\leq p}Q_2(X_j,e_j)^2\| P_{\widehat T}\bar\psi \|_2^21\left\{ \max_{1\leq j\leq p}Q_2(X_j,e_j)^2 > A\log^r n \right\}\right].
\end{align*}
Here, by \eqref{eq: BZ projection 1}, the first term on the right-hand side is bounded from above by $C\Ep[\|\hb\|_0 + 1](\log p)(\log^r n)$ and by \eqref{eq: BZ projection 2}, Assumption \ref{as: heterogeneity}, and H\"{o}lder's inequality, the second term is bounded from above by $C$ since $A$ is large enough. Combining all presented inequalities together gives the third asserted claim and completes the proof of the lemma.
\end{proof}

\subsection{Proof of Theorem \ref{thm: sparsity bound via l2 norm}}
All arguments in this proof are conditional on $X_1,\dots,X_n$ but we drop the conditioning sign for brevity of notation. Throughout the proof, we will assume that
\begin{equation}\label{eq: sparse eigenvalue sparsity theorem 1}
\sup_{\delta\in \mathcal S^p(s)}\|\delta\|_{2,n} \leq \bar C.
\end{equation}
Also, we use $C$ to denote a positive constant that depends only on $c_1$, $C_1$, $\bar C$, and $r$ but whose value can change from place to place.

Fix $\lambda >0$ and denote $\widehat\beta = \widehat\beta(\lambda)$, $\widehat s = \|\hb\|_0$, and $R_n = \Ep[\|\hb - \beta\|_{2,n}]$. We start with some preliminary inequalities. First, by H\"{o}lder's inequality and Assumption \ref{as: heterogeneity},
\begin{align}
&\sum_{i=1}^n \Ep\Big[\gamma_i(X_i'(\hb - \beta))^2\Big] \nonumber\\
&\quad \leq n\Ep\Big[ \|\hb - \beta\|_{2,n}^2\times\max_{1\leq i\leq n}|\gamma_i| \Big] 
\leq C(n\log n)\sqrt{\Ep[\|\hb - \beta\|_{2,n}^4]}\label{eq: sparsity proof 1}
\end{align}
and, similarly,
\begin{equation}\label{eq: sparsity proof 2}
\sum_{i=1}^n\Ep\left[\Big(\gamma_i Q_2(X_i,e_i)X_i'(\hb - \beta)\Big)^2 \right] \leq C(n\log^2n)\sqrt{\Ep[\|\hb - \beta\|_{2,n}^4]}.
\end{equation}
Second, by the triangle inequality and Fubini's theorem, 
\begin{align}
{\mathrm{E}}[\Vert \widehat{\beta }-\beta \Vert _{2,n}^{4}]&
\leq C\left(R_n^{4}+{\mathrm{E}}\left[ (\Vert \widehat{\beta }%
-\beta \Vert _{2,n}-R_{n})^{4}\right]\right) \nonumber\\
& =C\left(R_{n}^{4}+\int_{0}^{\infty }\Pr \left( |\Vert \widehat{\beta }%
-\beta \Vert _{2,n}-R_{n}|>t^{1/4}\right) dt \right)\nonumber\\
& \leq C\left( R_n^4  + \left( \frac{\log ^{r}n}{n}\right) ^{2}\right),\label{eq: sparsity proof 3}
\end{align}
where the last line follows from Lemma \ref{lem: 0.7} applied with $\kappa
=5 $ (for example). Third,
\begin{equation}\label{eq: sparsity proof 4}
\Pr\left(R_n > \| \hb - \beta \|_{2,n} + \sqrt{\frac{C\log^{r+1}n}{n}}\right) \leq \frac{1}{n} 
\end{equation}
by Lemma \ref{lem: 0.7} applied with $\kappa =\log n $.
%\begin{align} 
%\Ep\Big[\|\widehat\beta - \beta\|_{2,n}^2\Big]
%&\leq 2 \Ep\Big[\Big( \|\widehat\beta - \beta\|_{2,n} - R_n \Big)^2\Big] + 2R_n^2\nonumber\\
%&\leq 6\Ep\Big[\Big( \|\widehat\beta - \beta\|_{2,n} - R_n \Big)^2\Big] + 4\|\widehat\beta - \beta\|_{2,n}^2\nonumber\\
%&\leq C\Big( \|\widehat\beta - \beta\|_{2,n}^2 + \log(p n)/n \Big)
%\end{align}
%with probability at least $1 - 1/(p n)$ by Borell's inequality (see Theorem 2.1.12 in \cite{T12}) and Lemma \ref{lem: lasso continuity}. 
Fourth, by Lemma 9 of \cite{BC11} and \eqref{eq: sparse eigenvalue sparsity theorem 1},
\begin{align}
\|\widehat\beta - \beta\|_{2,n}^2 
& \leq \|\widehat\beta - \beta\|_2^2\times \sup_{\delta\in\mathcal S^p(\widehat s + s)}\|\delta\|_{2,n}^2\nonumber\\
& \leq\|\widehat\beta - \beta\|_2^2\times \frac{2(\widehat s + s)}{s}\sup_{\delta\in\mathcal S^p(s)}\|\delta\|_{2,n}^2 \leq \frac{C(\widehat s + s)}{s}\|\widehat\beta - \beta\|_2^2.\label{eq: sparsity proof 5}
\end{align}
We now prove the theorem with the help of these bounds. Denote
$$
V_1 = \mathrm{Var}\left( \widehat s - \sum_{i=1}^n\psi_i X_i'(\hb - \beta) \right)\text{ and }V_2 = \mathrm{Var}\left(\sum_{i=1}^n\psi_i X_i'(\hb - \beta)\right).
$$
Then for any $\bar t > 0$, with probability at least $1 - 2/\bar t^2$, by Chebyshev's inequality and Lemma \ref{lem: 0.5},
\begin{align}
\widehat s & \leq \sum_{i=1}^n\psi_i X_i'(\hb - \beta) + \bar t\sqrt{V_1}\nonumber\\
& = \sum_{i=1}^n\psi_i X_i'(\hb - \beta) + \Ep\left[1 + \sum_{i=1}^n\psi_i X_i'(\hb - \beta)\right] - \Ep[1 + \widehat s] + \bar t\sqrt{V_1}\nonumber\\
& \leq 1 + 2\sum_{i=1}^n\psi_i X_i'(\hb - \beta) - \Ep[1 + \widehat s] + \bar t(\sqrt{V_1} + \sqrt{V_2}).\label{eq: sparsity proof 0}
\end{align}
Here, $\bar t(\sqrt{V_1} + \sqrt{V_2})$ is bounded from above by
\begin{align*}
 & C\bar t\Big( \Ep[1 + \widehat s](\log p)(\log^r n) \Big)^{1/2} + C\bar t\sqrt{n}\log n\left( R_n + \sqrt{\frac{\log^r n}{n}} \right)\\
&\quad \leq \Ep[1 + \widehat s] + C\bar t^2(\log p)(\log^r n) + C\bar t\sqrt{n}\log n\left( R_n + \sqrt{\frac{\log^{r} n}{n}} \right)
\end{align*}
by Lemma \ref{lem: 0.5} and inequalities \eqref{eq: sparsity proof 1}, \eqref{eq: sparsity proof 2}, and \eqref{eq: sparsity proof 3}. Also, with probability at least $1 - 1/n$,
$$
R_n \leq C\left(\sqrt{\frac{\widehat s + s}{s}}\|\hb - \beta\|_2 + \sqrt{\frac{\log^{r+1}n}{n}}\right)
$$
by \eqref{eq: sparsity proof 4} and \eqref{eq: sparsity proof 5}. In addition,
$$
\sum_{i=1}^n \psi_i X_i'(\widehat\beta - \beta) \leq \left\| \sum_{i=1}^n\psi_i X_i \right\|_{\infty} \times \|\widehat\beta - \beta\|_1,
$$
where $\|\widehat\beta - \beta\|_1 \leq (\widehat s + s)^{1/2}\times\|\widehat\beta - \beta\|_2$ and with probability at least $1 - 1/n$,
$$
\left\| \sum_{i=1}^n\psi_i X_i \right\|_{\infty} \leq C\sqrt{\log(p n)}\times\max_{1\leq j\leq p}\left( \sum_{i=1}^n X_{i j}^2 \right)^{1/2} \leq C\sqrt{n\log(p n)},
$$
with the first inequality following from Assumption \ref{as: heterogeneity} and the union bound and the second from \eqref{eq: sparse eigenvalue sparsity theorem 1}. Substituting all these bounds into \eqref{eq: sparsity proof 0} and using $\bar t = (ts\log(p n))^{1/2}$ with $t \geq 1$ gives
$$
\widehat s \leq Cts(\log^r n)\log^2(p n) + C\sqrt{t(\widehat s + s)n(\log^2 n)\log(p n)}\|\hb - \beta\|_{2}
$$
with probability at least $1 - 2/(t s \log(p n)) - 2/n$. Solving this inequality for $\widehat s$ gives the asserted claim and completes the proof of the theorem.
\qed

\subsection{Proof of Theorem \protect \ref{lem: sparsity bound low
dimensional}\label{sec: proof of theorem on sparsity}}
All arguments in this proof are conditional on $X_1,\dots,X_n$ but we drop the conditioning sign for brevity of notation. Throughout the proof, we will assume that \eqref{eq: sparse eigenvalue sparsity theorem 2} holds. Also, we use $C$ to denote a positive constant that depends only on $c_1$, $C_1$, $\bar c$,  $\bar C$, and $r$ but whose value can change from place to place.

Fix $\lambda > 0$ and denote $\widehat{\beta }=\widehat{\beta }(\lambda )$, $\widehat{s}=\Vert \widehat{\beta }\Vert _{0}$, $J_n = J_n(\lambda)$, and $R_n = R_n(\lambda)$. Then by
Lemma \ref{lem: 0.5}, 
\begin{equation}\label{eq: thm 52 beginning}
{\mathrm{E}}[\widehat{s}]=\sum_{i=1}^{n}{\mathrm{E}}[\psi
_{i}X_{i}^{\prime }(\widehat{\beta }-\beta )]=\mathcal{I}_{1}+%
\mathcal{I}_{2},
\end{equation}
where 
\begin{align}
& \mathcal{I}_{1}=\sum_{i=1}^{n}{\mathrm{E}}\left[  \psi
_{i}X_{i}^{\prime }(\widehat{\beta }-\beta )1\left \{ \bar c \Vert \widehat{\beta }%
-\beta \Vert _{2}\leq \Vert \widehat{\beta }-\beta \Vert _{2,n}\right \}
\right] ,  \label{eq: sparsity i1} \\
& \mathcal{I}_{2}=\sum_{i=1}^{n}{\mathrm{E}}\left[ \psi
_{i}X_{i}^{\prime }(\widehat{\beta }-\beta )1\left \{ \bar c\Vert \widehat{\beta }%
-\beta \Vert _{2}>\Vert \widehat{\beta }-\beta \Vert _{2,n}\right \} \right].  \label{eq: sparsity i2}
\end{align}%
We bound $\mathcal{I}_{1}$ and $\mathcal{I}_{2}$
in turn. To bound $\mathcal{I}_{1}$, note that as in \eqref{eq: sparsity proof 3} of the proof of Theorem \ref{thm: sparsity bound via l2 norm},
\begin{equation}\label{eq: rn4 bound second sparsity theorem}
{\mathrm{E}}[\Vert \widehat{\beta }-\beta \Vert _{2,n}^{4}] \leq C\left( R_n^4  + \left( \frac{\log ^{r}n}{n}\right) ^{2}\right).
\end{equation}
Also, by Assumption \ref{as: heterogeneity} and \eqref{eq: sparse eigenvalue sparsity theorem 2},
\begin{align*}
{\mathrm{E}}\left[ \left \Vert \sum_{i=1}^{n}\psi _{i}X_{i}\right
\Vert _{\infty }^{4}\right] \leq C n^2\log^2 p.
\end{align*}%
Therefore, 
\begin{align*}
\mathcal{I}_{1}& \leq {\mathrm{E}}\left[ \left \Vert {\sum
\nolimits_{i=1}^{n}}\psi _{i}X_{i}\right \Vert _{\infty }\Vert \widehat{%
\beta }-\beta \Vert _{1}1\left \{\bar c \Vert \widehat{\beta }-\beta \Vert
_{2}\leq \Vert \widehat{\beta }-\beta \Vert _{2,n}\right \} \right] \\
& \leq {\mathrm{E}}\left[ \left \Vert {\sum \nolimits_{i=1}^{n}}\psi
_{i}X_{i}\right \Vert _{\infty }\Vert \widehat{\beta }-\beta \Vert _{2}(%
\widehat{s}+s)^{1/2}1\left \{ \bar c\Vert \widehat{\beta }-\beta \Vert _{2}\leq
\Vert \widehat{\beta }-\beta \Vert _{2,n}\right \} \right] \\
& \leq C {\mathrm{E}}\left[ \left \Vert {\sum \nolimits_{i=1}^{n}}\psi
_{i}X_{i}\right \Vert _{\infty }\Vert \widehat{\beta }-\beta \Vert _{2,n}(%
\widehat{s}+s)^{1/2}1\left \{ \bar c\Vert \widehat{\beta }-\beta \Vert _{2}\leq
\Vert \widehat{\beta }-\beta \Vert _{2,n}\right \} \right] \\
& \leq  C\left( {\mathrm{E}}\left[ \left \Vert {\sum \nolimits_{i=1}^{n}}\psi
_{i}X_{i}\right \Vert _{\infty }^{2}\Vert \widehat{\beta }-\beta \Vert
_{2,n}^{2}\right] {\mathrm{E}}\left[ \widehat{s}+s\right] \right) ^{1/2},
\end{align*}%
where the last line follows from H\"{o}lder's inequality. In turn, 
\begin{align*}
& \left( {\mathrm{E}}\left[ \left \Vert {\sum \nolimits_{i=1}^{n}}\psi
_{i}X_{i}\right \Vert _{\infty }^{2}\Vert \widehat{\beta }-\beta \Vert
_{2,n}^{2}\right] \right) ^{1/2} \\
& \qquad \leq \left( {\mathrm{E}}\left[ \left \Vert {\sum \nolimits_{i=1}^{n}%
}\psi _{i}X_{i}\right \Vert _{\infty }^{4}\right] {\mathrm{E}}%
\left[ \Vert \widehat{\beta }-\beta \Vert _{2,n}^{4}\right]
\right) ^{1/4} \\
& \qquad \leq C \sqrt{n\log p}\left( {\mathrm{E}}\left[ \Vert \widehat{%
\beta }-\beta \Vert _{2,n}^{4}\right] \right) ^{1/4}\leq C
\sqrt{n\log p}\left( R_{n}+\sqrt{\frac{\log ^{r}n}{n}}\right) .
\end{align*}%
Thus, 
\begin{equation}
\mathcal{I}_{1}\leq C \sqrt{n\log p}\left( R_{n}+\sqrt{\frac{%
\log ^{r}n}{n}}\right) ({\mathrm{E}}[\widehat{s}+s])^{1/2}.
\label{eq: i1 bound}
\end{equation}%
To bound $\mathcal{I}_{2}$, denote 
\begin{equation*}
A_{1}=\sqrt{\sum_{i=1}^{n}\psi _{i}^{2}}\quad \text{and}\quad A_{2}=\sqrt{%
\sum_{i=1}^{n}(X_{i}^{\prime }(\widehat{\beta }-\beta ))^{2}}=\sqrt{n}\Vert 
\widehat{\beta }-\beta \Vert _{2,n}
\end{equation*}%
and observe that by H\"{o}lder's inequality, 
\begin{equation*}
\mathcal{I}_{2}\leq {\mathrm{E}}\left[ A_{1}A_{2}1\left \{\bar c \Vert \widehat{%
\beta }-\beta \Vert _{2}>\Vert \widehat{\beta }-\beta \Vert _{2,n}\right \}
\right] \leq \mathcal{I}_{2,1}+\mathcal{I}_{2,2},
\end{equation*}%
where 
\begin{align*}
& \mathcal{I}_{2,1}={\mathrm{E}}\left[ A_{1}A_{2}1\left \{ A_{1}A_{2}>\mathcal{C}%
\left( nR_{n}+\sqrt{n\log ^{r+1}n}\right) \right \} \right] , \\
& \mathcal{I}_{2,2}=\mathcal{C}\left( nR_{n}+\sqrt{n\log ^{r+1}n}%
\right) \Pr \left(\bar c \Vert \widehat{\beta }-\beta \Vert _{2}>\Vert \widehat{%
\beta }-\beta \Vert _{2,n}\right) ,
\end{align*}%
for some constant $\mathcal{C}$ to be chosen later. To bound $\mathcal{I}_{2,1}$%
, note that 
\begin{equation*}
\Pr (A_{1}>\sqrt{\mathcal{C}n})\leq 1/n
\end{equation*}%
by Chebyshev's inequality and Assumption \ref{as: heterogeneity} if $\mathcal{C}$
is large enough. Also, by Lemma \ref{lem: 0.7} applied with $\kappa =\log n$, 
\begin{equation*}
\Pr \left( A_{2}/\sqrt{n}>R_{n}+\sqrt{\mathcal{C}\log ^{r+1}n/n}\right) \leq 1/n
\end{equation*}%
if $\mathcal{C}$ is large enough. Hence, if we set $\mathcal{C}$ in the definition
of $\mathcal{I}_{2,1}$ and $\mathcal{I}_{2,2}$ large enough (note that $\mathcal{%
C}$ can be chosen to depend only on $c_1$, $C_{1}$, and $r$), it follows that 
\begin{align*}
& \Pr \left( A_{1}A_{2}>\mathcal{C}\left( nR_{n}+\sqrt{n\log ^{r+1}n}%
\right) \right)\\
&\quad \leq \Pr (A_{1}>\sqrt{\mathcal{C}n})+\Pr \left( A_{2}/%
\sqrt{n}>R_{n}+\sqrt{\mathcal{C}\log ^{r+1}n/n}\right)\leq 2/n,
\end{align*}
and so $\mathcal{I}_{2,1}$ is bounded from above by 
\begin{align*}
& \left( {\mathrm{E}}[A_{1}^{2}A_{2}^{2}]\right) ^{1/2}\left(
\Pr \left( A_{1}A_{2}>\mathcal{C}\left( nR_{n}+\sqrt{n\log ^{r+1}n}%
\right) \right) \right) ^{1/2} \\
& \quad  \leq C n\left( R_{n}+\sqrt{\log ^{r}n/n}\right) /\sqrt{n}%
\leq C(\sqrt{n}R_{n}+\sqrt{\log ^{r}n}),
\end{align*}%
where the first inequality follows from H\"{o}lder's inequality, Assumption \ref{as: heterogeneity}, and \eqref{eq: rn4 bound second sparsity theorem}.
Also, by \eqref{eq: sparse eigenvalue sparsity theorem 2} and Markov's inequality, 
\begin{equation*}
\Pr \left( \bar c\Vert \widehat{\beta }-\beta \Vert _{2}>\Vert \widehat{\beta }%
-\beta \Vert _{2,n}\right) \leq \Pr \left( \widehat{s}%
+s>J_{n}\right) \leq \frac{{\mathrm{E}}[\widehat{s}+s]}{J_{n}},
\end{equation*}%
so that 
\begin{equation*}
\mathcal{I}_{2,2}\leq \frac{C(nR_{n}+\sqrt{n\log ^{r+1}n})}{J_{n}}%
{\mathrm{E}}[\widehat{s}+s],
\end{equation*}%
and so%
\begin{equation*}
\mathcal{I}_{2,2}\leq 3^{-1}{\mathrm{E}}[\widehat{s}+s]
\end{equation*}%
for all $n\geq n_0$ depending only on $c_1$, $C_1$, $\bar c$, $\bar C$, and $r$ by the definition of $J_n$.

Combining all inequalities, it follows that for all $n\geq n_0$,
\begin{align}
{\mathrm{E}}[\widehat{s}]& \leq C\sqrt{n\log p}\left(
R_{n}+\sqrt{\frac{\log ^{r}n}{n}}\right) ({\mathrm{E}}[\widehat{s}%
+s])^{1/2} \nonumber\\
& \quad +C\left( \sqrt{n}R_{n}+\sqrt{\log ^{r}n}\right) +3^{-1}{%
\mathrm{E}}[\widehat{s}+s],\label{eq: thm 52 end}
\end{align}
and so 
\begin{equation*}
{\mathrm{E}}[\Vert \widehat{\beta }\Vert _{0}]={\mathrm{E}}[\widehat{s}]\leq s+C(\log p)(nR_{n}^{2}+\log ^{r}n).
\end{equation*}%
This gives the asserted claim for all $n\geq n_0$ and since the asserted claim for $n<n_0$ is trivial, the proof is complete. \qed

\section{Technical Lemmas\label{sec: technical lemmas}}

\begin{lemma}
\label{lem: maximal inequality} Let $X_{1},\dots ,X_{n}$ be independent
centered random vectors in $\mathbb{R}^{p}$ with $p\geq 2$. Define $%
Z=\left
\Vert \sum_{i=1}^{n}X_{i}\right \Vert _{\infty }$, $M=\max_{1\leq
i\leq n}\left \Vert X_{i}\right \Vert _{\infty }$, and $\sigma
^{2}=\max_{1\leq j\leq p}\sum_{i=1}^{n}{\mathrm{E}}[X_{ij}^{2}]$. Then 
\begin{equation*}
{\mathrm{E}}[Z]\leq K\left( \sigma \sqrt{\log p}+\sqrt{{\mathrm{E}}[M^{2}]}%
\log p\right)
\end{equation*}%
where $K>0$ is a universal constant.
\end{lemma}

\begin{proof}
See Lemma E.1 in \cite{CCK14}.
\end{proof}

\begin{lemma}
\label{lem: maximal inequality 2} Consider the setting of Lemma \ref{lem:
maximal inequality}. For every $\eta >0$, $t>0$, and $q\geq 1$, we have 
\begin{equation*}
P\left( Z\geq (1+\eta ){\mathrm{E}}[Z]+t\right) \leq \exp (-t^{2}/(3\sigma
^{2}))+K{\mathrm{E}}[M^{q}]/t^{q}
\end{equation*}%
where the constant $K>0$ depends only on $\eta $ and $q$.
\end{lemma}

\begin{proof}
See Lemma E.2 in \cite{CCK14}.
\end{proof}

\begin{remark}
In Lemmas \ref{lem: maximal inequality} and \ref{lem: maximal inequality 2}, if, in addition, we assume that $X_1,\dots,X_n$ are Gaussian, then $\Ep[Z] \leq \sigma\sqrt{2\log p}$ by Lemma A.3.1 in \cite{T11} and for every $t>0$, $\Pr(Z > \Ep[Z] + t) \leq \text{exp}(-t^2/(2\sigma^2))$ by Theorem 2.1.1 in \cite{AT07}. \qed
\end{remark}

\begin{lemma}
\label{lem: rv} Let $X_{1},\dots ,X_{n}$ be i.i.d. random vectors in $%
\mathbb{R}^{p}$ with $p\geq 2$. Also, let $K=({\mathrm{E}}[\max_{1\leq i\leq
n}\max_{1\leq j\leq p}|X_{ij}^{2}])^{1/2}$ and for $\ell \geq 1$, let 
\begin{equation*}
\delta _{n}=\frac{K\sqrt{\ell \log p}}{\sqrt{n}}\left( 1+(\log \ell )(\log
^{1/2}n)\right) .
\end{equation*}%
Moreover, let $\mathcal{S}^{p}(\ell )=\{ \theta \in \mathbb{R}^{p}\colon
\Vert \theta \Vert =1$ and $\Vert \theta \Vert _{0}\leq \ell \}$. Then 
\begin{equation*}
{\mathrm{E}}\left[ \sup_{\theta \in \mathcal{S}^{p}(\ell )}\left \vert \frac{%
1}{n}\sum_{i=1}^{n}(X_{i}^{\prime }\theta )^{2}-{\mathrm{E}}[(X_{1}^{\prime
}\theta )^{2}]\right \vert \right] \leq C\left( \delta _{n}^{2}+\delta
_{n}\sup_{\theta \in \mathcal{S}^{p}(\ell )}\left( {\mathrm{E}}%
[(X_{1}^{\prime }\theta )^{2}]\right) ^{1/2}\right),
\end{equation*}%
where $C> 0$ is a universal constant.
\end{lemma}

\begin{proof}
See Lemma B.1 in \cite{BCC15}. See also \cite{RV08} for the original result.
\end{proof}

\begin{remark}
If $X_1,\dots,X_n$ are centered Gaussian random vectors in $\mathbb R^p$ with $p\geq 2$, then for any $\epsilon_1,\epsilon_2,\ell>0$ such that $\epsilon_1 + \epsilon_2 < 1$ and $\ell\leq \min(p,\epsilon_1^2 n)$,
$$
\sup_{\theta \in \mathcal{S}^{p}(\ell )}\left \vert \frac{%
1}{n}\sum_{i=1}^{n}(X_{i}^{\prime }\theta )^{2}-{\mathrm{E}}[(X_{1}^{\prime
}\theta )^{2}]\right \vert \leq 3(\epsilon_1 + \epsilon_2)\sup_{\theta\in\mathcal S^p(\ell)}\Ep[(X_1'\theta)^2] 
$$
with probability at least $1 - 2p^me^{-n\epsilon_2^2/2}$ by the proof of Proposition 2 in \cite{ZH08}.
\qed
\end{remark}

\begin{lemma}
\label{lem: integration by parts stein}
Let $\varepsilon$ be a random variable that is absolutely continuous with respect
to Lebesgue measure on $\mathbb{R}$ with continuously differentiable pdf $\chi$ and suppose that $f\colon \mathbb R\to\mathbb R$ is either Lipschitz-continuous or continuously differentiable with finite $\Ep[|f'(\varepsilon)|]$. Suppose also that both $\Ep[|f(\varepsilon)|]$ and $\Ep[|f(\varepsilon)\chi'(\varepsilon)|/\chi(\varepsilon)]$ are finite. Then
\begin{equation}\label{eq: stein integration by parts}
{\mathrm{E}}[f^{\prime }(\varepsilon)] = -{\mathrm{E}}[f(\varepsilon)\chi^{%
\prime }(\varepsilon)/\chi(\varepsilon)].
\end{equation}
\end{lemma}

\begin{remark}
When $\varepsilon$ has a $N(0,\sigma^2)$ distribution, the formula \eqref{eq: stein integration by parts} reduces to the well-known Stein identity, $\Ep[f'(\varepsilon)] = \Ep[\varepsilon f(\varepsilon)]/\sigma^2$.\qed
\end{remark}

\begin{proof}
The proof follows immediately from integration by parts and the Lebesgue dominated convergence theorem; for example, see Section 13.1.1 in \cite{CGS} for similar results.
\end{proof}

\begin{lemma}
\label{lem: nongaussian concentration} Let $e = (e_1,\dots,e_n)$ be a
standard Gaussian random vector and let $Q_i\colon \mathbb{R}\to \mathbb{R}$%
, $i=1,\dots,n$ be some strictly increasing continuously differentiable
functions. Denote $\varepsilon = (\varepsilon_1,\dots,\varepsilon_n)$ where $%
\varepsilon_i = Q_i(e_i)$, $i=1,\dots,n$, and let $f\colon \mathbb{R}^n\to%
\mathbb{R}$ be Lipschitz-continuous with Lipschitz constant $L>0$. Then for
any convex $u\colon \mathbb{R}\to \mathbb{R}_{+}$, the random variable 
\begin{equation*}
V = f(\varepsilon) = f(\varepsilon_1,\dots,\varepsilon_n)
\end{equation*}
satisfies the following inequality: 
\begin{equation*}
{\mathrm{E}}[u(V - {\mathrm{E}}[V])] \leq {\mathrm{E}}\left[ u\left( \frac{%
\pi L}{2}\max_{1\leq i\leq n} Q_i^{\prime }(e_i) \xi \right) \right],
\end{equation*}
where $\xi$ is a standard Gaussian random variable that is independent of $e$.
\end{lemma}

\begin{remark}
The proof of this lemma given below mimics the well-known interpolation proof of the Gaussian concentration inequality for Lipschitz functions; see Theorem 2.1.12 in \cite{T12} for example.\qed
\end{remark}

\begin{proof}
To prove the asserted claim, let $\tilde e=(\tilde e_{1},\dots,\tilde e_{n})$
be another standard Gaussian random vector that is independent of $e$. Also,
define 
\begin{equation*}
F(x)=F(x_{1},\dots,x_{n})=f(Q_1(x_{1}),\dots,Q_n(x_{n})),\quad
x=(x_{1},\dots,x_{n})\in \mathbb{R}^{n}.
\end{equation*}
Then 
\begin{align*}
&\mathbb{E}[u(V-\mathbb{E}V)] =\mathbb{E}\Big[u(F(e)-\mathbb{E}F(e))\Big] =%
\mathbb{E}\Big[u(F(e)-\mathbb{E}F(\tilde e))\Big] \\
& \qquad =\mathbb{E}\Big[u(\mathbb{E}[F(e)-F(\tilde e)\mid e])\Big] \leq%
\mathbb{E}\Big[\mathbb{E}\Big[u(F(e)-F(\tilde e))\mid e\Big]\Big] \\
&\qquad =\mathbb{E}\Big[u(F(e)-F(\tilde e))\Big].
\end{align*}
Further, define 
\begin{equation*}
h(\theta)=F\Big(\tilde e\cos(\pi \theta/2)+e\sin(\pi \theta/2)\Big),\quad
\theta \in[0,1],
\end{equation*}
so that $h(1)=F(e)$, $h(0)=F(\tilde e)$, and for all $\theta \in(0,1),$ 
\begin{align*}
h^{\prime }(\theta) & =\frac{\pi}{2}\sum_{i=1}^{n}F_{i}(\tilde e\cos(\pi
\theta/2)+e\sin(\pi \theta/2))(e_{j}\cos(\pi \theta/2)-\tilde e_{j}\sin(\pi
\theta/2)) \\
&=\frac{\pi}{2}(\nabla F(\widetilde{W}_{\theta}),W_{\theta}),
\end{align*}
where we denoted 
\begin{equation*}
W_{\theta}=e\cos(\pi \theta/2)-\tilde e\sin(\pi \theta/2)\text{ and }%
\widetilde{W}_{\theta}=\tilde e\cos(\pi \theta/2)+e\sin(\pi \theta/2).
\end{equation*}
Note that for each $\theta \in(0,1)$, the random vectors $W_{\theta}$ and $%
\widetilde{W}_{\theta}$ are independent standard Gaussian. Hence, 
\begin{align*}
{\mathrm{E}}\Big[u(F(e)-F(\tilde e))\Big] & =\mathbb{E}\Big[u(h(1)-h(0))%
\Big] =\mathbb{E}\Big[u\Big(\int_{0}^{1}h^{\prime }(\theta)d\theta \Big)\Big]
\\
& \leq \mathbb{E}\Big[\int_{0}^{1}u(h^{\prime }(\theta))d\theta \Big] =%
\mathbb{E}\Big[\int_{0}^{1}u\Big(\frac{\pi}{2}(\nabla F(\widetilde{W}%
_{\theta}),W_{\theta})\Big)d\theta \Big] \\
& =\int_{0}^{1}\mathbb{E}\Big[u\Big(\frac{\pi}{2}(\nabla F(\widetilde{W}%
_{\theta}),W_{\theta})\Big)\Big]d\theta \\
& =\int_{0}^{1}\mathbb{E}\Big[u\Big(\frac{\pi}{2}(\nabla F(e),\tilde e)\Big)%
\Big]d\theta =\mathbb{E}\Big[u\Big(\frac{\pi}{2}(\nabla F(e),\tilde e)\Big)%
\Big].
\end{align*}
Next, note that since $e$ and $\tilde e$ are independent standard Gaussian
random vectors, conditional on $e$, the random variable $(\nabla F(e),\tilde
e)$ is zero-mean Gaussian with variance 
\begin{align*}
\sum_{i=1}^{n}\Big(\frac{\partial F}{\partial e_i}(e)\Big)^{2} &
=\sum_{i=1}^{n}\Big(\frac{\partial f}{\partial \varepsilon_i}(\varepsilon)%
\Big)^{2}(Q^{\prime }_i(e_{i}))^{2} \\
& \leq \max_{1\leq i\leq n}(Q^{\prime }_i(e_{i}))^{2}\sum_{i=1}^{n}\Big(%
\frac{\partial f}{\partial \varepsilon_i}(\varepsilon)\Big)^{2}\leq
L^{2}\max_{1\leq i\leq n}(Q^{\prime }_i(e_{i}))^{2}.
\end{align*}
Therefore, using the fact that $u$ is convex, we conclude that 
\begin{align*}
\mathbb{E}\Big[u\Big(\frac{\pi}{2}(\nabla F(e),\tilde e)\Big)\Big] & =%
\mathbb{E}\Big[\mathbb{E}\Big[u\Big(\frac{\pi}{2}(\nabla F(e),\tilde e)\Big)%
\mid e\Big]\Big] \\
& =\mathbb{E}\Big[\mathbb{E}\Big[u\Big(\frac{\pi}{2}\Big(\sum_{i=1}^{n}\Big(%
\frac{\partial F}{\partial e_i}(e)\Big)^{2}\Big)^{1/2}\xi \Big)\mid e\Big]%
\Big] \\
& \leq \mathbb{E}\Big[\mathbb{E}\Big[u\Big(\frac{\pi L}{2}\max_{1\leq i\leq
n}Q^{\prime }_i(e_{i})\xi \Big)\mid e\Big]\Big] \\
& =\mathbb{E}\Big[u\Big(\frac{\pi L}{2}\max_{1\leq i\leq n}Q^{\prime
}_i(e_{i})\xi \Big)\Big],
\end{align*}
where $\xi$ is a standard Gaussian random variable that is independent of
the vector $e$. Combining presented inequalities gives the asserted claim.
\end{proof}

\begin{lemma}
\label{lem: quantile properties} Let $X_{1},\dots ,X_{m}$ be random
variables (not necessarily independent). Then for all $\alpha \in (0,1)$, 
\begin{equation*}
Q_{1-\alpha }(X_{1}+\dots +X_{m})\leq Q_{1-\alpha /(2m)}(X_{1})+\dots
+Q_{1-\alpha /(2m)}(X_{m}),
\end{equation*}%
where for any random variable $Z$ and any number $\alpha \in (0,1)$, $%
Q_{\alpha }(Z)$ denotes the $\alpha $th quantile of the distribution of $Z$,
i.e. $Q_{\alpha }(Z)=\inf \{z\in \mathbb{R}\colon \alpha \leq \Pr (Z\leq
z)\} $.
\end{lemma}

\begin{proof}
To prove the asserted claim, suppose to the contrary that 
\begin{equation*}
Q_{1-\alpha }(X_{1}+\dots +X_{m})>Q_{1-\alpha /(2m)}(X_{1})+\dots
+Q_{1-\alpha /(2m)}(X_{m}).
\end{equation*}%
Then by the union bound, 
\begin{align*}
\alpha & \leq \Pr (X_{1}+\dots +X_{m}\geq Q_{1-\alpha }(X_{1}+\dots +X_{m}))
\\
& \leq \Pr (X_{1}+\dots +X_{m}>Q_{1-\alpha /(2m)}(X_{1})+\dots +Q_{1-\alpha
/(2m)}(X_{m})) \\
& \leq \Pr (X_{1}>Q_{1-\alpha /(2m)}(X_{1}))+\dots +\Pr (X_{m}>Q_{1-\alpha
/(2m)}(X_{m})) \\
& \leq \alpha /(2m)+\dots +\alpha /(2m)=\alpha /2,
\end{align*}%
which is a contradiction. Thus, the asserted claim follows.
\end{proof}

%\begin{lemma}\label{lem: tropp}
%Let $(p_1',q_1')',\dots,(p_n',q_n')'$ be a sequence of i.i.d. random vectors where $p_i$'s are vectors in $\mathbb R^K$ and $q_i$'s are vectors in $\mathbb R^J$. Assume that $\|p_1\|\leq \xi_n$, $q_1\leq \xi_n$, $\|E[p_1 p_1']\| \leq C_p$ and $\|E[q_1 q_1']\|\leq C_q$ where $\xi_n\geq 1$. Then for all $t>0$,
%$$
%P\left(\left\|\frac{1}{n}\sum_{i=1}^n p_i q_i' - E[p_1 q_1']\right\|>t\right)\leq \exp\left(\log(K + J) - \frac{A n t^2}{\xi_n^2(1 + t)}\right)
%$$
%where $A>0$ is a constant depending only on $C_p$ and $C_q$.
%\end{lemma}
%\begin{proof}
%See Lemma 10 in \cite{CW15}. See also \cite{T12} for the original result.
%\end{proof}

\newpage

\begin{figure}[tbp]
\includegraphics[width=13cm]{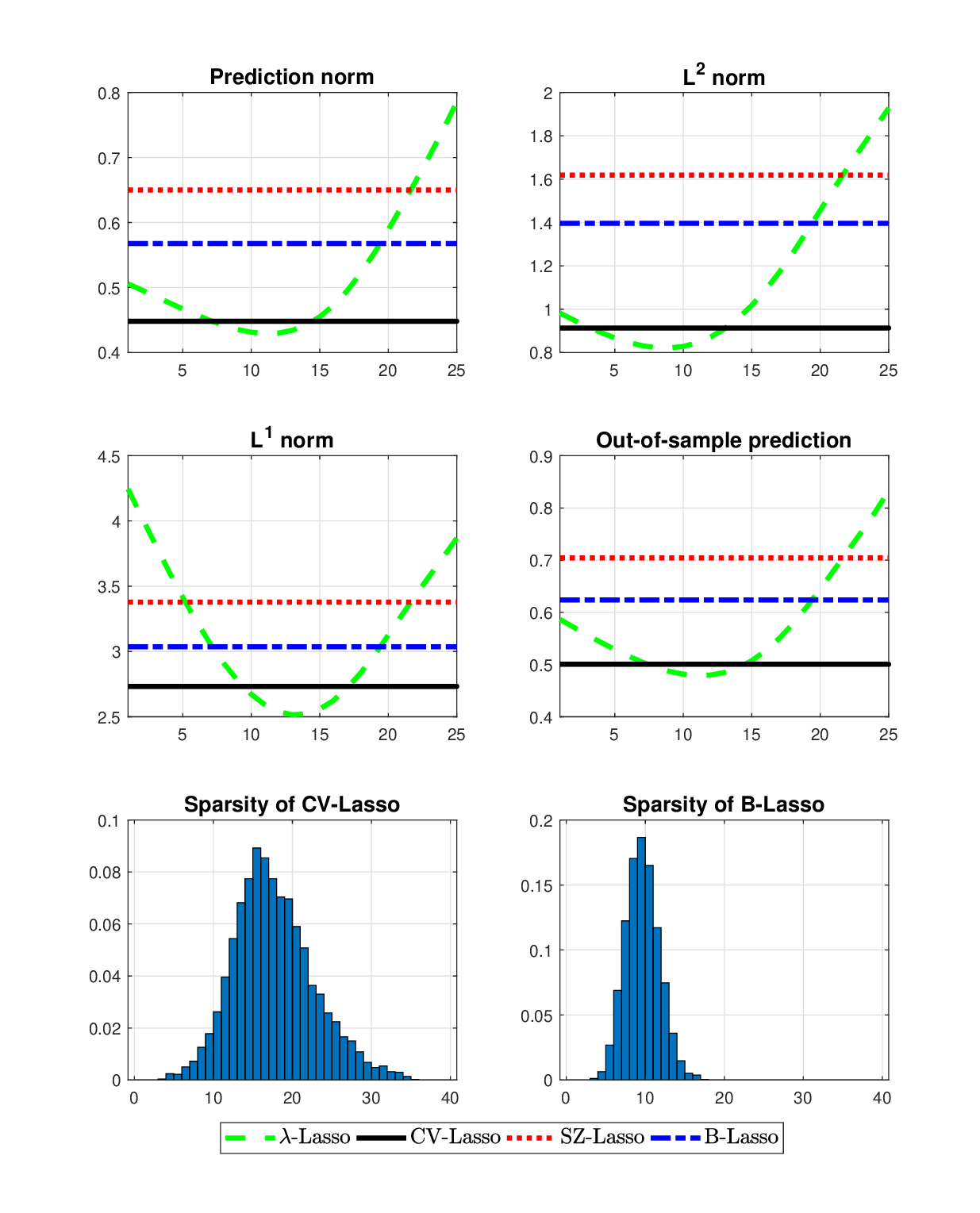}
\caption{DGP1, $n = 100$, $p=40$, and $\protect \rho=0.75$. The top-left,
top-right, middle-left, and middle-right panels show the mean of estimation
error of Lasso estimators in the prediction, $L^2$, $L^1$, and out-of-sample
prediction norms. The dashed line represents the mean of estimation error of
the Lasso estimator as a function of $\protect \lambda$ (we perform the Lasso
estimator for each value of $\protect \lambda$ in the candidate set $%
\Lambda_n $; we sort the values in $\Lambda_n$ from the smallest to the
largest, and put the order of $\protect \lambda$ on the horizontal axis; we
only show the results for values of $\protect \lambda$ up to order 25 as
these give the most meaningful comparisons). The solid, dotted, and
dashed-dotted horizontal lines represent the mean of the estimation error of
the CV-Lasso, SZ-Lasso, and B-Lasso estimators, respectively.}
\label{fig:1}
\end{figure}

\begin{figure}[tbp]
\includegraphics[width=13cm]{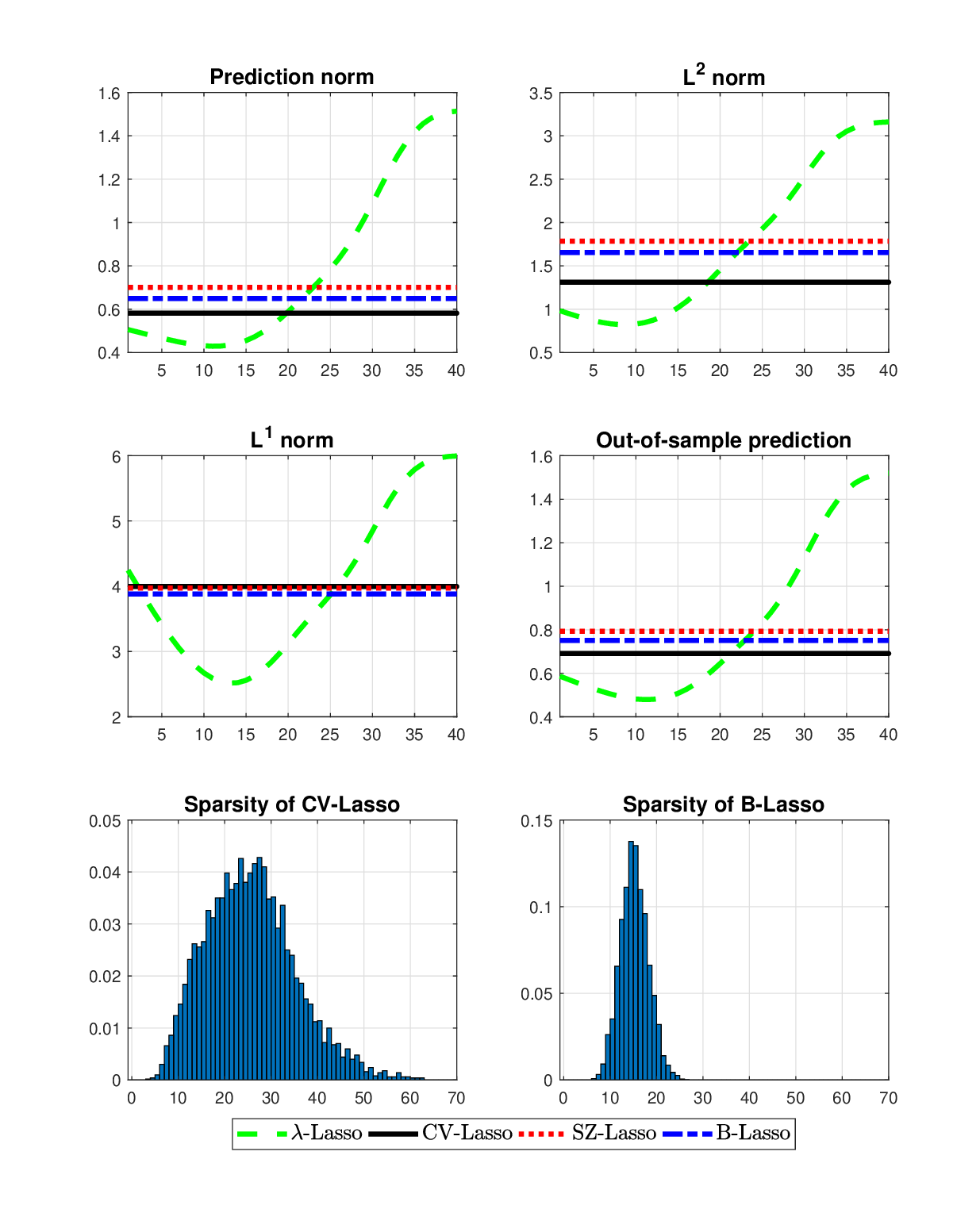}
\caption{DGP1, $n = 100$, $p=400$, and $\protect \rho=0.75$. The top-left,
top-right, middle-left, and middle-right panels show the mean of estimation
error of Lasso estimators in the prediction, $L^2$, $L^1$, and out-of-sample
prediction norms. The dashed line represents the mean of estimation error of
the Lasso estimator as a function of $\protect \lambda$ (we perform the Lasso
estimator for each value of $\protect \lambda$ in the candidate set $%
\Lambda_n $; we sort the values in $\Lambda_n$ from the smallest to the
largest, and put the order of $\protect \lambda$ on the horizontal axis; we
only show the results for values of $\protect \lambda$ up to order 25 as
these give the most meaningful comparisons). The solid, dotted, and
dashed-dotted horizontal lines represent the mean of the estimation error of
the CV-Lasso, SZ-Lasso, and B-Lasso estimators, respectively.}
\label{fig:2}
\end{figure}

\newpage

%\begin{sidewaystable}[ph!]
%\begin{adjustbox}{angle=90}
\begin{landscape}
	\begin{table}[tbp]
		\begin{center}
			\caption{The mean of estimation error of Lasso estimators} 
			\begin{tabular}{cccccccccccc}
				\hline \hline
				& \multicolumn{11}{c}{DGP1 ($\rho = 0.5$)} \\ \cline{2-12}
				& \multicolumn{3}{c}{Prediction norm} &  & \multicolumn{3}{c}{$L^{2}$ %
					norm} &  & \multicolumn{3}{c}{Out-of-Sample prediction norm} \\ 
				\cline{2-4}\cline{6-8}\cline{10-12}
				& {\small CV-Lasso} & $\lambda ${\small -Lasso} & {\small B-Lasso} &  & 
				{\small CV-Lasso}   & $\lambda ${\small -Lasso} & {\small B-Lasso} &  & 
				{\small CV-Lasso}   & $\lambda ${\small -Lasso} & {\small B-Lasso} \\ 
				\cline{2-4}\cline{6-8}\cline{10-12}
				\multicolumn{1}{l}   {\small (n, p)=(100, 40)} %
				& {\small 0.4252} & {\small 0.4097} & {\small 0.4435} &  & 
				{\small 0.6164} & {\small 0.5700} & {\small 0.7013} &  & 
				{\small 0.4701} & {\small 0.4530} & {\small 0.4883} \\
				
				\multicolumn{1}{l}   {\small (n, p)=(100, 100)} %
				& {\small 0.5243} & {\small 0.5040} & {\small 0.5303} &  & 
				{\small 0.8206} & {\small 0.7598} & {\small 0.8897} &  & 
				{\small 0.6091} & {\small 0.5885} & {\small 0.6139} \\
				
				\multicolumn{1}{l}   {\small (n, p)=(100, 400)} %
				& {\small 0.7023} & {\small 0.6448} & {\small 0.6595} &  & 
				{\small 1.2629} & {\small 1.1624} & {\small 1.2548} &  & 
				{\small 0.8852} & {\small 0.8474} & {\small 0.8565} \\
				
				\multicolumn{1}{l}   {\small (n, p)=(400, 40)} %
				& {\small 0.2116} & {\small 0.2047} & {\small 0.2174} &  & 
				{\small 0.2875} & {\small 0.2634} & {\small 0.3186} &  & 
				{\small 0.2164} & {\small 0.2095} & {\small 0.2224} \\
				
				\multicolumn{1}{l}   {\small (n, p)=(400, 100)} %
				& {\small 0.2581} & {\small 0.2501} & {\small 0.2561} &  & 
				{\small 0.3674} & {\small 0.3301} & {\small 0.3790} &  & 
				{\small 0.2667} & {\small 0.2588} & {\small 0.2648} \\
				
				\multicolumn{1}{l}   {\small (n, p)=(400, 400)} %
				& {\small 0.3300} & {\small 0.3206} & {\small 0.3206} &  & 
				{\small 0.5018} & {\small 0.4546} & {\small 0.4807} &  & 
				{\small 0.3473} & {\small 0.3391} & {\small 0.3391} \\ \hline
				
				\multicolumn{1}{l}{} & \multicolumn{11}{c}{DGP2 ($\rho = 0.5$)} \\ \cline{2-12}
				\multicolumn{1}{l}{} & \multicolumn{3}{c}{Prediction norm} &  & 
				\multicolumn{3}{c}{$L^{2}$ norm} &  & \multicolumn{3}{c}{Out-of-Sample prediction norm} \\ 
				\cline{2-4}\cline{6-8}\cline{10-12}
				
				\multicolumn{1}{l}   {\small (n, p)=(100, 40)} %
				& {\small 0.7532} & {\small 0.7123} & {\small 0.7672} &  & 
				{\small 1.1041} & {\small 0.9907} & {\small 1.2107} &  & 
				{\small 0.8293} & {\small 0.7857} & {\small 0.8419} \\
				
				\multicolumn{1}{l}   {\small (n, p)=(100, 100)} %
				& {\small 0.9237} & {\small 0.8641} & {\small 0.8917} &  & 
				{\small 1.4644} & {\small 1.3044} & {\small 1.4792} &  & 
				{\small 1.0551} & {\small 1.0048} & {\small 1.0264} \\
				
				\multicolumn{1}{l}   {\small (n, p)=(100, 400)} %
				& {\small 1.1497} & {\small 1.0465} & {\small 1.0493} &  & 
				{\small 1.9868} & {\small 1.8541} & {\small 1.8962} &  & 
				{\small 1.3631} & {\small 1.3103} & {\small 1.3118} \\
				
				\multicolumn{1}{l}   {\small (n, p)=(400, 40)} %
				& {\small 0.3647} & {\small 0.3521} & {\small 0.3746} &  & 
				{\small 0.4961} & {\small 0.4529} & {\small 0.5485} &  & 
				{\small 0.3731} & {\small 0.3603} & {\small 0.3831} \\
				
				\multicolumn{1}{l}   {\small (n, p)=(400, 100)} %
				& {\small 0.4470} & {\small 0.4325} & {\small 0.4431} &  & 
				{\small 0.6351} & {\small 0.5717} & {\small 0.6550} &  & 
				{\small 0.4616} & {\small 0.4473} & {\small 0.4577} \\
				
				\multicolumn{1}{l}   {\small (n, p)=(400, 400)} %
				& {\small 0.5739} & {\small 0.5564} & {\small 0.5561} &  & 
				{\small 0.8714} & {\small 0.7882} & {\small 0.8333} &  & 
				{\small 0.6037} & {\small 0.5885} & {\small 0.5882} \\ \hline

				& \multicolumn{11}{c}{DGP1 ($\rho = 0.75$)} \\ \cline{2-12}
				& \multicolumn{3}{c}{Prediction norm} &  & \multicolumn{3}{c}{$L^{2}$ %
					norm} &  & \multicolumn{3}{c}{Out-of-Sample prediction norm} \\ 
				\cline{2-4}\cline{6-8}\cline{10-12}
				& {\small CV-Lasso} & $\lambda ${\small -Lasso} & {\small B-Lasso} &  & 
				{\small CV-Lasso} & $\lambda ${\small -Lasso} & {\small B-Lasso} &  & 
				{\small CV-Lasso} & $\lambda ${\small -Lasso} & {\small B-Lasso} \\ 
				\cline{2-4}\cline{6-8}\cline{10-12}
				\multicolumn{1}{l}   {\small (n, p)=(100, 40)} %
				& {\small 0.4481} & {\small 0.4292} & {\small 0.5677} &  & 
				{\small 0.9133} & {\small 0.8213} & {\small 1.3963} &  & 
				{\small 0.5005} & {\small 0.4791} & {\small 0.6238} \\
				
				\multicolumn{1}{l}   {\small (n, p)=(100, 100)} %
				& {\small 0.5817} & {\small 0.5486} & {\small 0.6496} &  & 
				{\small 1.3110} & {\small 1.1144} & {\small 1.6547} &  & 
				{\small 0.6907} & {\small 0.6611} & {\small 0.7514} \\
				
				\multicolumn{1}{l}   {\small (n, p)=(100, 400)} %
				& {\small 0.7616} & {\small 0.6957} & {\small 0.7288} &  & 
				{\small 2.0360} & {\small 1.8350} & {\small 2.0207} &  & 
				{\small 0.9836} & {\small 0.9525} & {\small 0.9543} \\
				
				\multicolumn{1}{l}   {\small (n, p)=(400, 40)} %
				& {\small 0.2206} & {\small 0.2141} & {\small 0.2829} &  & 
				{\small 0.4143} & {\small 0.3745} & {\small 0.6556} &  & 
				{\small 0.2263} & {\small 0.2196} & {\small 0.2894} \\
				
				\multicolumn{1}{l}   {\small (n, p)=(400, 100)} %
				& {\small 0.2782} & {\small 0.2717} & {\small 0.3322} &  & 
				{\small 0.5381} & {\small 0.4688} & {\small 0.7766} &  & 
				{\small 0.2897} & {\small 0.2830} & {\small 0.3436} \\
				
				\multicolumn{1}{l}   {\small (n, p)=(400, 400)} %
				& {\small 0.3847} & {\small 0.3771} & {\small 0.4112} &  & 
				{\small 0.8217} & {\small 0.6751} & {\small 0.9774} &  & 
				{\small 0.4151} & {\small 0.4081} & {\small 0.4402} \\ \hline
				
				\multicolumn{1}{l}{} & \multicolumn{11}{c}{DGP2 ($\rho = 0.75$)} \\ \cline{2-12}
				\multicolumn{1}{l}{} & \multicolumn{3}{c}{Prediction norm} &  & 
				\multicolumn{3}{c}{$L^{2}$ norm} &  & \multicolumn{3}{c}{Out-of-Sample prediction norm} \\ 
				\cline{2-4}\cline{6-8}\cline{10-12}
				
				\multicolumn{1}{l}   {\small (n, p)=(100, 40)} %
				& {\small 0.7730} & {\small 0.7285} & {\small 0.8393} &  & 
				{\small 1.6151} & {\small 1.3895} & {\small 1.9690} &  & 
				{\small 0.8520} & {\small 0.8072} & {\small 0.9105} \\
				
				\multicolumn{1}{l}   {\small (n, p)=(100, 100)} %
				& {\small 0.9619} & {\small 0.8843} & {\small 0.9407} &  & 
				{\small 2.1316} & {\small 1.8093} & {\small 2.2295} &  & 
				{\small 1.0938} & {\small 1.0293} & {\small 1.0631} \\
				
				\multicolumn{1}{l}   {\small (n, p)=(100, 400)} %
				& {\small 1.2454} & {\small 1.0586} & {\small 1.0740} &  & 
				{\small 2.8271} & {\small 2.4914} & {\small 2.6602} &  & 
				{\small 1.3966} & {\small 1.3298} & {\small 1.3298} \\
				
				\multicolumn{1}{l}   {\small (n, p)=(400, 40)} %
				& {\small 0.3811} & {\small 0.3696} & {\small 0.4876} &  & 
				{\small 0.7141} & {\small 0.6427} & {\small 1.1292} &  & 
				{\small 0.3907} & {\small 0.3788} & {\small 0.4984} \\
				
				\multicolumn{1}{l}   {\small (n, p)=(400, 100)} %
				& {\small 0.4859} & {\small 0.4719} & {\small 0.5710} &  & 
				{\small 0.9443} & {\small 0.8132} & {\small 1.3320} &  & 
				{\small 0.5061} & {\small 0.4920} & {\small 0.5910} \\
				
				\multicolumn{1}{l}   {\small (n, p)=(400, 400)} %
				& {\small 0.6790} & {\small 0.6499} & {\small 0.6834} &  & 
				{\small 1.5102} & {\small 1.1683} & {\small 1.6067} &  & 
				{\small 0.7229} & {\small 0.7028} & {\small 0.7291} \\ \hline  
				
				\hline
			\end{tabular}
		\end{center}
	\end{table}
\end{landscape}
%\end{adjustbox}
%\end{sidewaystable}

\newpage

\begin{table}[tbp]
\caption{Probabilities for the number of non-zero coefficients of the
CV-Lasso estimator hitting different brackets}
\label{tab:2}
\begin{center}
\begin{tabular}[H]{lcccccccc}
\hline \hline
& \multicolumn{8}{c}{DGP1 ($\rho = 0.5$)} \\ \cline{2-9}
& {\small [0, 5]} & {\small [6, 10]} & {\small [11, 15]} & {\small [16, 20]}
& {\small [21, 25]} & {\small [26, 30]} & {\small [31, 35]} & {\small [36,
p] } \\ \cline{2-9}
{\small (n, p)=(100, 40)} & {\small 0.0008} & {\small 0.0766} & {\small %
0.3598} & {\small 0.3548} & {\small 0.1582} & {\small 0.0390} & {\small %
0.0088} & {\small 0.0020} \\ 
{\small (n, p)=(100, 100)} & {\small 0.0006} & {\small 0.0120} & {\small %
0.0822} & {\small 0.2146} &  {\small 0.2606} & {\small 0.1994} & {\small %
0.1186} & {\small 0.1120} \\ 
{\small (n, p)=(100, 400)} & {\small 0.0010} & {\small 0.0190} & {\small %
0.0480} & {\small 0.0760} &  {\small 0.0978} & {\small 0.1196} & {\small %
0.1288} & {\small 0.5098} \\ 
{\small (n, p)=(400, 40)} & {\small 0.0006} & {\small 0.0964} & {\small %
0.3926} & {\small 0.3460} & {\small 0.1292} & {\small 0.0316} & {\small %
0.0034} & {\small 0.0002} \\ 
{\small (n, p)=(400, 100)} & {\small 0.0006} & {\small 0.0176} & {\small %
0.1404} & {\small 0.2624} &  {\small 0.2596} & {\small 0.1780} & {\small %
0.0828} & {\small 0.0586} \\ 
{\small (n, p)=(400, 400)} & {\small 0.0000} & {\small 0.0016} & {\small %
0.0212} & {\small 0.0728} &  {\small 0.1372} & {\small 0.1618} & {\small %
0.1664} & {\small 0.4390} \\ \hline
& \multicolumn{8}{c}{DGP2 ($\rho = 0.5$)} \\ \cline{2-9}
& {\small [0, 5]} & {\small [6, 10]} & {\small [11, 15]} & {\small [16, 20]}
& {\small [21, 25]} & {\small [26, p]} & {\small [31, p]} & {\small [36, p]}
\\ \cline{2-9}
{\small (n, p)=(100, 40)} & {\small 0.0142} & {\small 0.1436} & {\small %
0.3418} & {\small 0.3070} & {\small 0.1402} & {\small 0.0432} & {\small %
0.0094} & {\small 0.0006} \\ 
{\small (n, p)=(100, 100)} & {\small 0.0158} & {\small 0.1096} & {\small %
0.1866} & {\small 0.2186} &  {\small 0.1828} & {\small 0.1338} & {\small %
0.0754} & {\small 0.0774} \\ 
{\small (n, p)=(100, 400)} & {\small 0.0310} & {\small 0.0988} & {\small %
0.1586} & {\small 0.1752} &  {\small 0.1446} & {\small 0.1042} & {\small %
0.0830} & {\small 0.2046} \\ 
{\small (n, p)=(400, 40)} & {\small 0.0008} & {\small 0.1030} & {\small %
0.4032} & {\small 0.3334} & {\small 0.1258} & {\small 0.0268} & {\small %
0.0060} & {\small 0.0010} \\ 
{\small (n, p)=(400, 100)} & {\small 0.0002} & {\small 0.0202} & {\small %
0.1358} & {\small 0.2530} & {\small 0.2684} & {\small 0.1704} & {\small %
0.0814} & {\small 0.0706} \\ 
{\small (n, p)=(400, 400)} & {\small 0.0002} & {\small 0.0020} & {\small %
0.0274} & {\small 0.0798} & {\small 0.1280} & {\small 0.1590} & {\small %
0.1592} & {\small 0.4444} \\ \hline
& \multicolumn{8}{c}{DGP1 ($\rho = 0.75$)} \\ \cline{2-9}
& {\small [0, 5]} & {\small [6, 10]} & {\small [11, 15]} & {\small [16, 20]}
& {\small [21, 25]} & {\small [26, p]} & {\small [31, p]} & {\small [36, p]}
\\ \cline{2-9}
{\small (n, p)=(100, 40)} & {\small 0.0028} & {\small 0.0448} & {\small %
0.2658} & {\small 0.3920} & {\small 0.2050} & {\small 0.0716} & {\small %
0.0178} & {\small 0.0002} \\ 
{\small (n, p)=(100, 100)} & {\small 0.0006} & {\small 0.0316} & {\small %
0.1080} & {\small 0.1604} & {\small 0.1948} & {\small 0.2000} & {\small %
0.1470} & {\small 0.1576} \\ 
{\small (n, p)=(100, 400)} & {\small 0.0206} & {\small 0.0194} & {\small %
0.0506} & {\small 0.1110} & {\small 0.1534} & {\small 0.1660} & {\small %
0.1398} & {\small 0.3392} \\ 
{\small (n, p)=(400, 40)} & {\small 0.0000} & {\small 0.0278} & {\small %
0.2926} & {\small 0.4222} & {\small 0.1966} & {\small 0.0510} & {\small %
0.0090} & {\small 0.0008} \\ 
{\small (n, p)=(400, 100)} & {\small 0.0000} & {\small 0.0002} & {\small %
0.0136} & {\small 0.1156} & {\small 0.2480} & {\small 0.2920} & {\small %
0.1836} & {\small 0.1470} \\ 
{\small (n, p)=(400, 400)} & {\small 0.0000} & {\small 0.0000} & {\small %
0.0002} & {\small 0.0004} & {\small 0.0060} & {\small 0.0192} & {\small %
0.0530} & {\small 0.9212} \\ \hline
& \multicolumn{8}{c}{DGP2 ($\rho = 0.75$)} \\ \cline{2-9}
& {\small [0, 5]} & {\small [6, 10]} & {\small [11, 15]} & {\small [16, 20]}
& {\small [21, 25]} & {\small [26, p]} & {\small [31, p]} & {\small [36, p]}
\\ \cline{2-9}
{\small (n, p)=(100, 40)} & {\small 0.0254} & {\small 0.2152} & {\small %
0.3326} & {\small 0.2546} & {\small 0.1206} & {\small 0.0392} & {\small %
0.0116} & {\small 0.0008} \\
{\small (n, p)=(100, 100)} & {\small 0.0904} & {\small 0.1024} & {\small %
0.2192} & {\small 0.2262} & {\small 0.1606} & {\small 0.0958} & {\small %
0.0502} & {\small 0.0552} \\ 
{\small (n, p)=(100, 400)} & {\small 0.3916} & {\small 0.1022} & {\small %
0.0988} & {\small 0.0906} & {\small 0.0826} & {\small 0.0650} & {\small %
0.0558} & {\small 0.1134} \\ 
{\small (n, p)=(400, 40)} & {\small 0.0002} & {\small 0.0290} & {\small %
0.2976} & {\small 0.4314} & {\small 0.1862} & {\small 0.0468} & {\small %
0.0082} & {\small 0.0006} \\ 
{\small (n, p)=(400, 100)} & {\small 0.0000} & {\small 0.0050} & {\small %
0.0282} & {\small 0.1264} &  {\small 0.2370} & {\small 0.2820} & {\small %
0.1804} & {\small 0.1410} \\ 
{\small (n, p)=(400, 400)} & {\small 0.0002} & {\small 0.0134} & {\small %
0.0582} & {\small 0.0974} &  {\small 0.1156} & {\small 0.1020} & {\small %
0.0860} & {\small 0.5272} \\ \hline \hline
\end{tabular}%
\end{center}
\end{table}

\begin{table}[tbp]
\caption{Probabilities for $\max_{1\leq j\leq p}n^{-1}|\textstyle{\
\sum_{i=1}^n}X_{i j}\protect \varepsilon_i|/\widehat \protect \lambda$ hitting
different brackets}
\begin{center}
\begin{tabular}{lccccccc}
\hline \hline
& \multicolumn{7}{c}{DGP1 ($\rho = 0.50$)} \\ \cline{2-8}
& {\small [0, 0.5)} & {\small [0.6,\ 1)} & {\small [1, 1.5)} & {\small [1.5,
2)} & {\small [2, 2.5)} & {\small [2.5, 3)} & {\small [3, }$\infty ${\small %
) } \\ \cline{2-8}
{\small (n, p)=(100, 40)} & {\small 0.0002} & {\small 0.0910} & {\small %
0.3458} & {\small 0.2842} & {\small 0.1460} & {\small 0.0670} & {\small %
0.0648} \\ 
{\small (n, p)=(100, 100)} & {\small 0.0000} & {\small 0.1560} & {\small %
0.4376} & {\small 0.2470} & {\small 0.0910} & {\small 0.0322} & {\small %
0.0338} \\ 
{\small (n, p)=(100, 400)} & {\small 0.0116} & {\small 0.3262} & {\small %
0.3374} & {\small 0.1396} & {\small 0.0592} & {\small 0.0282} & {\small %
0.0566} \\ 
{\small (n, p)=(400, 40)} & {\small 0.0000} & {\small 0.1118} & {\small %
0.4292} & {\small 0.3042} & {\small 0.0988} & {\small 0.0364} & {\small %
0.0182} \\ 
{\small (n, p)=(400, 100)} & {\small 0.0000} & {\small 0.2648} & {\small %
0.5784} &  {\small 0.1362} & {\small 0.0158} & {\small 0.0032} & {\small %
0.0004} \\ 
{\small (n, p)=(400, 400)} & {\small 0.0000} & {\small 0.5828} & {\small %
0.3972} &  {\small 0.0162} & {\small 0.0004} & {\small 0.0000} & {\small %
0.0000} \\ \hline
& \multicolumn{7}{c}{DGP2 ($\rho = 0.50$)} \\ \cline{2-8}
& {\small [0, 0.5)} & {\small [0.6,\ 1)} & {\small [1, 1.5)} & {\small [1.5,
2)} & {\small [2, 2.5)} & {\small [2.5, 3)} & {\small [3, }$\infty ${\small %
) } \\ \cline{2-8}
{\small (n, p)=(100, 40)} & {\small 0.0020} & {\small 0.1522} & {\small %
0.3296} & {\small 0.2502} & {\small 0.1322} & {\small 0.0596} & {\small %
0.0674} \\ 
{\small (n, p)=(100, 100)} & {\small 0.0084} & {\small 0.3096} & {\small %
0.3772} &  {\small 0.1650} & {\small 0.0624} & {\small 0.0254} & {\small %
0.0208} \\ 
{\small (n, p)=(100, 400)} & {\small 0.0394} & {\small 0.5254} & {\small %
0.2252} & {\small 0.0616} & {\small 0.0210} & {\small 0.0090} & {\small %
0.0252} \\ 
{\small (n, p)=(400, 40)} & {\small 0.0002} & {\small 0.1170} & {\small %
0.4452} & {\small 0.2860} & {\small 0.1006} & {\small 0.0302} & {\small %
0.0204} \\ 
{\small (n, p)=(400, 100)} & {\small 0.0000} & {\small 0.2676} & {\small %
0.5656} &  {\small 0.1422} & {\small 0.0198} & {\small 0.0022} & {\small %
0.0014} \\ 
{\small (n, p)=(400, 400)} & {\small 0.0000} & {\small 0.5908} & {\small %
0.3894} & {\small 0.0156} & {\small 0.0008} &  {\small 0.0002} & {\small %
0.0000} \\ \hline
& \multicolumn{7}{c}{DGP1 ($\rho = 0.75$)} \\ \cline{2-8}
& {\small [0, 0.5)} & {\small [0.6,\ 1)} & {\small [1, 1.5)} & {\small [1.5,
2)} & {\small [2, 2.5)} & {\small [2.5, 3)} & {\small [3, }$\infty ${\small %
) } \\ \cline{2-8}
{\small (n, p)=(100, 40)} & {\small 0.0000} & {\small 0.0224} & {\small %
0.1220} & {\small 0.2250} & {\small 0.2012} & {\small 0.1488} & {\small %
0.2796} \\ 
{\small (n, p)=(100, 100)} & {\small 0.0008} & {\small 0.1144} & {\small %
0.2546} &  {\small 0.2306} & {\small 0.1698} & {\small0.0944} & {\small %
0.1312} \\ 
{\small (n, p)=(100, 400)} & {\small 0.0316} & {\small 0.4068} & {\small %
0.3408} &  {\small 0.1072} & {\small 0.0346} & {\small 0.0164} & {\small %
0.0284} \\ 
{\small (n, p)=(400, 40)} & {\small 0.0000} & {\small 0.0098} & {\small %
0.1384} & {\small 0.2800} & {\small 0.2620} & {\small 0.1526} & {\small %
0.1572} \\ 
{\small (n, p)=(400, 100)} & {\small 0.0000} & {\small 0.0144} & {\small %
0.2918} &  {\small 0.4250} & {\small 0.1868} & {\small 0.0592} & {\small %
0.0228} \\ 
{\small (n, p)=(400, 400)} & {\small 0.0000} & {\small 0.0684} & {\small %
0.6724} &  {\small 0.2304} & {\small 0.0242} & {\small 0.0040} & {\small %
0.0006} \\ \hline
& \multicolumn{7}{c}{DGP2 ($\rho = 0.75$)} \\ \cline{2-8}
& {\small [0, 0.5)} & {\small [0.6,\ 1)} & {\small [1, 1.5)} & {\small [1.5,
2)} & {\small [2, 2.5)} & {\small [2.5, 3)} & {\small [3, }$\infty ${\small %
) } \\ \cline{2-8}
{\small (n, p)=(100, 40)} & {\small 0.0062} & {\small 0.1090} & {\small %
0.2424} & {\small 0.2142} & {\small 0.1508} & {\small 0.1040} & {\small %
0.1674} \\ 
{\small (n, p)=(100, 100)} & {\small 0.0686} & {\small 0.2298} & {\small %
0.3256} &  {\small 0.1842} & {\small 0.0798} & {\small 0.0382} & {\small %
0.0518} \\ 
{\small (n, p)=(100, 400)} & {\small 0.3616} & {\small 0.3000} & {\small %
0.1594} & {\small 0.0508} & {\small 0.0186} & {\small 0.0080} & {\small %
0.0118} \\ 
{\small (n, p)=(400, 40)} & {\small 0.0000} & {\small 0.0102} & {\small %
0.1306} & {\small 0.2918} & {\small 0.2750} & {\small 0.1482} & {\small %
0.1442} \\ 
{\small (n, p)=(400, 100)} & {\small 0.0000} & {\small 0.0292} & {\small %
0.2984} &  {\small 0.4072} & {\small 0.1864} & {\small 0.0560} & {\small %
0.0226} \\ 
{\small (n, p)=(400, 400)} & {\small 0.0004} & {\small 0.3798} & {\small %
0.4626} & {\small 0.1344} & {\small 0.0134} & {\small 0.0016} & {\small %
0.0002} \\ \hline \hline
\end{tabular}%
\end{center}
\end{table}

\end{document}